\documentclass[a4paper]{amsart}

\usepackage{comment}
\usepackage{todonotes}

\usepackage{standalone}
\usepackage{graphicx}
\usepackage{subcaption}
\usepackage{tikz}
\usetikzlibrary{calc, arrows.meta, decorations.pathreplacing, patterns}
\usepackage{pgfplots}
\usepackage{tikz-3dplot}
\DeclareUnicodeCharacter{2212}{−}
\usepgfplotslibrary{groupplots,dateplot}
\usepackage{xparse}
\usepackage{bm}
\usepackage{mathtools}
\usepackage{amsmath}
\usepackage{amsmath}
\usepackage{amssymb}

\usepackage{appendix}
\usepackage{hyperref}
\usepackage[msc-links]{amsrefs}

\usepackage{color}

\usepackage{mhchem}

\hypersetup{
	colorlinks,
	linkcolor={red!50!black},
	citecolor={blue!50!black},
	urlcolor={blue!80!black}
}

\usepackage[noabbrev,capitalise]{cleveref}

\usepackage{enumitem}
\setlist[enumerate]{itemsep=0.5pt,topsep=1pt, partopsep=0pt}
\setlist[enumerate,1]{label={\textup{(\arabic*)}}}
\creflabelformat{enumi}{#2\textit{#1}#3} %

\crefdefaultlabelformat{\normalfont #2#1#3}
\crefformat{equation}{\normalfont \textup{(#2#1#3)}}
\crefrangeformat{equation}{\normalfont (#3#1#4) to~(#5#2#6)}
\crefmultiformat{equation}{\normalfont (#2#1#3)}{\normalfont\ and~(#2#1#3)}{\normalfont , (#1#1#3)}{\normalfont\ and~(#2#1#3)}

\ifpdf
\hypersetup{
  pdftitle={Error analysis of DGTD for linear Maxwell equations with inhomogeneous interface conditions},
  pdfauthor={B. Doerich, J. Doerner, M. Hochbruck}
}
\fi

\ExplSyntaxOn

\NewDocumentCommand{\definealphabet}{mmmm}
 {%
  \int_step_inline:nnn { `#3 } { `#4 }
   {
    \cs_new_protected:cpx { #1 \char_generate:nn { ##1 }{ 11 } }
     {
      \exp_not:N #2 { \char_generate:nn { ##1 } { 11 } }
     }
   }
 }

\ExplSyntaxOff

\definealphabet{bb}{\mathbb}{A}{Z}

\definealphabet{cal}{\mathcal}{A}{Z}

\definealphabet{scr}{\mathscr}{A}{Z}

\definealphabet{frak}{\mathfrak}{A}{Z}

\definealphabet{vv}{\bm}{A}{Z}
\definealphabet{vv}{\bm}{a}{z}

\newcommand{\vvdelta}{\bm{\delta}}

\newcommand{\vvnu}{\bm{\nu}}

\newcommand{\vvphi}{\bm{\phi}}

\newcommand{\vvpsi}{\bm{\psi}}

\newcommand{\macroTildeBold}[1]{\tilde{\bm{#1}}}
\definealphabet{vt}{\macroTildeBold}{A}{Z}
\definealphabet{vt}{\macroTildeBold}{a}{z}

\newcommand{\macroHatBold}[1]{\widehat{\bm{#1}}}
\definealphabet{vh}{\macroHatBold}{A}{Z}
\definealphabet{vh}{\macroHatBold}{a}{z}

\newcommand{\macroCheckBold}[1]{\check{\bm{#1}}}
\definealphabet{vc}{\macroCheckBold}{A}{Z}
\definealphabet{vc}{\macroCheckBold}{a}{z}

\NewDocumentCommand{\ScalarProductFactory}{m m o}{
        \IfValueTF{#3}{
                \NewDocumentCommand{#1}{s O{} m m o}{
                        \IfBooleanTF{##1}{%
                                \IfValueTF{##5}{        %
                                        #2*{##3,##4}_{#3,##5}
                                }
                                {                       %
                                        #2*{##3,##4}_{#3}
                                }
                        }
                        {       %
                                \IfValueTF{##5}{        %
                                        #2[{##2}]{##3,##4}_{#3,##5}
                                }
                                {                       %
                                        #2[{##2}]{##3,##4}_{#3}
                                }
                        }
                }
        }
        {
                \NewDocumentCommand{#1}{s O{} m m}{
                        \IfBooleanTF{##1}{%
                                #2*{##3,##4}
                        }
                        {       %
                                #2[{##2}]{##3,##4}
                        }
                }
        }
}

\DeclarePairedDelimiter{\schDelim}{\langle}{\rangle}
\ScalarProductFactory{\sch}{\schDelim}

\DeclarePairedDelimiter{\sprodDelim}{(}{)}
\ScalarProductFactory{\sprod}{\sprodDelim}

\DeclarePairedDelimiter{\brackDelim}{[}{]}
\ScalarProductFactory{\brackprod}{\brackDelim}

\ScalarProductFactory{\sprodmu}{\sprodDelim}[\mu]
\ScalarProductFactory{\sprodeps}{\sprodDelim}[\varepsilon]
\ScalarProductFactory{\sprodmueps}{\sprodDelim}[\mu\times\varepsilon]
\ScalarProductFactory{\sprodepsmu}{\sprodDelim}[\varepsilon\times\mu]

\DeclarePairedDelimiter{\singlenorm}{|}{|}

\DeclarePairedDelimiter{\norm}{\|}{\|}

\DeclareFontFamily{U}{matha}{\hyphenchar\font45}
\DeclareFontShape{U}{matha}{m}{n}{
<-6> matha5 <6-7> matha6 <7-8> matha7
<8-9> matha8 <9-10> matha9
<10-12> matha10 <12-> matha12
}{}
\DeclareSymbolFont{matha}{U}{matha}{m}{n}

\DeclareFontFamily{U}{mathx}{\hyphenchar\font45}
\DeclareFontShape{U}{mathx}{m}{n}{
<-6> mathx5 <6-7> mathx6 <7-8> mathx7
<8-9> mathx8 <9-10> mathx9
<10-12> mathx10 <12-> mathx12
}{}
\DeclareSymbolFont{mathx}{U}{mathx}{m}{n}

\DeclareMathDelimiter{\vvvert} {0}{matha}{"7E}{mathx}{"17}%

\NewDocumentCommand{\NormFactory}{m m m}{
        \NewDocumentCommand{#1}{s O{} m o}{
                \IfBooleanTF{##1}{%
                        \IfValueTF{##4}{        %
                                #2*{##3}_{#3,##4}
                        }
                        {                       %
                                #2*{##3}_{#3}
                        }
                }
                {       %
                        \IfValueTF{##4}{        %
                                #2[{##2}]{##3}_{#3,##4}
                        }
                        {                       %
                                #2[{##2}]{##3}_{#3}%
                        }%
                }%
        }%
}%
\NormFactory{\normmu}{\norm}{\mu}
\NormFactory{\normeps}{\norm}{\varepsilon}
\NormFactory{\normmueps}{\norm}{\mu\times\varepsilon}
\NormFactory{\normepsmu}{\norm}{\varepsilon\times\mu}

\DeclarePairedDelimiter{\funbracks}{(}{)}
\NewDocumentCommand{\confuns}{s O{} m m}{
        \IfBooleanTF{#1}
        {
                C^{{#3}}\funbracks*{{#4}}
        }
        {
                C^{{#3}}\funbracks[{#2}]{{#4}}
        }
}

\NewDocumentCommand{\lebfuns}{s O{} m m}{
        \IfBooleanTF{#1}
        {
                L^{{#3}}\funbracks*{{#4}}
        }
        {
                L^{{#3}}\funbracks[{#2}]{{#4}}
        }
}

\NewDocumentCommand{\sobfuns}{s O{} m m}{
        \IfBooleanTF{#1}
        {
                H^{{#3}}\funbracks*{{#4}}
        }
        {
                H^{{#3}}\funbracks[{#2}]{{#4}}
        }
}

\NewDocumentCommand{\hsobfuns}{s O{} m m}{
        \IfBooleanTF{#1}
        {
                H_0^{{#3}}\funbracks*{{#4}}
        }
        {
                H_0^{{#3}}\funbracks[{#2}]{{#4}}
        }
}

\NewDocumentCommand{\hsobfunszeroo}{s O{} m m}{
        \IfBooleanTF{#1}
        {
                H_0^{{#3}}\funbracks*{{#4}}
        }
        {
                H_{00}^{{#3}}\funbracks[{#2}]{{#4}}
        }
}

\NewDocumentCommand{\curlfuns}{s O{} m}{
    \IfBooleanTF{#1}
    {
        H\funbracks*{\operatorname{curl}, {#3}}
    }
    {
        H\funbracks[{#2}]{\operatorname{curl}, {#3}}
    }
}

\NewDocumentCommand{\hcurlfuns}{s O{} m}{
    \IfBooleanTF{#1}
    {
        H_0\funbracks*{\operatorname{curl}, {#3}}
    }
    {
        H_0\funbracks[{#2}]{\operatorname{curl}, {#3}}
    }
}

\NewDocumentCommand{\divfuns}{s O{} m O{}}{
    \IfBooleanTF{#1}
    {
        H\funbracks*{\operatorname{div}#4, {#3}}
    }
    {
        H\funbracks[{#2}]{\operatorname{div}#4, {#3}}
    }
}

\NewDocumentCommand{\hdivfuns}{s O{} m O{}}{
    \IfBooleanTF{#1}
    {
        H_0\funbracks*{\operatorname{div}#4, {#3}}
    }
    {
        H_0\funbracks[{#2}]{\operatorname{div}#4, {#3}}
    }
}

\NewDocumentCommand{\testfuns}{s O{} m}{
    \IfBooleanTF{#1}
    {
        C^\infty_0\funbracks*{{#3}}
    }
    {
        C^\infty_0\funbracks[{#2}]{{#3}}
    }
}

\NewDocumentCommand{\restricttestfuns}{s O{} m}{
    \IfBooleanTF{#1}
    {
        C^\infty\funbracks*{\overline{{#3}}}
    }
    {
        C^\infty\funbracks[{#2}]{\overline{{#3}}}
    }
}

\NewDocumentCommand{\PolyAtMost}{m m o}{
        \IfValueTF{#3}{
                \bbP_{#1}^{#2}(#3)
        }{
                \bbP_{#1}^{#2}
        }
}

\newcommand{\dt}{\partial_t}

\DeclareMathOperator{\curl}{curl}

\let\div\relax\DeclareMathOperator{\div}{div}

\DeclarePairedDelimiterX\setc[2]{\{}{\}}{\,#1 \;\delimsize\vert\; #2\,}
\DeclarePairedDelimiter\monosetc{\{}{\}}

\NewDocumentCommand{\restrict}{s O{} m m}{
    \IfBooleanTF{#1}{
        \left.{#3}\right|_{#4}
    }{
        {#3}#2|_{#4}
    }
}

\newcommand{\avg}[1]{\{\hspace{-3.8pt}\{#1\}\hspace{-3.8pt}\}}
\newcommand{\jmp}[1]{[\![#1]\!]}

\newcommand{\defpnt}{{\hspace*{0.15em}:\hspace*{0.15em}}}

\newcommand{\dintegral}[1]{\ \mathrm{d}#1}
\newcommand{\dx}{\dintegral{x}}
\newcommand{\ds}{\dintegral{s}}

\newcommand{\domain}{Q}
\newcommand{\domainleft}{\domain_-}
\newcommand{\domainright}{\domain_+}
\newcommand{\domainlr}{\domain_\pm}
\newcommand{\Fint}{{F_{\operatorname{int}}}}
\newcommand{\nint}{\bm{n}_{\operatorname{int}}}

\newcommand{\Jvol}[1][]{\vvJ^{#1}}
\newcommand{\Jvolpm}[1][]{\Jvol[#1]_\pm}
\newcommand{\jvol}[1][]{\vvj^{#1}}
\newcommand{\rhovol}[1][]{\rho^{#1}}
\newcommand{\rhovolpm}[1][]{\rhovol[#1]_\pm}
\newcommand{\Jsurf}[1][]{\vvJ^{#1}_{\operatorname{surf}}}
\newcommand{\JsurfTwo}[1][]{J^{#1}_{\operatorname{surf},2}}
\newcommand{\JsurfThree}[1][]{J^{#1}_{\operatorname{surf},3}}
\newcommand{\rhosurf}[1][]{\rho^{#1}_{\operatorname{surf}}}

\newcommand{\Jvolh}[1][]{\vvJ^{#1}_{h}}
\newcommand{\jvolh}[1][]{\vvj^{#1}_{h}}

\newcommand{\Jsurfh}[1][]{\vvJ^{#1}_{\operatorname{surf},h}}
\newcommand{\jsurfh}[1][]{\vvj^{#1}_{\operatorname{surf},h}}
\newcommand{\JsurfhInterp}[1][]{\vcJ^{#1}_{\operatorname{surf},h}}
\newcommand{\jsurfhInterp}[1][]{\vcj^{#1}_{\operatorname{surf},h}}

\newcommand{\Eperm}[1][]{\varepsilon^{#1}}
\newcommand{\Epermpm}[1][]{\Eperm[#1]_\pm}
\newcommand{\Msus}[1][]{\mu^{#1}}
\newcommand{\Msuspm}[1][]{\Msus[#1]_\pm}
\newcommand{\sol}{c}
\newcommand{\solpm}{\sol_\pm}

\newcommand{\Msusmax}{\Msus_\infty}

\newcommand{\solmax}{\sol_\infty}

\newcommand{\Ef}[1][]{\vvE^{#1}}
\newcommand{\Hf}[1][]{\vvH^{#1}}
\newcommand{\Efpm}[1][]{\Ef[#1]_\pm}
\newcommand{\Hfpm}[1][]{\Hf[#1]_\pm}
\newcommand{\Efh}[1][]{\Ef[#1]_h}
\newcommand{\Hfh}[1][]{\Hf[#1]_h}
\newcommand{\TeFunH}[1][]{\vvphi^{#1}_h}
\newcommand{\TeFunE}[1][]{\vvpsi^{#1}_h}
\newcommand{\uf}[1][]{\vvu^{#1}}

\newcommand{\ufh}[1][]{\vvu^{#1}_h}

\newcommand{\trapedefect}[1][]{\vvdelta^{#1}}

\newcommand{\pcurlfuns}{PH(\curl, \domain)}
\newcommand{\psobfuns}[1]{PH^{#1}(\domain)}

\newcommand{\INTMOpChar}{\bm{\calC}}
\newcommand{\pMOpH}{\operatorname{\widehat{\INTMOpChar}_H}}
\newcommand{\MOpH}{\operatorname{\INTMOpChar_H}}
\newcommand{\MOpE}{\operatorname{\INTMOpChar_E}}
\newcommand{\pMOp}{\operatorname{\widehat{\INTMOpChar}}}
\newcommand{\MOp}{\operatorname{\INTMOpChar}}
\newcommand{\DompMOpH}{D(\pMOpH)}
\newcommand{\DomMOpH}{D(\MOpH)}
\newcommand{\DomMOpE}{D(\MOpE)}
\newcommand{\DompMOp}{D(\pMOp)}
\newcommand{\DomMOp}{D(\MOp)}

\newcommand{\INTDGMOpChar}{\bm{\frakC}}
\newcommand{\pDgMOpH}{\operatorname{\widehat{\INTDGMOpChar}_{H}}}
\newcommand{\DgMOpH}{\operatorname{\INTDGMOpChar_H}}
\newcommand{\DgMOpE}{\operatorname{\INTDGMOpChar_E}}
\newcommand{\pDgMOp}{\operatorname{\widehat{\INTDGMOpChar}}}
\newcommand{\DgMOp}{\operatorname{\INTDGMOpChar}}
\newcommand{\DisMOp}{\pDgMOp}
\newcommand{\DgLFPert}{\operatorname{\widehat{\bm{\frakD}}}}

\newcommand{\INTLiftOpChar}{\bm{\frakL}}
\newcommand{\LiftOp}{{\INTLiftOpChar}_{\operatorname{int}}}

\newcommand{\Hparallelspace}{\bm{H}^{1/2}_{\parallel,00}(\Fint)}
\newcommand{\HparallelspaceDual}{\bm{H}^{-1/2}_{\parallel,00}(\Fint)}

\newcommand{\TraceSpaceParallel}{\bm{H}^{-1/2}_{\parallel,00}(\div_\Fint,\Fint)}

\newcommand{\pVastH}{\widehat{V}_\ast^{\vvH}}
\newcommand{\VastH}{V^{\vvH}_\ast}
\newcommand{\VastE}{V^{\vvE}_\ast}
\newcommand{\pVast}{\widehat{V}_\ast}
\newcommand{\Vast}{V_\ast}

\newcommand{\mesh}{\calT}
\newcommand{\meshh}{\mesh_h}

\newcommand{\JextH}[1][]{\vvJ^{#1}_{\vvH}}
\newcommand{\jextH}[1][]{\vvj^{#1}_{\vvH}}

\newcommand{\Ffaces}{\calF_h}
\newcommand{\Finterior}{\calF_h^\circ}
\newcommand{\Fboundary}{\calF_h^{\partial}}
\newcommand{\Finterface}{\calF_h^{\operatorname{int}}}

\newcommand{\element}{K}
\newcommand{\nvK}{\vvn_{\element}}
\newcommand{\face}{F}
\newcommand{\nvF}{\vvn_{\face}}
\newcommand{\elementKFl}{K_{F,l}}
\newcommand{\elementKFr}{K_{F,r}}

\newcommand{\Vapprox}{V_h}

\newcommand{\VasthE}{V_{\ast,h}^{\vvE}}
\newcommand{\VasthH}{V_{\ast,h}^{\vvH}}
\newcommand{\Vasth}{V_{\ast,h}}
\newcommand{\pVasthH}{\widehat{V}_{\ast,h}^{\vvH}}
\newcommand{\pVasth}{\widehat{V}_{\ast,h}}

\newcommand{\OProj}{\Pi_h}

\newcommand{\NInterp}{\calI^h}
\newcommand{\FInterp}{\frakI^h}
\newcommand{\CFInterp}{C}

\newcommand{\regcoeff}{r_\ast}

\newcommand{\INTROpChar}{\bm{\frakR}}

\newcommand{\Rhlr}{\widehat{\INTROpChar}_{\pm}}
\newcommand{\Rhl}{\widehat{\INTROpChar}_{-}}
\newcommand{\Rhr}{\widehat{\INTROpChar}_{+}}
\newcommand{\Rh}{\widehat{\INTROpChar}}

\newcommand{\Cstab}{C_{\operatorname{stb}}}

\newcommand{\CFL}{\operatorname{CFL}}
\newcommand{\tauCFL}{\tau_{\CFL}}
\newcommand{\thetaCFL}{\theta}

\newcommand{\wellposedness}{well-posedness}
\newcommand{\Wellposedness}{Well-posedness}
\newcommand{\wellposed}{well-posed}

\newcommand{\welldefined}{well-defined}

\newcommand{\cf}{cf.}
\newcommand{\eg}{e.g.}
\newcommand{\ie}{i.e.}
\newcommand{\scsemigroup}{\(C^0\)-semigroup}
\newcommand{\meshsize}{mesh size}

\newcommand{\piecewise}{piecewise}

\newcommand{\componentwise}{component-wise}

\newcommand{\semidiscrete}{semi-discrete}

\newcommand{\semidiscretization}{semi-discretization}

\newcommand{\timestep}{time step}
\newcommand{\stepsize}{step size}
\newcommand{\submesh}{sub-mesh}
\newcommand{\seminorm}{semi-norm}
\newcommand{\quasilinear}{quasilinear}
\newcommand{\Graphene}{Graphene}

\newcommand{\statespace}{state-space}

\newcommand{\dG}{dG}
\newcommand{\leapfrog}{leapfrog}

\newcommand{\CrankNicolson}{Crank--Nicolson}
\newcommand{\discontinuous}{discontinuous}

\newcommand{\inhomogeneous}{inhomogeneous}

\newcommand{\simplicial}{simplicial}

\newtheorem{theorem}{Theorem}[section]
\newtheorem{lemma}[theorem]{Lemma}
\newtheorem{corollary}[theorem]{Corollary}

\newtheorem{assumption}[theorem]{Assumption}
\crefname{assumption}{assumption}{assumptions}
\Crefname{Assumption}{Assumption}{Assumptions}

\theoremstyle{definition}

\theoremstyle{remark}
\newtheorem{remark}[theorem]{Remark}

\numberwithin{equation}{section}

\begin{document}

\title[DGTD for linear Maxwell equations with interface current]{%
  Error analysis of DGTD for linear Maxwell equations with inhomogeneous interface conditions
}

\author[Benjamin Dörich]{Benjamin Dörich}
\address{Institute for Applied and Numerical Mathematics,
    Karlsruhe Institute of Technology, Englerstr. 2, 76131 Karls\-ru\-he, Germany}
\email{\{benjamin.doerich,julian.doerner,marlis.hochbruck\}@kit.edu}
\thanks{Funded by the Deutsche Forschungsgemeinschaft (DFG, German Research Foundation) -- Project-ID 258734477 -- SFB 1173}

\author[Julian Dörner]{Julian Dörner}

\author[Marlis Hochbruck]{Marlis Hochbruck}

\subjclass[2020]{Primary
65M12, 65M15, 65M60
Secondary 35Q60}

\date{\today}

\begin{abstract}
    In the present paper we consider linear and isotropic Maxwell equations with inhomogeneous interface conditions.
    We discretize the problem with the discontinuous Galerkin method in space and with the \leapfrog{} scheme in time.
    An analytical setting is provided in which we show \wellposedness{} of the problem, derive stability estimates, and exploit this in the error analysis to prove rigorous error bounds for both the spatial and full discretization.
    The theoretical findings are confirmed with numerical experiments.
\end{abstract}

\maketitle

\section{Introduction}\label{sec:introduction}

\Graphene{} is a monolayer of carbon atoms arranged in a hexagonal lattice and has demonstrated exceptional thermal, electrical and optical properties, as well as structural robustness. 
This makes it an attractive candidate for reinforcement in materials, as well as for applications in organic electronics and optoelectronics. 
There is significant interest in both the scientific community and industrial sectors in graphene.
In recent years, numerous other 2D materials have emerged, with the term referring to crystalline solids that have a reduced thickness, usually consisting of a single or only a few atomic layers. 
This unique characteristic gives them exceptional properties.
Another important class of composite 2D materials are the
semiconducting Transition Metal Dichalcogenides (TMDCs) such as \ce{MoS2} or \ce{WS2} which also come in single layers, albeit with a more complicated unit cell than graphene.
These materials have a wide range of applications, including catalysis, spintronics, and optoelectronics; see the reviews \cite{Cho.L.S.Et.2010,Ted.L.O.2016}.
Numerical simulations are vital for studying such materials.

The optical properties of such materials can be studied by depositing a sheet on a thin dielectric layer on top of a metal plate and exciting it with light pulses. 
The interaction between the pulses and the material is described by Maxwell equations coupled to quantum mechanical models, see, \eg{}, \cite{Cha.E.Z.Et.2008,Kan.M.2005,Sta.2014}. 
A simpler way of modeling the interaction of the 2D material in the Maxwell equations is to use conductivity surfaces or current sheets.
Here it is assumed that the material sheet has zero thickness, thus is truly two-dimensional, and the constitutive equation
\begin{equation*}
  \Jsurf(\omega) = \sigma_{\operatorname{surf}}(\omega) \Ef(\omega)
\end{equation*}
holds along the plane of the material in the frequency domain,
where \(\sigma_{\operatorname{surf}}\) describes the surface conductivity of the 2D material.
We refer to Chapter~1 in \cite{Dep.2016.GOES} for details about the modelling.

\begin{figure}[htbp]
    \centering
    \begin{subfigure}[b]{0.45\textwidth}
        \centering
        \scalebox{1}{\begin{tikzpicture}%

\definecolor{color01}{HTML}{1f77b4}
\definecolor{color02}{HTML}{ff7f0e}
\definecolor{color03}{HTML}{2ca02c}
\definecolor{color04}{HTML}{d62728}
\definecolor{color05}{HTML}{9467bd}
\definecolor{color06}{HTML}{8c564b}
\definecolor{color07}{HTML}{e377c2}
\definecolor{color08}{HTML}{7f7f7f}
\definecolor{color09}{HTML}{bcbd22}
\definecolor{color10}{HTML}{17becf}

\coordinate (zeroo) at (0,0,2);
\coordinate (zero) at (0,0,0);
\coordinate (x1) at (2.5,0,0);
\coordinate (x3) at (0,2.5,0);
\coordinate (x2) at (0,0,2.5);

\draw[black, -{Latex}] (zeroo) -- ($(zeroo) + (1,0,0)$) node[black, scale=0.5, anchor = north west] {$x_1$};
\draw[black, -{Latex}] (zeroo) -- ($(zeroo) + (0,0,-1)$) node[black, scale=0.5, anchor = north west] {$x_2$};
\draw[black, -{Latex}] (zeroo) -- ($(zeroo) + (0,1,0)$) node[black, scale=0.5, anchor = south east] {$x_3$};

\coordinate (b1) at (2,0,0);
\coordinate (b2) at (-2,0,0);
\coordinate (b3) at (-2,0,2);
\coordinate (b4) at (2,0,2);

\coordinate (t1) at ($(b1) + (0,2,0)$);
\coordinate (t2) at ($(b2) + (0,2,0)$);
\coordinate (t3) at ($(b3) + (0,2,0)$);
\coordinate (t4) at ($(b4) + (0,2,0)$);

\coordinate (i1) at (zero);
\coordinate (i2) at (0,0,2);
\coordinate (i3) at ($(i2) + (0,2,0)$);
\coordinate (i4) at ($(i1) + (0,2,0)$);

\coordinate (im) at ($0.5*(i3)$);

\draw[color01, dashed] (b1) -- (b2) -- (b3);
\draw[color01] (b3) -- (b4) -- (b1);
\draw[color01] (t1) -- (t2) -- (t3) -- (t4) -- cycle;
\draw[color01] (b1) -- (t1);
\draw[color01, dashed] (b2) -- (t2);
\draw[color01] (b3) -- (t3);
\draw[color01] (b4) -- (t4);

\draw[black, -{Latex}] (im) -- ($(im) + (0.75,0,0)$) 
node[black, scale = 0.65, anchor = north west] {$\nint$};

\node[black, anchor = north west, scale=0.75] at (t3) {$\domainleft$};
\node[black, anchor = south east, scale=0.75] at (b4) {$\domainright$};

\draw[color02, dashed] (i4) -- (i1) -- (i2);
\draw[color02] (i2) -- (i3) -- (i4);
\draw[pattern=north west lines, pattern color=color02, opacity=0.17] (i1) -- (i2) -- (i3) -- (i4) -- cycle;
\node[black, anchor = north east, xshift=0.4cm, yshift=-0.25cm, scale=0.75] at (im) {$\Fint$};

\end{tikzpicture}}
        \caption{First domain}
        \label{fig:NormalDomain}
    \end{subfigure}
    \hfill
    \begin{subfigure}[b]{0.45\textwidth}
        \centering
        \scalebox{1}{\begin{tikzpicture}%

\definecolor{color01}{HTML}{1f77b4}
\definecolor{color02}{HTML}{ff7f0e}
\definecolor{color03}{HTML}{2ca02c}
\definecolor{color04}{HTML}{d62728}
\definecolor{color05}{HTML}{9467bd}
\definecolor{color06}{HTML}{8c564b}
\definecolor{color07}{HTML}{e377c2}
\definecolor{color08}{HTML}{7f7f7f}
\definecolor{color09}{HTML}{bcbd22}
\definecolor{color10}{HTML}{17becf}

\coordinate (zeroo) at (-1,0,2);
\coordinate (zero) at (0,0,0);
\coordinate (x2) at (1,0,0);
\coordinate (x3) at (0,1,0);
\coordinate (x1) at (0,0,1);

\coordinate (b1) at (zero);
\coordinate (b2) at ($(zero) + 2*(x1)$) ;
\coordinate (b3) at ($(zero) + 2*(x1) + 2*(x2)$);
\coordinate (b4) at ($(zero) + 4*(x1) + 2*(x2)$);
\coordinate (b5) at ($(zero) + 4*(x1) + 4*(x2)$);
\coordinate (b6) at ($(zero) + 0*(x1) + 4*(x2)$);

\draw[color01, dashed] (b1) -- (b2);
\draw[color01] (b2) -- ($(zero) + 2*(x1) + 1.25*(x2)$);
\draw[color01, dashed] ($(zero) + 2*(x1) + 1.25*(x2)$) -- (b3);
\draw[color01, dashed] (b3) -- (b4);
\draw[color01] (b4) -- (b5);
\draw[color01] (b5) -- (b6);
\draw[color01, dashed] (b6) -- (b1);

\coordinate (t1) at ($(b1) + (0,2,0)$);
\coordinate (t2) at ($(b2) + (0,2,0)$);
\coordinate (t3) at ($(b3) + (0,2,0)$);
\coordinate (t4) at ($(b4) + (0,2,0)$);
\coordinate (t5) at ($(b5) + (0,2,0)$);
\coordinate (t6) at ($(b6) + (0,2,0)$);

\draw[color01] (t1) -- (t2) -- (t3) -- (t4) -- (t5) -- (t6) -- cycle;

\draw[color01, dashed] (b1) -- (t1);
\draw[color01] (b2) -- (t2);
\draw[color01, dashed] (b3) -- (t3);
\draw[color01] (b4) -- (t4);
\draw[color01] (b5) -- (t5);
\draw[color01] (b6) -- (t6);

\coordinate (i1) at (b3);
\coordinate (i2) at (b6);
\coordinate (i3) at (t6);
\coordinate (i4) at (t3);

\coordinate (dp1) at ($0.707106*(x1) + 0.707106*(x2)$);
\coordinate (dp2) at ($-0.707106*(x1) + 0.707106*(x2)$);
\coordinate (dp3) at (x3);

\coordinate (im) at ($3*(x2) + (x1) + (x3)$);

\draw[pattern=north west lines, pattern color=color02, opacity=0.17] (i1) -- (i2) -- (i3) -- (i4) -- cycle;
\draw[color02, dashed] (i4) -- (i1) -- (i2);
\draw[color02] (i2) -- (i3) -- (i4);
\draw[black, -{Latex}] ($(im) - 0.55*(dp2)  +0.45*(dp3)$) -- ($(im) +  0.707106*(1,0,1) - 0.55*(dp2) +0.45*(dp3)$) 
 node[black, scale = 0.65, anchor = north west] {$\nint$};
\node[black, anchor = north east, xshift=-0.4cm, yshift=+0.25cm, scale=0.75] at (im) {$\Fint$};

\draw[black, -{Latex}] (i1) -- ($(i1) + (dp1)$) node[black, scale=0.5, anchor = west] {$x_1$};
\draw[black, -{Latex}] (i1) -- ($(i1) + (dp2)$) node[black, scale=0.5, anchor = north west] {$x_2$};
\draw[black, -{Latex}] (i1) -- ($(i1) + (dp3)$) node[black, scale=0.5, anchor = south east] {$x_3$};

\node[black, anchor = north west, scale=0.75] at (t1) {$\domainleft$};
\node[black, anchor = south east, scale=0.75] at (b5) {$\domainright$};

\end{tikzpicture}}
        \caption{Second domain}
        \label{fig:LDomain}
    \end{subfigure}

    \vspace{1em} %

    \begin{subfigure}[b]{0.6\textwidth}
        \centering
        \scalebox{1}{\tdplotsetmaincoords{45+60}{-45+90}  %

\begin{tikzpicture}[tdplot_main_coords, scale=2]

\definecolor{color01}{HTML}{1f77b4}
\definecolor{color02}{HTML}{ff7f0e}
\definecolor{color03}{HTML}{2ca02c}
\definecolor{color04}{HTML}{d62728}
\definecolor{color05}{HTML}{9467bd}
\definecolor{color06}{HTML}{8c564b}
\definecolor{color07}{HTML}{e377c2}
\definecolor{color08}{HTML}{7f7f7f}
\definecolor{color09}{HTML}{bcbd22}
\definecolor{color10}{HTML}{17becf}

\def\radius{1}    %
\def\length{2}    %

\foreach \i in {0,1,...,5} {
    \pgfmathsetmacro\angleA{60*\i}
    \pgfmathsetmacro\angleB{60*(\i+1)}

    \coordinate (b\i) at ({0*\length}, {\radius*cos(\angleA)}, {\radius*sin(\angleA)});
    \coordinate (t\i) at ({1*\length}, {\radius*cos(\angleA)}, {\radius*sin(\angleA)});
    \coordinate (i\i) at ({0.5*\length}, {\radius*cos(\angleA)}, {\radius*sin(\angleA)});
    
    \draw[color01] 
        (0, {\radius*cos(\angleA)}, {\radius*sin(\angleA)}) -- 
        (0, {\radius*cos(\angleB)}, {\radius*sin(\angleB)});
    
    \ifthenelse{\i=2 \OR \i=3 \OR \i=4}{
      \draw[orange, dashed] 
        ({0.5*\length}, {\radius*cos(\angleA)}, {\radius*sin(\angleA)}) -- 
        ({0.5*\length}, {\radius*cos(\angleB)}, {\radius*sin(\angleB)});
    }{
      \draw[orange] 
        ({0.5*\length}, {\radius*cos(\angleA)}, {\radius*sin(\angleA)}) -- 
        ({0.5*\length}, {\radius*cos(\angleB)}, {\radius*sin(\angleB)});
    }    
    
    \ifthenelse{\i=2 \OR \i=3 \OR \i=4}{
      \draw[color01, dashed] 
        (\length, {\radius*cos(\angleA)}, {\radius*sin(\angleA)}) -- 
        (\length, {\radius*cos(\angleB)}, {\radius*sin(\angleB)});
    }{
      \draw[color01] 
        (\length, {\radius*cos(\angleA)}, {\radius*sin(\angleA)}) -- 
        (\length, {\radius*cos(\angleB)}, {\radius*sin(\angleB)});
    }

    \ifthenelse{\i=3 \OR \i=4}{
      \draw[color01, dashed]
        (0, {\radius*cos(\angleA)}, {\radius*sin(\angleA)}) -- 
        (\length, {\radius*cos(\angleA)}, {\radius*sin(\angleA)});
    }{
      \draw[color01]
        (0, {\radius*cos(\angleA)}, {\radius*sin(\angleA)}) -- 
        (\length, {\radius*cos(\angleA)}, {\radius*sin(\angleA)});
    }
}

\draw[pattern=north west lines, pattern color=color02, opacity=0.17] (i0) -- (i1) -- (i2) -- (i3) -- (i4) -- (i5) -- cycle;

\draw[black, -{Latex}] (i4) -- ($(i4) + 0.5*(-1,0,0)$) node[black, scale=0.5, anchor = north west] {$x_1$};
\draw[black, -{Latex}] (i4) -- ($(i4) + 0.5*(0,1,0)$) node[black, scale=0.5, anchor = north] {$x_2$};
\draw[black, -{Latex}] (i4) -- ($(i4) + 0.5*(0,0,1)$) node[black, scale=0.5, anchor = east] {$x_3$};

\coordinate (im) at ($({0.5*\length},0,0)$);
\draw[black, -{Latex}] (im) -- ($(im) - (0.5,0,0)$)
  node[black, scale = 0.65, anchor = north west] {$\nint$};
\node[black, anchor = north east, xshift=0.75cm, yshift=+0.55cm, scale=0.75] at (im) {$\Fint$};

\node[black, anchor = south east, scale=0.75, yshift=0.15cm, xshift=-0.25cm] at (t0) {$\domainleft$};
\node[black, anchor = south west, scale=0.75, yshift=-0.5cm, xshift=0.25cm] at (b3) {$\domainright$};

\end{tikzpicture}}
        \caption{Third domain}
        \label{fig:HexagonalDomain}
    \end{subfigure}

    \caption{Sketches of different domain configurations \(\domain\).}
    \label{fig:domains}
\end{figure}

As a first step toward the full model, we consider linear and isotropic time-dependent Maxwell equations on a polyhedral domain \(\domain \subset \bbR^3\) with plane faces.
Furthermore, we assume that \(\domain\) is composed of two plane-faced sub-polyhedra \(\domainleft\) and \(\domainright\) which share a common face \(\Fint = \overline{\domainleft} \cap \overline{\domainright}\).
We choose the coordinate frame such that \(\Fint\) is a subset of the \(x_2x_3\)-plane and the normal vector \(\nint\) on \(\Fint\) is aligned with the \(x_1\)-axis.
See \cref{fig:domains} for different examples we have in mind.
We assume that the surface current \(\Jsurf[]\) supported on \(\Fint\) is a given function.
The governing equations read
\begin{subequations}\label{eq:MaxwellEquations}
    \begin{align}
        \dt\Hfpm &= -\Msuspm[-1]\curl \Efpm, & 
        \div\Bigl(\Msuspm\Hfpm\Bigr) & = 0, \label{eq:MaxwellEquations:H}\\
        \dt\Efpm 
            &= \Epermpm[-1]\curl \Hfpm 
                - \Epermpm[-1] \Jvolpm, & 
        \div\Bigl(\Epermpm\Efpm\Bigr)
        &= \rhovolpm, \label{eq:MaxwellEquations:E}
    \end{align}
\end{subequations}
on \(\domainlr\), for \(t\geq 0\).
We denote by \(f_\pm = \restrict[\big]{f}{\domainlr}\) the restriction of a function \(f\in\lebfuns{2}{\domain}\).
Here, for \(x\in\domainlr\), \(\Ef(t,x),\Hf(t,x) \in \bbR^3\) denote the electric and magnetic field, \(\Jvol(t,x)\in\bbR^3\) the volume current density and \(\rhovol(t,x)\in\bbR\) the charge density, respectively.
We assume that the material parameters \(\Msuspm,\Epermpm \geq \delta > 0\) are constant on \(\domainlr\).
The equations are equipped with perfectly conducting boundary conditions
\begin{equation}\label{eq:PECBoundary}
        \mu\Hf\cdot \vvnu = 0,\hspace*{2em}
        \Ef\times \vvnu = 0,
\end{equation}
on \(\partial\domain\), for \(t\geq 0\) with outer unit normal vector \(\vvnu\), see, \eg{}, \cite[Sec.~I.4.2.4]{Dau.L.1990.MANM}.
At the interface \(\Fint\), the conditions
\begin{subequations}{\label{eq:InterfaceCond}}
    \begin{align}
        \jmp{\mu\Hf\cdot\nint}_\Fint &= 0, &
        \jmp{\varepsilon\Ef\cdot \nint}_\Fint &= \rhosurf, \label{eq:InterfaceCond:Normal}\\
        \jmp{\Hf\times\nint}_\Fint &= \Jsurf, &
        \jmp{\Ef\times\nint}_\Fint &= 0,\label{eq:InterfaceCond:Tangential}
    \end{align}
\end{subequations}
hold for \(t\geq 0\), where \(\nint\) denotes the inner unit normal vector on \(\Fint\) pointing from \(\domainleft\) to \(\domainright\) and \(\jmp{f}_\Fint = \restrict[\big]{f_+}{\Fint} - \restrict[\big]{f_-}{\Fint}\) denotes the jump on \(\Fint\) whenever the functions \(f_\pm\) admit well-defined traces on the interface.

By \cref*{eq:InterfaceCond:Tangential}, we obtain the condition
\begin{equation*}
    \Jsurf[] \cdot \nint = \jmp{\vvH\times\nint} \cdot \nint = 0,
\end{equation*}
since the tangential trace of the magnetic field belongs to the tangential plane of \(\nint\).
With the specific coordinate frame chosen above, this reads as 
\[
    \Jsurf(t,x) = \bigl(0,\JsurfTwo[](t,x),\JsurfThree[](t,x)\bigr).
\]

Finally, by \(\rhosurf(t,x) \in \bbR\) we denote the surface charge density, see, \eg{}, for details \cite[Sec.~I.4.2.2]{Dau.L.1990.MANM}.

\begin{remark}
    It is worth noting that, in principle, the geometric framework could be extended to accommodate curved polyhedra, to permit the sharing of multiple faces at interfaces, and to allow for a disconnected collection of interfaces.
    However, such generalizations significantly complicate the presentation and thus we focus on thin, two-dimensional materials.
\end{remark}

\subsection*{Discretization}%
The discontinuous Galerkin (\dG{}) time-domain method is a well-establish method for Maxwell equations, see, \eg{}, \cite{Hes.W.2002}.
We briefly recall the construction of the \dG{} space discretization and refer to \cref{subsec:DiscreteSetting,subsec:SpatialDiscretization} for details.

Assume that \(\meshh\) is the union of suitable meshes for \(\domainlr\) with elements \(\element\) and matching faces at \(\Fint\).
The set of all element faces \(\face\) is denoted by \(\Ffaces\).
The broken polynomial space of degree at most \(k\geq 1\) is defined as
\begin{subequations} \label{eq:BrokenandVapprox}
\begin{equation}\label{eq:BrokenPolySpace}
	\PolyAtMost{3}{k}[\meshh] = \setc[\big]{v_h\in\lebfuns{2}{\domain}}{\restrict{v_h}{\element} \in \PolyAtMost{3}{k}[\element] \text{ for all } \element\in\meshh},
\end{equation}
which are polynomials on every element \(\element\) and, in general, \discontinuous{} across element faces \(\face\).
The vector valued ansatz space for the magnetic and electric field is given by the broken finite element space of degree \(k\) defined as
\begin{equation}\label{eq:Vapprox}
	\Vapprox = \PolyAtMost{3}{k}[\meshh]^3.
\end{equation}
\end{subequations}

The \dG{} method is a non-conforming method in the sense that a function \(\vvU_h \in \Vapprox\) does not admit a \(\curl\) on the whole domain \(\domain\), \ie{}, \(\Vapprox \not\subset \curlfuns[]{\domain}\).
Therefore, we need to introduce a discretized \(\curl\) operator acting on \(\Vapprox\).
One way to do so is given by means of the central flux discretization, see, \eg{}, \cite{Hes.W.2002}.
The notation is based on \cite{Hoc.S.2016}.
We define the discrete operator \(\curl_h \defpnt \Vapprox \to \Vapprox\)
such that for all \(\vvphi_h\in\Vapprox\) it holds
\begin{equation}\label{eq:DiscreteCurlIdea}
  \begin{aligned}
    \int_{\domain}\curl_h\vvU_h\cdot\vvphi_h\dx 
    =& \sum_{\element\in\meshh} \int_{\element}\curl \vvU_h \cdot\vvphi_h\dx \\
    &- \sum_{\face\in\Ffaces} \int_{\face}\jmp{\vvU_h\times \vvn_{\face}}_\face \cdot \avg{\vvphi_h}_\face\ds.
  \end{aligned}
\end{equation}
Here, \(\jmp{\vvU_h\times\vvn_\face}_\face\) denotes the tangential jump of \(\vvU_h\) and \(\avg{\vvphi_h}_\face\) a weighted average of \(\vvphi_h\) on \(\face\) defined below \cref{eq:WeightedAverage}.
The first sum on the right hand side of \cref{eq:DiscreteCurlIdea} acts locally on single elements, thus decoupling the action of the curl while the second sum couples neighboring elements through tangential jumps.
The coupling terms are referred to as numerical fluxes and they penalize non-zero tangential jumps across faces.
We recall that functions \(\Hf\in\curlfuns{\domain}\) have zero tangential jumps across faces, \ie{},
\begin{equation*}
  \jmp{\Hf \times \vvn_{\face}}_\face = 0. 
\end{equation*}
By \(\curl_{h,0} \defpnt \Vapprox \to \Vapprox\) we denote a discrete operator that additionally enforces homogeneous tangential boundary conditions of the electric field \cref{eq:PECBoundary}.

In order to incorporate the surface current \(\Jsurf[]\), we follow the idea of \cref{eq:DiscreteCurlIdea}, but instead of penalizing zero tangential jumps, we apply the inhomogeneous interface condition \cref{eq:InterfaceCond:Tangential} for all faces \(\face\in\Finterior\) with \(\face\subset\Fint\), \ie{},
\begin{equation*}
  \jmp{\Hf\times \vvn_\face}_\face = \restrict[\big]{\Jsurf[]}{\face}.
\end{equation*}
Equation \cref{eq:DiscreteCurlIdea} motivates to define an extension \(\Jsurfh[] \in \Vapprox\)
via
\begin{equation}\label{eq:DiscreteLiftIdea}
  \int_\domain \Jsurfh\cdot \vvphi_h\dx
  = \sum_{\face\in\Ffaces, \face\subset\Fint} \int_{\face} \Jsurf[]\cdot \avg{\vvphi_h}\ds
\end{equation}
for all \(\vvphi_h \in \Vapprox\).
Note that \(\Jsurf[]\) is defined only on the interface \(\Fint\) whereas \(\Jsurfh[]\) is defined on the whole domain \(\domain\), with support only on elements adjacent to \(\Fint\).
We end up with the following spatially discrete system of differential equations
\begin{subequations}\label{eq:FirstHandwavingDiscretization}
  \begin{align}
    \dt \Hfh(t) &= -\Msus[-1]\curl_{h,0}\Efh(t),\\
    \dt \Efh(t) &= \Eperm[-1]\curl_h\Hfh(t) - \Jvolh(t) - \Jsurfh(t),
  \end{align}
\end{subequations}
for \(t\geq 0\).
Here, the inhomogeneous interface conditions \cref{eq:InterfaceCond:Tangential} are incorporated by \(\Jsurfh[]\) that acts like an artificial current on the evolution of the electric field.
We refer to \cref{subsec:DiscreteSetting,subsec:SpatialDiscretization} for the precise definitions.

We integrate the spatially discrete system in time by the second order \leapfrog{} method.
Let \(\tau > 0\) be the \timestep{} size and \(t_n = n\tau\) for \(n\in\bbN\).
The fully discrete scheme then reads
\begin{align*}
  \Hfh[n+1/2] - \Hfh[n] &= -\frac{\tau}{2}\Msus[-1]\curl_{h,0} \Efh[n],\\
  \Efh[n+1] - \Efh[n] &= \tau\Eperm[-1]\curl_h \Hfh[n+1/2]
    - \frac{\tau}{2}(\Jvolh[n] + \Jvolh[n+1])
    - \frac{\tau}{2}(\Jsurfh[n] + \Jsurfh[n+1]),\\
  \Hfh[n+1] - \Hfh[n+1/2] &= -\frac{\tau}{2} \Msus[-1]\curl_{h,0} \Efh[n+1],
\end{align*}
for \(n\geq 0\), starting from appropriate initial values \(\bigl(\Hfh[0], \Efh[0]\bigr)\in \Vapprox^2\).
Other time integration schemes can be applied to \cref{eq:FirstHandwavingDiscretization} as well.

It is well known that the explicit \leapfrog{} scheme exhibits a \stepsize{} restriction, which is also known as the Courant-Friedrichs-Lewy (CFL) condition.
The scheme is only stable for \timestep{} sizes \(\tau < \tauCFL\), with \(\tauCFL \sim h_\text{min}\), where \(h_\text{min}\) denotes the diameter of the smallest element \(\element\in\meshh\).

\subsection*{Contributions of the paper}%

The challenges associated with interface problems have been thoroughly investigated from both analytical and numerical perspectives, albeit within a geometric framework different from the one specified earlier. 
In that context, it is assumed that a positive distance exists between the interface and the domain boundary. 
For example, \wellposedness{} and regularity of \quasilinear{} Maxwell equations for such a geometric setting is found in \cite{Sch.S.2022}.
From a numerical point of view, finite element methods have been explored in \cite{Che.Z.1998} concerning elliptic and parabolic problems, and in \cite{Dek.2017,Dek.S.2012} for hyperbolic equations.
However, these results are not applicable to the problem described in \eqref{eq:MaxwellEquations}, \eqref{eq:PECBoundary} and \eqref{eq:InterfaceCond}.

In a recent study by Dörich and Zerulla \cite{Dor.Z.2023}, a different technique is employed to establish \wellposedness{} and regularity for the model problem \eqref{eq:MaxwellEquations}, \eqref{eq:PECBoundary} and \eqref{eq:InterfaceCond}.
on the cuboid domain depicted in \cref{fig:NormalDomain}.
From a numerical perspective, the discontinuous Galerkin time-domain method was successfully applied to an interface problem concerned with \Graphene{} sheets in the above mentioned setting, see, \eg{}, \cite{Wer.W.M.Et.2015,Wer.K.B.Et.2016}.
There, the focus is on the physical modelling of such sheets.
Their excellent numerical results motivate a thorough mathematical error analysis.

In this paper, we provide a mathematical framework that is suitable for both, analysis and numerics of the problem at hand. 
We prove \wellposedness{} and stability for the governing equations building up on the techniques in \cite{Dor.Z.2023}.
Transferring the ideas from analysis, a rigorous spatial and full discretization error analysis is provided for the numerical scheme.
Under suitable regularity conditions on the exact solution, we prove that the error of the scheme is of second order in time and of \(k\)th order in space with respect to the \(L^2\)-norm, \ie{},
\begin{equation*}
  \norm{\bigl(\Hf,\Ef\bigr)(t_n) - (\Hfh[n], \Efh[n])}_{\lebfuns{2}{\domain}^3 \times\lebfuns{2}{\domain}^3} \leq C(\tau^2 + h^{k}), \quad 0 \leq t_n \leq T.
\end{equation*}
Furthermore, the results are underpinned by several numerical examples showing the sharpness of the estimates with respect to spatial regularity.

Note that the results are consistent with the case where the surface current vanishes, \ie{}, \(\Jsurf[] = 0\).
However, new techniques are required for \(\Jsurf[] \neq 0\).
We address the problems in brief.

One of the challenges is that by the interface condition \cref{eq:InterfaceCond:Tangential}, the \statespace{} \(\curlfuns{\domain}\times \hcurlfuns{\domain}\), typically used for the evolution of linear Maxwell equations, is no longer suitable for the problem described by \eqref{eq:MaxwellEquations}, \eqref{eq:PECBoundary} and \eqref{eq:InterfaceCond}.
Circumventing this, we enlarge the \statespace{} with functions \(\vvV\in\lebfuns{2}{\domain}^3\) that only possess a weak \(\curl\) on each sub-polyhedra, \ie{}, \(\vvV_\pm \in \curlfuns{\domainlr}\).
This causes several problems both from an analytical and numerical perspective.
Analytically, \scsemigroup{} techniques are no longer applicable and numerically, we must handle a non-consistent discretization.
Motivated by the treatment of \inhomogeneous{} Dirichlet boundary conditions, we modify the problem that we can treat it in a standard way. 
Nonetheless, since the interface \(\Fint\) intersects with the boundary, special care is necessary to treat the perfectly conducting boundary conditions \cref{eq:PECBoundary} correctly in the modified problem.
For this, we extend the ideas from \cite{Dor.Z.2023} to the more general geometric setting described above.
They turn out to be essential for both analysis and numerics.

\subsection*{Structure of the paper}%

In \cref{sec:Wellposedness}, we first introduce a suitable analytical framework for
the problem described by \eqref{eq:MaxwellEquations}, \eqref{eq:PECBoundary} and \eqref{eq:InterfaceCond}.
We proceed by presenting the main result of this section concerning an existence and stability result for the analytical problem.
With the strategy of proof outlined, we introduce an important extension result that is later used frequently.
The section is closed with the proof of the main result.

\Cref{sec:SpatialDiscretization} is concerned with space discretization.
We first provide a standard description of the \discontinuous{} Galerkin method and point out in detail how inhomogeneous interface problems are treated.
After that, the main result of this section provides an error bound on the \semidiscretization.
The section carries on with the discussion of an important extension of the \semidiscrete{} scheme utilizing a nodal interpolation on the interface.
The section again closes with the remaining proofs.

\Cref{sec:SpatialDiscretization} is concerned with space discretization.
We provide a detailed description of the \discontinuous{} Galerkin method and point out in detail how inhomogeneous interface problems are treated.
We present the main result of this section consisting of error bounds for the \semidiscretization{} as well as an important extension of the \semidiscrete{} scheme utilizing a nodal interpolation on the interface.
The section again closes with the remaining proofs.

The main result of \cref{sec:FullDiscretization} is concerned with an error bound on 
the full discretization.
We first prove stability of the scheme and provide afterwards the proof of the main result.

In \cref{sec:NumericalExperiments}, we provide several numerical experiments in a two-dimensional setting that confirm our theoretical findings.

\section{\Wellposedness{}}\label{sec:Wellposedness}

\subsection*{General setting and notation}\label{subsec:Setting}

The volume charge density \(\rhovol\) is determined by the volume current \(\Jvol\) through
\begin{equation}\label{eq:VolumeChargeCurrentRelation}
    \rhovolpm(t) = \rhovolpm(0) + \int_0^t \div \Jvolpm(s)\ \mathrm{d}s
\end{equation}
on \(\domainlr\), for \(t\geq 0\). Equation \cref{eq:VolumeChargeCurrentRelation} is called the continuity relation for electricity.
It is well known that the divergence conditions \cref{eq:MaxwellEquations:H,eq:MaxwellEquations:E} and the magnetic boundary condition in \cref{eq:PECBoundary} hold, if they are valid for \(t=0\) and \cref{eq:VolumeChargeCurrentRelation} holds.
We refer to \cite[Sec.~I.4.1.2]{Dau.L.1990.MANM} for details.

A similar relation exists for the surface charge density \(\rhosurf\). 
It is determined by the volume current \(\Jvol\) and the surface current \(\Jsurf\) through
\begin{equation}\label{eq:SurfaceChargeCurrentRelation}
    \rhosurf(t) = \rhosurf(0) + \int_0^t \div_\Fint \Jsurf(s) - \jmp{\Jvol(s)\cdot\nint}_\Fint\ \mathrm{d}s
\end{equation}
on \(\Fint\), for \(t\geq 0\), where \(\div_\Fint\) denotes the two-dimensional divergence on \(\Fint\).
It is shown in \cite[Lemma~8.1]{Sch.S.2022} that equations \cref{eq:InterfaceCond:Normal} are valid for \(t\geq 0\) if they are valid for \(t=0\) and \cref{eq:SurfaceChargeCurrentRelation} holds.
Thus, it remains to solve the \(\curl\)-equations in \cref{eq:MaxwellEquations} subject to the boundary conditions \cref{eq:PECBoundary} and the tangential interface conditions \cref{eq:InterfaceCond:Tangential}.

The speed of light is denoted by \(\solpm=(\Msuspm\Epermpm)^{-1/2}\) and we use the notation
\begin{equation*}
  \eta_\infty = \max\monosetc{\eta_-,\eta_+},
  \quad
  \eta \in \monosetc{\Eperm[],\Msus[],\sol}
\end{equation*}
for piecewise defined constants.
We employ the weighted inner \(L^2\)-products
\begin{equation} \label{eq:weightedL2}
    \sprodmu{\cdot}{\cdot} = \sprod{\mu\cdot}{\cdot}_{\lebfuns{2}{\domain}},
    \hspace{1em} \sprodeps{\cdot}{\cdot} = \sprod{\varepsilon\cdot}{\cdot}_{\lebfuns{2}{\domain}},
    \hspace{1em} \sprodmueps{\cdot}{\cdot} = \sprodmu{\cdot}{\cdot} + \sprodeps{\cdot}{\cdot}
\end{equation}
and their induced norms \(\normmu{\cdot}\), \(\normeps{\cdot}\) and \(\normmueps{\cdot}\).
Note that they are equivalent to the standard \(\lebfuns{2}{Q}\)-norm.

The spaces \(\sobfuns{s}{\domain}\) for \(s>0\) denote fractional Sobolev spaces.
We write \(\norm{\cdot}_{\sobfuns{s}{\domain}}\) and \(\singlenorm{\cdot}_{\sobfuns{s}{\domain}}\) for their associated norms and semi-norms and refer to \cite{Edm.E.2023.FSSI} for details.
On the interface \(\Fint\), we denote by \(H^{1/2}(\Fint)\) the fractional Sobolev space, and by  \(H^{-1/2}(\Fint)\) its dual with pivot \(L^2(\Fint)\).

We introduce the space of functions that exhibit a weak variation \(\curl\)
\begin{equation*}
  \curlfuns[]{\domain} = \setc[\big]{\vvV\in\lebfuns{2}{\domain}^3}{\curl \vvV \in\lebfuns{2}{\domain}^3},
\end{equation*}
as well as the subspace that contains functions with a vanishing tangential trace
\begin{equation*}
  \hcurlfuns{\domain} = \setc[\big]{\vvV \in \curlfuns{\domain}}{\restrict{\vvV\times \vvnu}{\Gamma} = 0}.
\end{equation*}
Since we deal with solutions that lag regularity across the interface \(\Fint\), we employ the notion of \piecewise{} spaces.
For \(s > 0\) we denote the \piecewise{} Sobolev spaces
\begin{equation*}
  \psobfuns{s} = \setc[\big]{v\in\lebfuns{2}{\domain}}{v_\pm\in\sobfuns{s}{\domainlr}},
\end{equation*}
and we define analogously
\begin{equation*}
    \pcurlfuns = \setc[\big]{\vvV\in\lebfuns{2}{\domain}^3}{\curl \vvV_\pm \in \lebfuns{2}{\domainlr}^3}.
\end{equation*}
Note that we use the same symbol for both \(\curl\)-operators.
The following connection between \(\curlfuns{\domain}\) and \(\pcurlfuns\) is a simple corollary of Green's formula; see, for example, \cite[Thm.~2.11]{Gir.R.1986.FEMNa}.
\begin{corollary}\label{cor:VanishingJump}
  Let \(\vvV\in\pcurlfuns\).
  It holds \(\vvV\in\curlfuns{\domain}\) if and only if
  \begin{equation*}
    \jmp{\vvV\times\nint}_\Fint = 0,\quad \text{in } \sobfuns{-1/2}{\Fint}^3.
  \end{equation*}
\end{corollary}

\begin{subequations}\label{eq:ContinuousMaxwellOperators}
  We define the Maxwell operators acting on the magnetic and electric field respectively as
  \begin{alignat}{3}\label{eq:PiecewiseOperators}
    \pMOpH &\defpnt \DompMOpH = \pcurlfuns \to \lebfuns{2}{\domain}^3,\ 
    &&\vvH &&\mapsto \varepsilon^{-1}\curl\vvH,\\
    \MOpE &\defpnt \DomMOpE = \hcurlfuns{\domain} \to \lebfuns{2}{\domain}^3,\ 
    &&\vvE &&\mapsto \mu^{-1}\curl \vvE.
  \end{alignat}
The operator acting on the combined field \(\vvu = \bigl(\vvH,\vvE\bigr)\) is defined as
\begin{equation}\label{eq:CombinedMaxwellOperators}
  \pMOp \defpnt \DompMOp = \DompMOpH \times \DomMOpE \to \lebfuns{2}{Q}^6,\ 
  \pMOp = \begin{pmatrix}
    0 && -\MOpE\\
    \pMOpH && 0
  \end{pmatrix}.
\end{equation}
We emphasize that \(\curlfuns{\domain} \subset \pcurlfuns\) and define the restricted operators
\begin{align}\label{eq:NonPiecewiseOperators}
  \MOpH \defpnt \DomMOpH &= \curlfuns{\domain} \to \lebfuns{2}{\domain}^3, 
  & \MOpH &= \restrict{\pMOpH}{\DomMOpH},\\
  \MOp \defpnt \DomMOp &= \curlfuns{\domain} \times \hcurlfuns{\domain} \to \lebfuns{2}{\domain}^6,
  & \MOp &= \restrict{\pMOp}{\DomMOp}.
\end{align}
Note that operators with a hat are always associated with \piecewise{} domains.
We stick to this notation throughout the paper.
\end{subequations}

The Maxwell equations now read: seek \(\bigl(\vvH(t),\vvE(t)\bigr)\in \DompMOpH \times \DomMOpE\) such that
\begin{subequations}\label{eq:MaxwellFormulation}
  \begin{align}
    \dt \vvH &= -\MOpE \vvE                               &&\text{in }[0,T]\times \domain,\\
    \dt \vvE &= \pMOpH \vvH - \varepsilon^{-1} \Jvol   &&\text{in }[0,T]\times \domain,\\
    \vvH(0) &= \vvH^0,\quad \vvE(0) = \vvE^0                               &&\text{in }\domain,\\
    \jmp{\vvH\times \nint}_\Fint &= \Jsurf                            &&\text{on }[0,T]\times \Fint,\label{eq:MaxwellSystem:InterfaceCondition}%
  \end{align}
\end{subequations}

\subsection*{Surface current}\label{subsec:SurfaceCurrentSpaces}

As mentioned in the introduction, the normal component of the surface current vanishes, \ie{}, \(\Jsurf[] \cdot \nint = 0\).
Therefore, we can always identify \(\Jsurf[]\) with a two-dimensional vector-valued function in the tangential plane described by \(\nint\).
We will therefore use the same notation \(\Jsurf[]\) interchangeably as a three-dimensional vector-valued function perpendicular to the normal vector \(\nint\) and as a two-dimensional function on \(\Fint\) identified as a subset of \(\bbR^2\).

The surface current needs to satisfy certain compatibility conditions for the problem to be \wellposed{}.
This boils down to whether \(\Jsurf[]\) is extensible by zero at the boundary of \(\Fint\) in the sense of \(H^{1/2}\).
This leads to a delicate notion of spaces, typically denoted with \(H^{1/2}_{00}\) in the literature, caused by the difficulty of assigning a meaning to traces of functions with Sobolev-regularity \(1/2\). 

We choose here a definition that is intrinsic to \(\Fint\) interpreted as a bounded polygon in \(\bbR^2\) and, therefore, independent of both polyhedra \(\domainlr\).
We follow the construction and notation in \cite{Buf.C.2001,Buf.C.2001b}, and we stick to the notation of this paper.
\begin{figure}
    \centering
    \scalebox{1}{\definecolor{color01}{HTML}{1f77b4}
\definecolor{color02}{HTML}{ff7f0e}
\definecolor{color03}{HTML}{2ca02c}
\definecolor{color04}{HTML}{d62728}
\definecolor{color05}{HTML}{9467bd}
\definecolor{color06}{HTML}{8c564b}
\definecolor{color07}{HTML}{e377c2}
\definecolor{color08}{HTML}{7f7f7f}
\definecolor{color09}{HTML}{bcbd22}
\definecolor{color10}{HTML}{17becf}

\begin{tikzpicture}[scale=1.25]

  \coordinate (A) at (-3.0,-1.0);
  \coordinate (B) at (-2.2,1.2);
  \coordinate (C) at (0.2,1.6);
  \coordinate (D) at (1.0,0.3);
  \coordinate (E) at (2.5,1.0);
  \coordinate (F) at (3.0,-0.6);
  \coordinate (G) at (0.5,-1.5);

  \draw[pattern=north west lines, pattern color=color02, opacity=0.17]
    (A) -- (B) -- (C) -- (D) -- (E) -- (F) -- (G) -- cycle;
    \draw[color02] (A) -- (B) -- (C) -- (D) -- (E) -- (F) -- (G) -- cycle;

  \newcommand{\normal}[4]{
    \path let \p1 = (#1), \p2 = (#2),
          \n1 = {0.5*(\x1+\x2)},
          \n2 = {0.5*(\y1+\y2)} in
          coordinate (#3) at (\n1,\n2);

    \fill (#3) circle(0.5pt);

   \draw[-{Latex}, color01] let 
        \p1 = (#2), 
        \p2 = (#1),
        \n3 = {\x1 - \x2},      %
        \n4 = {\y1 - \y2},      %
        \n5 = {0.1*veclen(\n3,\n4)},%
        \n6 = {-1*(\n4/\n5)},     %
        \n7 = {\n3/\n5}         %
    in
        (#3) -- ++({\n6},{\n7}) node[xshift={0.5*\n6}, yshift={0.5*\n7}] {$\nu_{#4}$};

    \draw[-{Latex}, color01] let 
        \p1 = (#2), 
        \p2 = (#1),
        \n3 = {\x2 - \x1},      %
        \n4 = {\y2 - \y1},      %
        \n5 = {0.1*veclen(\n3,\n4)},%
        \n6 = {\n3/\n5},     %
        \n7 = {\n4/\n5},         %
        \n{xs} = {0.25*\n6 + 0.5*\n7},
        \n{ys} = {0.25*\n7 - 0.5*\n6}
    in
        (#3) -- ++({\n6},{\n7}) node[xshift=\n{xs}, yshift=\n{ys}] {$\tau_{#4}$};

    \path let 
        \p1 = (#2), 
        \p2 = (#1),
        \n3 = {\x1 - \x2},      %
        \n4 = {\y1 - \y2},      %
        \n5 = {0.1*veclen(\n3,\n4)},%
        \n6 = {-1*(\n4/\n5)},     %
        \n7 = {\n3/\n5}         %
    in
        ($(#3) - ({0.75*\n6},{0.75*\n7})$) node[color01] {$\Gamma_{#4}$};
  }

  \normal{A}{B}{MAB}{1}
  \normal{G}{A}{MGA}{2}
  \normal{F}{G}{MFG}{3}
  \normal{E}{F}{MEF}{4}
  \normal{D}{E}{MDE}{5}
  \normal{C}{D}{MCD}{6}
  \normal{B}{C}{MBC}{7}

  \node (0,0) {$\Fint$};

  \draw[-{Latex}, thick] 
    let 
        \p1 = (A),
        \p2 = (B),
        \n1 = {\x1 - \x2},
        \n2 = {\y1 - \y2},
        \n{length} = {0.05*veclen(\n1,\n2)},
        \n{x} = {\n1/\n{length}},
        \n{y} = {\n2/\n{length}},
        \n6 = {-1*(\n{y})},     %
        \n7 = {\n{x}}
    in
        (B) -- ++({\n{x}},{\n{y}}) node[xshift={-0.35*\n6}, yshift={-0.35*\n7}] {\(x_2\)};
    
    \draw[-{Latex}, thick] 
    let 
        \p1 = (A),
        \p2 = (B),
        \n1 = {\x1 - \x2},
        \n2 = {\y1 - \y2},
        \n{length} = {0.05*veclen(\n1,\n2)},
        \n{x} = {\n1/\n{length}},
        \n{y} = {\n2/\n{length}},
        \n6 = {-1*(\n{y})},     %
        \n7 = {\n{x}}
    in
        (B) -- ++({\n{6}},{\n{7}}) node[xshift={0.35*\n{x}}, yshift={0.35*\n{y}}] {\(x_3\)};

\end{tikzpicture}}
    \caption{Interface \(\Fint\) in the \(x_2x_3\)-plane with edges \(\Gamma_i\), tangential vectors \(\tau_i\) and outer normal vectors \(\nu_i\).}
    \label{fig:InterfaceWithEdges}
\end{figure}
Doing so, we denote the edges of \(\Fint\) as \(\Gamma_i \subset \bbR^2\) for \(i\in\monosetc{1,\ldots,N_{\Fint}}\).
The outer unit normal of \(\Gamma_i\) is denoted with \(\nu_i\in\bbR^2\) and the tangential vector of \(\Gamma_i\) as \(\tau_i\in\bbR^2\).
See \cref{fig:InterfaceWithEdges} for an example.

We further define the distance functions
\begin{equation*}
  \delta_{\Gamma_i}(x) = \operatorname{dist}(x,\Gamma_i), \text{ for } x\in\bbR^2.
\end{equation*}

The space of functions on \(\Fint\) with Sobolev-regularity \(1/2\) which are extensible by zero to the whole of \(\bbR^2\) is defined as

\begin{align*}
  \hsobfunszeroo[]{1/2}{\Fint} &= \setc[\big]{f\in\sobfuns[]{1/2}{\Fint}}{\delta_{\Gamma_i}^{-1/2}f \in \lebfuns[]{2}{\Fint}, i=1,\ldots,N_{\Fint}},\\
  \norm{f}_{\hsobfunszeroo[]{1/2}{\Fint}}^2 &=
    \norm{f}_{\sobfuns[]{1/2}{\Fint}}^2 + \sum_{i=1}^{N_{\Fint}} \norm{\delta_{\Gamma_i}^{-1/2}f}_{\lebfuns[]{2}{\Fint}}^2
\end{align*}
The additional constraints thus penalize functions that do not decay rapidly enough to zero in close neighborhood to the boundary of \(\Fint\).
The dual space of \(\hsobfunszeroo[]{1/2}{\Fint}\) with respect to the pivot space \(\lebfuns[]{2}{\Fint}\) is denoted with \(\hsobfunszeroo[]{-1/2}{\Fint}\) and equipped with the introduced dual-norm.

Since we deal with vector-valued functions, we further introduce the space of vector fields with Sobolev-regularity \(1/2\) that are extensible by zero parallel to the edges as
\begin{align*}
  \Hparallelspace &= \setc[\big]{\vvF\in\sobfuns[]{1/2}{\Fint}^2}{\delta_{\Gamma_i}^{-1/2}\vvF\cdot\tau_i \in \lebfuns[]{2}{\Fint}^2, i=1,\ldots,N_{\Fint}},\\
  \norm{\vvF}_{\Hparallelspace}^2 &=
    \norm{\vvF}_{\sobfuns[]{1/2}{\Fint}^2}^2 + \sum_{i=1}^{N_{\Fint}} \norm{\delta_{\Gamma_i}^{-1/2}\vvF\cdot \tau_i}_{\lebfuns[]{2}{\Fint}^2}^2
\end{align*}
The dual with respect to \(\lebfuns[]{2}{\Fint}^2\) is denoted with \(\HparallelspaceDual\) and also equipped with the introduced dual-norm.

We introduce yet another space of functions with Sobolev-regularity \(3/2\), defined by
\begin{equation*}
  \hsobfunszeroo[]{3/2}{\Fint} = \setc{f\in \hsobfuns[]{1}{\Fint}}{
    \nabla_{\Fint} F \in \Hparallelspace 
  },
\end{equation*}
endowed with the associated graph norm.
Note that this notation suggests that there are no compatibility conditions in the direction of the normal derivatives.
And indeed, this definition is equivalent with norms to the space \(\hsobfuns[]{1}{\Fint}\cap H^{3/2}(\Fint)\).
We refer to \cite[Rem.~3.12]{Buf.C.2001} for details.
The dual of \(\hsobfunszeroo[]{3/2}{\Fint}\) with respect to \(\lebfuns[]{2}{\Fint}\) is denoted with \(\sobfuns[]{-3/2}{\Fint}\).
We define \(\div_{\Fint} \defpnt \HparallelspaceDual \to \sobfuns[]{-3/2}{\Fint}\) via duality by
\begin{equation*}
  \sch{\div_{\Fint} \vvF}{f}_{3/2}
    = -\sch{\vvF}{\nabla_{\Fint}f}_{\parallel,1/2},
    \quad f\in \hsobfunszeroo[]{3/2}{\Fint},
\end{equation*}
with the dual pairing in \(\hsobfunszeroo[]{3/2}{\Fint}\) and in \(\Hparallelspace\) on the left- and right-hand side, respectively.

We finally define the space
\begin{equation}\label{def:TraceSpaceParallel}
  \TraceSpaceParallel
  = \setc{\vvF\in \HparallelspaceDual}{\div_\Fint \vvF \in \hsobfunszeroo[]{-1/2}{\Fint}},
\end{equation}
endowed with the graph norm.
This space will suffice as a domain for the spatial component of the surface current, since the tangential jump is surjective onto this space.
The following result follows from \cite[Thm.~6.6]{Buf.C.2001b}.
\begin{corollary}\label{cor:NewExtensionCoro}
  The tangential jump 
  \[
    \jmp{\cdot \times \nint}_\Fint\defpnt \pcurlfuns \to \TraceSpaceParallel,
    \quad
    \vvJ\to \jmp{\vvJ\times \nint}_{\Fint}
  \]
  is a linear, bounded and surjective map.
  It admits a linear and bounded right-inverse
    \[
        \mathcal{R}_{\Fint}\defpnt \TraceSpaceParallel \to \pcurlfuns,
    \]
    whose restriction to \(\sobfuns[]{1/2}{\Fint}^2\) is bounded in \(\psobfuns{1}^3\).
\end{corollary}
\begin{proof}
    By \cite[Thm~6.6]{Buf.C.2001b},
    the tangential trace \(\vvJ_{\pm} \mapsto \restrict{(\vvJ_{\pm}\times\nint)}{\Fint}\) is a bounded, linear and surjective operator from \(\curlfuns[]{\domainlr}\) onto \(\TraceSpaceParallel\).
    Thus, we denote with \(\mathcal{R_{\domainlr}} \defpnt \TraceSpaceParallel \to \curlfuns[]{\domainlr}\) some right inverse.
    The operator 
    \[
      \mathcal{R}_{\Fint} \defpnt \TraceSpaceParallel \to \pcurlfuns
    \]
    defined as \(\restrict[]{\mathcal{R}_\Fint \vvJ}{\domainlr} =  \pm \frac{1}{2}\mathcal{R}_{\domainlr}\vvJ\), satisfies the statement.
\end{proof}

\subsection*{Main result}\label{subsec:MainResult}

Our first main result states existence, uniqueness and stability under appropriate regularity assumptions in weak variational spaces.
\begin{theorem}\label{thm:UniqueExistence}
  If \(\vvu^0 \in \DompMOp\), \(\Jvol \in \confuns[\big]{0}{[0,T], \DomMOpH} + \confuns[\big]{1}{[0,T], \lebfuns{2}{\domain}^3},\) and
  \begin{equation*}
    \Jsurf[] \in 
    \confuns[\big]{2}{[0,T], \TraceSpaceParallel},
  \end{equation*}
  then there exists a unique solution 
  \begin{equation*}
    \vvu = \bigl(\vvH,\vvE\bigr) \in \confuns[\big]{0}{[0,T], \DompMOp}\cap\confuns[\big]{1}{[0,T],\lebfuns{2}{\domain}^6}
  \end{equation*}
  of \cref{eq:MaxwellFormulation}.
  Furthermore, for all \(t \in [0,T]\) it holds
  \begin{equation}\label{eq:StabilityBound}
    \begin{aligned}
      \normmueps{\vvu(t)} \lesssim& \normmueps{\vvu^0} 
        +  \norm{\Jsurf(0)}_{\lebfuns{2}{\Fint}^3}
        +  \norm{\Jsurf(t)}_{\lebfuns{2}{\Fint}^3}\\
        &+  \int_0^t \norm{\Jvol(s)}_{\lebfuns{2}{\domain}^3}\ \mathrm{d}s
        +  \int_0^t \norm{\dt\Jsurf(s)}_{\TraceSpaceParallel}\ \mathrm{d}s\\
        &+ 
          \int_0^t\norm[]{\Jsurf[](s)}_{\TraceSpaceParallel}\ \mathrm{d}s,
    \end{aligned}
  \end{equation}
  with a constant which is independent of \(\Jvol, \Jsurf\) and \(\vvu\).
\end{theorem}
\begin{remark}
  Note that in the theory of \(C^0\)-semigroups, one can trade temporal and spatial regularity of the inhomogeneity \cite[Thm.~2.4]{Paz.1983.SLOA}.
  Thus, either \(\Jvol \in \confuns[\big]{1}{[0,T], L^2(Q)^3}\) or \(\Jvol \in \confuns[\big]{0}{[0,T], \DomMOpH}\) gives rise to strong solutions. Compare Corollary~2.5 and Corollary~2.6 in \cite{Paz.1983.SLOA}, respectively.
  By the linear nature of our problem, we write
  \(
    \Jvol \in \confuns[\big]{0}{[0,T], \DomMOpH} + \confuns[\big]{1}{[0,T], \lebfuns{2}{\domain}^3}
  \).
\end{remark}

\subsection*{Strategy of proof}

It is well known that the operator \(\MOp \defpnt \DomMOp\to\lebfuns{2}{\domain}^6\) is the generator of a unitary \scsemigroup{}, whereas the operator \(\pMOp\defpnt \DompMOp \to \lebfuns{2}{\domain}^6\) does not inherit any good properties, see, \eg{}, \cite[Rem.~2.1]{Dor.Z.2023}.
  Thus, we define a lifting \(\JextH[] = \mathcal{R}_{\Fint}\Jsurf[]\)
  and introduce a shifted magnetic field \(\vtH = \vvH - \JextH\) with a vanishing tangential jump \(\jmp{\vtH\times\nint}_\Fint = 0\).
The shifted system then reads: seek \(\bigl(\vtH(t),\vvE(t)\bigr) \in \DomMOpH  \times \DomMOpH\) such that
\begin{subequations}\label{eq:ShiftedMaxwellSystem}
  \begin{align}
    \dt \vtH &= -\MOpE \vvE - \vtJ_{1}    &&\text{in }[0,T]\times \domain,\\
    \dt \vvE &= \MOpH \vtH -\vtJ_{2}      &&\text{in }[0,T]\times \domain,\\
    \vtH(0) &= \vtH^0,\quad \vvE(0) = \vvE^0               &&\text{in }\domain,
  \end{align}
\end{subequations}
with \(\vtJ_1 = -\dt\JextH\), \(\vtJ_2 = -\varepsilon^{-1}\Jvol + \varepsilon^{-1}\curl\JextH\) and \(\vtH^0 = \vvH^0 - \JextH(0)\).
We then use \scsemigroup{} theory to obtain unique existence and stability for the shifted problem, and, therefore also for \cref{eq:MaxwellFormulation}.

\subsection*{Proof of main result}\label{subsec:ProofMainResult}
We write \cref{eq:ShiftedMaxwellSystem} in the compact form:
seek \(\vtu(t) \in \DomMOp\) such that
\begin{subequations}\label{eq:ShiftedMaxwellSystemCombined}
  \begin{align}
    \dt \vtu(t) &= \MOp\vtu(t) + \vtj(t), \text{ for } t\in [0,T],\\
    \vtu(0) &= \vtu^0,
  \end{align}
\end{subequations}
with \(\vtu(t) = \bigl(\vtH(t),\vvE(t)\bigr) \), \(\vtj(t) = \bigl(-\dt\JextH(t), -\varepsilon^{-1}\Jvol(t) + \varepsilon^{-1}\curl\JextH(t)\bigr)\), \(\jextH(t) = \bigl(\JextH(t), 0\bigr)\) and \(\vtu^0 = \vvu^0 - \jextH(0)\).
This problem fits into the framework of Cauchy problems and standard theorems for existence and stability can be applied.

\begin{lemma}\label{lem:UniqueExistenceShifted}
  If \(\vvu^0 \in \DompMOp\), \(\Jvol \in \confuns[\big]{0}{[0,T], \DomMOpH} + \confuns[\big]{1}{[0,T], \lebfuns{2}{\domain}^3},\) and
\begin{equation*}
    \Jsurf[] \in 
    \confuns[\big]{2}{[0,T], \TraceSpaceParallel},
  \end{equation*}
  then there exists a unique solution
  \begin{equation*}
    \vtu = \bigl(\vtH, \vvE\bigr) 
    \in \confuns[\big]{0}{[0,T], \DomMOp} \cap \confuns[\big]{1}{[0,T], \lebfuns{2}{\domain}^6} 
  \end{equation*}
  of \cref{eq:ShiftedMaxwellSystem} given by
  \begin{equation}\label{eq:VoCShifted}
    \vtu(t) = e^{t\MOp}\vtu^0 + \int_0^t e^{(t-s)\MOp}\vtj(s)\ \mathrm{d}s.
  \end{equation}
\end{lemma}
\begin{proof}
  In view of standard results for Cauchy problems, \cf{} \cite[Thm.~4.2.4, Cor.~4.2.5 ]{Paz.1983.SLOA}, we need to check the two conditions
  \begin{equation*}
    \vtu^0\in\DomMOp, \quad \vtj \in \confuns[\big]{0}{[0,T]; \DomMOp} + \confuns[\big]{1}{[0,T]; \lebfuns{2}{\domain}^6}.
  \end{equation*}

  By construction \(\vtu^0 \in \DompMOp\).
  Furthermore, by \cref{cor:NewExtensionCoro} it holds
  \begin{equation*}
    \jmp{\vtH^0\times \nint}_\Fint 
      = \jmp{\vvH^0\times\nint}_\Fint - \jmp{\JextH(0)\times\nint}_\Fint 
      = 0.
    \end{equation*}
  Therefore, we conclude with \cref{cor:VanishingJump} that \(\vtu^0\in\DomMOp\).
  Again by \cref{cor:NewExtensionCoro}, we see that
  \begin{equation*}
    \varepsilon^{-1}\curl\JextH, \dt \JextH \in \confuns[\big]{1}{[0,T]; \lebfuns{2}{\domain}^3}.
  \end{equation*}
  This proves the claim together with the assumption on \(\Jvol\).
\end{proof}

We are now able to prove the first main result.
\begin{proof}[{Proof of \cref{thm:UniqueExistence}}]\label{prf:UniqueExistence}
  A straightforward calculation shows that \(\vvu = \vtu + \jextH\) solves \cref{eq:MaxwellFormulation}.
  This solution is unique as a consequence of the uniqueness in \cref{lem:UniqueExistenceShifted}.
  
  It remains to prove stability.
  Taking norms in \cref{eq:VoCShifted}, we obtain
  \begin{align*}
    \normmueps{\vtu(t)} 
      \leq&\ \normmueps{\vvu^0} + \normmu{\JextH(0)}\\
      &+\int_0^t 
        \normeps[\big]{\varepsilon^{-1}\Jvol(s)}
        + \normeps[\big]{\varepsilon^{-1} \curl\JextH(s)}
        + \normmu[\big]{\dt \JextH(s)}\ \mathrm{d}s\\
      \leq&\ \normmueps{\vvu^0} + \sqrt{\Msusmax}\norm{\JextH(0)}_{\lebfuns{2}{\domain}^3}\\
      &+\frac{1}{\sqrt{\delta}}\int_0^t \norm{\Jvol(s)}_{\lebfuns{2}{\domain}^3}\ \mathrm{d}s
      +\sqrt{\Msusmax}\int_0^t \norm{\dt \JextH(s)}_{\lebfuns{2}{\domain}^3}\ \mathrm{d}s\\
      &+\frac{1}{\sqrt{\delta}}\int_0^t \norm[\big]{\curl\JextH(s)}_{\lebfuns{2}{\domain}^3}\ \mathrm{d}s.
  \end{align*}
  By \cref{cor:NewExtensionCoro}, we estimate further
  \begin{align*}
    \normmueps{\vtu(t)}
      \lesssim&\ \normmueps{\vvu^0} + \norm{\Jsurf(0)}_{\lebfuns{2}{\Fint}^3}\\
      &+ \int_0^t \norm{\Jvol(s)}_{\lebfuns{2}{\domain}^3}\ \mathrm{d}s
      + \int_0^t \norm{\dt\Jsurf(s)}_{\TraceSpaceParallel}\ \mathrm{d}s\\
      &+ \int_0^t \norm[]{\Jsurf[](s)}_{\TraceSpaceParallel}\ \mathrm{d}s.
  \end{align*}
  The claim follows with \(\normmueps{\vvu(t)} \leq \normmueps{\vtu(t)} + \normmueps{\jextH(t)}\).
\end{proof}

\section{Spatial discretization}\label{sec:SpatialDiscretization}

In this chapter, we introduce a concrete space discretization and rigorously derive the discrete curl in \cref{eq:DiscreteCurlIdea} and the discrete extension in \cref{eq:DiscreteLiftIdea}.
We first present our two main results involving rigorous error bounds for the spatially discrete scheme and an extended scheme.
The chapter then proceeds with a spatially discrete analogue of the stability bound \cref{thm:UniqueExistence} and is concluded with the proofs of the main results.

\subsection*{Discrete setting}\label{subsec:DiscreteSetting}
We denote with \(\meshh\) matching simplicial meshes of the domain \(\domain\), generated by a reference element \(\widehat{\element}\).
The subscript \(h\) indicates the \meshsize{}, defined as \(h = \max_{\element\in\meshh} h_\element\), where \(h_\element\) denotes the diameter of a mesh element \(\element\).
Furthermore, we assume that the mesh sequence is shape regular in the sense of \cite[Def.~11.2]{Ern.G.2021.FEAI}.
Thus, there exists \(\sigma > 0\) independent of \(h\) such that
\(h_\element \leq \sigma\rho_\element\), where \(\rho_\element\) denotes the diameter of the largest inscribing ball of \(\element\).

We collect the faces \(\face\) of all mesh elements in the set \(\Ffaces = \Finterior \cup \Fboundary\), where \(\Finterior\) denotes the set of all faces in the interior of \(\domain\) and \(\Fboundary\) the set of all faces on the boundary \(\partial\domain\). 
Refer to \cite[Def.~8.10]{Ern.G.2021.FEAI} for a precise definition of mesh faces.

The outer unit normal vector of \(\element\) is denoted by \(\nvK\).
Every interior face \(\face\in\Finterior\) intersects two elements \(\elementKFl\) and \(\elementKFr\).
The order of the elements is arbitrary, but fixed.
We choose the unit normal \(\nvF\) to \(\face\) pointing from \(\elementKFl\) to \(\elementKFr\).
For boundary faces \(\face \in \Fboundary\), we choose the unit normal \(\nvF\) to \(\face\) as the
outer unit normal vector \(\nvK\) of the associated element \(\element\).  

Let \(\face\) be an interior face and \(v\defpnt \domain \to \bbR\) be a function that admits a \welldefined{} trace on \(\face\).
The weighted average of \(v\) on the face \(\face\) is defined as
\begin{subequations}
  
  \begin{equation}\label{eq:WeightedAverage}
    \avg{v}^\omega_\face = 
    \frac{
      \omega_{\elementKFl}\restrict[\big]{(\restrict{v}{\elementKFl})}{\face}
      + \omega_{\elementKFr} \restrict[\big]{(\restrict{v}{\elementKFr})}{\face}}
      {\omega_{\elementKFl} + \omega_{\elementKFr}},
    \end{equation}
    where \(\omega \defpnt \domain \to (0,\infty)\) denotes a positive weight function that is \piecewise{} constant, \ie{}, \(\restrict{\omega}{K} \equiv \omega_K\) for all \(\element\in\meshh\).
    Analogously, we define the jump of \(v\) on \(\face\) as
    \begin{equation}
      \jmp{v}_\face = \restrict[\big]{(\restrict{v}{\elementKFr})}{\face} - \restrict[\big]{(\restrict{v}{\elementKFl})}{\face}.
    \end{equation}
\end{subequations}
For vector fields, both definitions hold \componentwise{}.

The following assumption is necessary to resolve the interface conditions \cref{eq:InterfaceCond}.
\begin{assumption} \label{ass:interface_aligned}
  We assume that every element \(K\in\meshh\) lies completely on one side of the interface \(\Fint\), \ie{},
  \begin{equation*}
    \element \cap \Fint = \emptyset, \quad\text{for all } \element\in\meshh.
  \end{equation*}
  Furthermore, we assume that the unit normal \(\nvF\) for every face \(\face\in\Finterior\) with \(F\subset\Fint\) points in the same direction as \(\nint\), \ie{},
  \begin{equation*}
    \nvF \cdot \nint = 1, \quad\text{for all } \face\in\Finterior \text{ with }\face\subset\Fint.
  \end{equation*}
  The set of all faces \(\face\in\Finterior\) with \(\face\subset \Fint\) is denoted by \(\Finterface\).
\end{assumption}

Similar to the definition of the broken polynomial space \cref{eq:BrokenPolySpace}, we introduce for \(s\geq 0\) the broken Sobolev space on \(\meshh\) defined by
\begin{subequations}
\begin{equation}\label{eq:BrokenSobolev}
  \sobfuns{s}{\meshh} = \setc{v\in\lebfuns{2}{\domain}}{\restrict{u}{\element}\in\sobfuns{s}{\element} \text{ for all } \element \in \meshh} .
\end{equation}
The \piecewise{} \seminorm{} on \(\sobfuns{s}{\meshh}\) is denoted by \(\singlenorm{\cdot}_{\sobfuns{s}{\meshh}}\) and we define 
\begin{equation}
  \norm{\cdot}_{\sobfuns[]{s}{\meshh}}^2 = 
    \norm{\cdot}_{\lebfuns[]{2}{\domain}}^2 
    + \singlenorm{\cdot}_{\sobfuns{s}{\meshh}}^2.
\end{equation}
\end{subequations}

In the following, more regularity of the solution is assumed such that it admits classical traces on element faces.
Therefore, we define the spaces
\begin{subequations}\label{eq:MoreRegularSolutionSpaces}
  \begin{equation}
    \pVastH = \DompMOpH \cap \sobfuns{1}{\meshh}^3,
    \quad\VastE = \DomMOpE \cap \sobfuns{1}{\meshh}^3,
    \quad\pVast = \pVastH \times \VastE
  \end{equation}
  and the restricted spaces  
  \begin{equation}
    \VastH = \DomMOpH \cap \sobfuns{1}{\meshh}^3,
    \quad \Vast = \VastH \times \VastE.
  \end{equation}
\end{subequations}
Since functions of the approximation space \(\Vapprox\), defined in \cref{eq:BrokenPolySpace}, do not admit a \welldefined{} \(\curl\), we introduce the following spaces containing both the analytical solution and the approximation
\begin{subequations}\label{eq:SpacesContainingBothAnaAndApprox}
  \begin{equation}
    \pVasthH = \pVastH + \Vapprox,
    \quad \VasthE = \VastE + \Vapprox,
    \pVasth = \pVasthH \times \VasthE,
  \end{equation}
  and similarly
  \begin{equation}
    \VasthH = \VastH + \Vapprox,
    \quad \Vasth = \VasthH \times \VasthE.
  \end{equation}
\end{subequations}

\begin{remark}
  Note that the results are not specific to matching \simplicial{} meshes, but are also valid for quadrilateral meshes and general meshes as described in \cite[Sec.~1.2]{Di.E.2012.MADG}.
  We omit the details for the sake of presentation.
\end{remark}

\subsection*{Spatial discretization}\label{subsec:SpatialDiscretization}

\begin{subequations}\label{eq:DiscreteMaxwellOperators}
  As motivated in the introduction with \cref{eq:DiscreteLiftIdea}, we define the discrete lift operator 
  \begin{equation}\label{eq:LiftOp}
    \LiftOp\defpnt \lebfuns{2}{\Fint}^3 \to \Vapprox,\quad
    \sprodeps{\LiftOp \vvV}{\TeFunH} = -\sum_{\face\in\Finterface} \sprod{\restrict{\vvV}{\face}}{\avg{\TeFunH}^{\overline{\mu c}}_\face}
  \end{equation}
  for \(\TeFunH\in\Vapprox\), and the discrete magnetic Maxwell operator \(\pDgMOpH\defpnt \pVasthH \to \Vapprox\)
  \begin{equation}\label{eq:DGpMaxwellOperatorH}
    \sprodeps{\pDgMOpH \vvH}{\TeFunH} =
    \begin{multlined}[t]
      \sum_{\element\in\meshh} \sprod{\vvH}{\curl \TeFunH}_\element\\
      -\sum_{\face\in\Fboundary} \sprod{\vvH\times\nvF}{\TeFunH}_\face 
      -\sum_{\face\in\Finterior} \sprod{\avg{\vvH}^{\mu c}_\face}{\jmp{\TeFunH}_\face\times \nvF}_\face.
    \end{multlined}
  \end{equation}
  Analogously, we define the electric Maxwell operator \(\DgMOpE\defpnt\VasthE \to \Vapprox\) for \(\vvE\in\VasthE\) and \(\TeFunE\in\Vapprox\) by
  \begin{equation}\label{eq:DGMaxwellOperatorE}
    \sprodmu{\DgMOpE\vvE}{\TeFunE}
    = \sum_{\element\in\meshh} \sprod{\vvE}{\curl \TeFunE}_\element 
    - \sum_{\face\in\Finterior} \sprod{\avg{\vvE}^{\varepsilon c}_\face}{\jmp{\TeFunE}_\face\times \nvF}_\face.
  \end{equation}
  This definition incorporates the perfectly conducting boundary condition for the electric field.
  The discrete operator acting on the combined field is defined as
  \begin{equation}\label{eq:CombinedDiscreteMaxwellOperatos}
    \DgMOp\defpnt \pVasth \to \Vapprox^2,\quad \DgMOp = 
    \begin{pmatrix}
      0 & - \DgMOpE\\
      \pDgMOpH & 0
    \end{pmatrix}.
  \end{equation}
  Analogously to \cref{eq:NonPiecewiseOperators}, we define the restrictions
  \begin{align}\label{eq:DiscreteNonPiecewiseOperators}
    \DgMOpH \defpnt \VasthH &\to \Vapprox, 
    & \DgMOpH &= \restrict{\pDgMOpH}{\VasthH},\\
    \DgMOp \defpnt \Vasth &\to \Vapprox^2,
    & \DgMOp &= \restrict{\pDgMOp}{\Vasth}.
  \end{align}
\end{subequations}

The \semidiscrete{} problem now reads: seek \(\bigl(\vvH_h(t),\vvE_h(t)\bigr)\in \Vapprox^2\) such that
\begin{subequations}\label{eq:SemiDiscreteMaxwell}
  \begin{align}
    \dt \vvH_h(t) &= -\DgMOpE \vvE_h(t)     &&\text{for } t\in [0,T],\\
    \dt \vvE_h(t) &= \pDgMOpH \vvH_h(t) -\Jvolh(t) - \Jsurfh(t)     &&\text{for } t\in [0,T],\\
    \vvH_h(0) &= \vvH_h^0,\quad \vvE_h(0) = \vvE_h^0,
  \end{align}
\end{subequations}
where \(\Jsurfh = \LiftOp \Jsurf\), \(\vvH_h^0 = \OProj\vvH^0\), \(\vvE_h^0 = \OProj\vvE^0\) and \(\Jvolh = \OProj\varepsilon^{-1}\Jvol\).
We denote with \(\OProj\defpnt\lebfuns{2}{\domain}\to\PolyAtMost{3}{k}[\meshh]\) the broken \(L^2\)-orthogonal projection defined by
\begin{equation}\label{eq:BrokenOrthProjection}
  \sprod{v - \OProj v}{\phi_h}_{\lebfuns{2}{\domain}} = 0\qquad \text{for all } \phi_h \in \PolyAtMost{k}{3}[\meshh].
\end{equation}
For typical properties of this projection, compare \cite[Sec.~18.4]{Ern.G.2021.FEAI} or \cite[Sec.~1.4.4]{Di.E.2012.MADG}.

The second main result gives an error bound on the spatially discrete solution \(\vvu_h=\bigl(\vvH_h,\vvE_h\bigr)\) of \cref{eq:SemiDiscreteMaxwell}.
For a sufficiently regular problem, we obtain convergence in the mesh parameter \(h\).
The proof is given below.
\begin{theorem}\label{thm:SpatialConvergence}
  Let the solution \(\vvu=\bigl(\vvH,\vvE\bigr)\) of \cref{eq:MaxwellFormulation} satisfy
  \begin{equation}\label{eq:SpatialRegularity}
    \vvu \in \confuns{0}{[0,T], \pVast \cap \sobfuns{1+s}{\meshh}^6} \cap \confuns{1}{[0,T], \lebfuns{2}{\domain}^6},
  \end{equation}
  with \(s\geq 0\).
  Furthermore, let \Cref{ass:interface_aligned} hold.
  Then, the appropriation \(\ufh[]=\bigl(\Hfh[],\Efh[]\bigr)\) defined in \cref{eq:SemiDiscreteMaxwell} satisfies
  \begin{equation*}
    \normmueps{\vvu(t) - \vvu_h(t)} \leq C h^{\regcoeff},
    \qquad 0\leq t \leq T,
  \end{equation*}
  with a constant \(C>0\) independent of \(h\).
  Here, \(\regcoeff=\min\monosetc{s,k}\) with \(k\) denoting the polynomial degree of the approximation space defined in \cref{eq:BrokenPolySpace}.
\end{theorem}
Note that this agrees with the results obtained for the special case \(\Jsurf = 0\), see, \eg{}, \cite{Hes.W.2002}.

\subsection*{Nodal interpolation}\label{subsec:NodalInterpolation}
The calculation of the lift operator \cref{eq:LiftOp} involves the evaluation of integrals over mesh faces.
In practice, those integrals are approximated by quadrature formulas.
This can be quite expensive, since the evaluation may be required at every \timestep{}.
Moreover, in the future, we want to consider currents that depend on the electric field, \ie{}, \(\Jsurf \approx \Jsurf(\Efh[])\).
In such a situation, the evaluation of the lift would cause evaluations of the finite element function at every quadrature point which is quite expensive and should be avoided.
In such cases, nodal discontinuous Galerkin methods are attractive since they allow for a fast evaluation of integrals and functions, see, \eg{}, \cite[App.~2]{Di.E.2012.MADG} for a detailed discussion.
In the following, we construct a scheme that makes use of nodal interpolation of \(\Jsurf\) and provide error bounds.

We specify the construction from \cref{subsec:DiscreteSetting} and choose \(\calN_k = \dim{\PolyAtMost{3}{k}}\) nodes \(\Sigma_{\widehat{\element}} = \monosetc{\sigma_{\widehat{\element},1},\ldots,\sigma_{\widehat{\element},\calN_k}}\) in the closure of the reference element \(\widehat{\element}\).
Then, the Lagrange polynomials, defined by \(\theta_{\element,i}(\sigma_{\element,j}) = \delta_{ij}\) for \(i,j\in\monosetc{1,\ldots,\calN_k}\), form a basis of \(\PolyAtMost{k}{3}[K]\).
Thus, we can define for \(\ell>3/2\) the local interpolation operator %
\begin{equation*}
  \NInterp_\element \defpnt \sobfuns{\ell}{\element} \to \PolyAtMost{k}{3}[\element],
  \quad \NInterp_\element v = \sum_{j=1}^{\calN_k} \restrict{v}{K}(\sigma_{K,i}) \theta_{K,i}
\end{equation*}
and hence, the global interpolation operator by restriction, \ie{},
\begin{equation*}
  \NInterp\defpnt\sobfuns{\ell}{\meshh} \to \PolyAtMost{k}{3}[\meshh],
  \quad \restrict{\NInterp v}{\element} = \NInterp_\element v, \quad \text{for } \element\in\meshh.
\end{equation*}
Note that the interpolation operator acts \componentwise{} for vector fields.

The surface current \(\Jsurf\) is only supported on the interface \(\Fint\) and hence we construct an interpolation operator on the \submesh{} \(\Finterface\).
Therefore, we need the following two assumptions.
\begin{assumption}[{\cite[Ass.~20.1]{Ern.G.2021.FEAI}}]\label{ass:FaceUnisolvence}
  Let \(\widehat{\face}\) be a face of the reference element \(\widehat{\element}\) and denote with \(\Sigma_{\widehat{\face}}\) the nodes that are located on \(\widehat{\face}\), \ie{}, \(\Sigma_{\widehat{\face}} = \Sigma_{\widehat{\element}} \cap \widehat{\face}\). 
  We assume that for any \(p\in\PolyAtMost{3}{k}[\widehat{\element}]\) it holds \(\restrict{p}{\widehat{\face}} \equiv 0\) if and only if
  \(p(\sigma) = 0\) for all \(\sigma \in \Sigma_{\widehat{\face}}\).
\end{assumption}
We also need to make sure how the nodes of neighboring elements come in contact with each other.
\begin{assumption}[{\cite[Ass.~20.3]{Ern.G.2021.FEAI}}]\label{ass:FaceMatching}
  For any face \(F \in \Finterior\) it holds
  \begin{equation*}
    \Sigma_{\elementKFl}\cap\face = \Sigma_{\elementKFr}\cap\face \eqqcolon \Sigma_F.
  \end{equation*}
  We write again \(\Sigma_F = \monosetc{\sigma_{F,1},\ldots,\sigma_{F,\frakN_k}}\).
\end{assumption}

\begin{remark}
  These assumptions ensure that the triple \((\face, \PolyAtMost{2}{k}[\face], \Sigma_{\face})_{\face\in\Finterface}\) is again a finite element for \(\Fint\) in the sense of \cite[Def.~5.2]{Ern.G.2021.FEAI}, see \cite[Lem.~20.2]{Ern.G.2021.FEAI} for details.
  Note that the usual \(\mathbb{P}_k\) and \(\mathbb{Q}_k\) nodal Lagrange elements satisfy both assumptions, see \cite[Sec.~20.2]{Ern.G.2021.FEAI} for details.
\end{remark}

Given \Cref{ass:FaceUnisolvence,ass:FaceMatching}, we are able to define for \(\kappa > 1\) the local interpolation operator%
\begin{equation*}
  \FInterp_\face \defpnt \sobfuns{\kappa}{\face} \to \PolyAtMost{2}{k}[\face],
  \quad\FInterp_\face v = \sum_{j = 1}^{\frakN_k} \restrict{v}{\face}(\sigma_{\face,j}) \theta_{\face,j}
\end{equation*}
and the global interpolation operator
\begin{equation*}
  \FInterp \defpnt \sobfuns{\kappa}{\Finterface} \to \PolyAtMost{2}{k}[\Finterface],\quad
  \restrict{\FInterp v}{\face} = \FInterp_\face v, \quad \text{for } \face\in\Finterface.
\end{equation*}

The problem now reads: seek \(\bigl(\vcH(t), \vcE(t)\bigr) \in \Vapprox^2\) such that
\begin{subequations}\label{eq:InterpolateSemiDiscreteMaxwell}
  \begin{align}
    \dt \vcH_h(t) &= -\DgMOpE \vcE_h(t)     &&\text{for } t\in [0,T],\\
    \dt \vcE_h(t) &= \pDgMOpH \vcH_h(t) -\Jvolh(t) - \JsurfhInterp(t)    &&\text{for } t\in [0,T],\\
    \vcH_h(0) &= \vvH_h^0,\quad \vcE_h(0) = \vvE_h^0,
  \end{align}
\end{subequations}
with \(\JsurfhInterp[] = \LiftOp\FInterp\Jsurf\).
Note that the \semidiscrete{} solutions of \cref{eq:SemiDiscreteMaxwell} and \cref{eq:InterpolateSemiDiscreteMaxwell} only differ in the fact that we use nodal interpolation under the lift operator.
Our third main result concerns the error introduced by this additional approximation.
The proof is given below.
\begin{theorem}\label{thm:InterpolationComparison}
	Let \Cref{ass:interface_aligned,ass:FaceUnisolvence,ass:FaceMatching} hold and further
	let the solution \(\vvu=\bigl(\vvH,\vvE\bigr)\) of \cref{eq:MaxwellFormulation} satisfy
  \begin{equation}\label{eq:SpatialRegularityInterpolation}
    \vvu \in \confuns{0}{[0,T], \pVast \cap \sobfuns{1+s}{\meshh}^6} \cap \confuns{1}{[0,T], \lebfuns{2}{\domain}^6},
  \end{equation}
  with \(s > 1/2\).
  For the approximations \(\ufh[]\) defined in \cref{eq:SemiDiscreteMaxwell} and \(\vcu_h\) defined in \cref{eq:InterpolateSemiDiscreteMaxwell} it holds 
  \begin{equation*}
    \normmueps{\vvu_h(t) - \vcu_h(t)} \leq C h^{\min\monosetc{s,k+1/2}}
  \end{equation*}
  with a constant \(C>0\) which is independent of \(h\).
  Here, \(k\) denotes the polynomial degree of the approximation space \cref{eq:BrokenPolySpace}.
\end{theorem}

The following corollary follows immediately from \cref{thm:SpatialConvergence,thm:InterpolationComparison}.
\begin{corollary}\label{cor:InterpolationSpatialConvergence}
  Under the assumptions of \cref{thm:InterpolationComparison}, it holds
  \begin{equation*}
    \normmueps{\vvu(t) - \vcu_h(t)} \leq C h^{\regcoeff},
  \end{equation*}
  a constant \(C>0\) which is independent of \(h\).
  Here, \(\regcoeff = \min\monosetc{s,k}\) with \(k\) denoting the polynomial degree of the approximation space \cref{eq:BrokenPolySpace}. 
\end{corollary}

\subsection*{Stability}\label{subsec:Stability}

We proceed by proving a discrete analogue to the stability bound \cref{eq:StabilityBound}.

The broken \(L^2\)-projection \(\OProj\), defined in \cref{eq:BrokenOrthProjection}, has the following \piecewise{} approximation properties, see, \eg{}, \cite[Sec.~18.4]{Ern.G.2021.FEAI}.
\begin{lemma}\label{lem:BrokenOrthApprox}
    For all \(\element \in\meshh\) and all \(v\in\sobfuns{1+s}{\element}\) with \(s\geq 0\) it holds
    \begin{subequations}
        \begin{align}
            \norm{v - \OProj v}_{\lebfuns{2}{\element}} 
                &\leq C h^{\regcoeff+1}_{\element} |v|_{\sobfuns{\regcoeff + 1}{\element}} ,
                \label{eq:ElementApprox}\\
            \norm{v-\OProj v}_{\lebfuns{2}{\face}}
                &\leq C h^{\regcoeff + 1/2}_{\element} |v|_{\sobfuns{\regcoeff+1}{\element}} ,
                \label{eq:FaceApprox}
            \end{align}
        \end{subequations}
    with constants \(C>0\) that are independent of \(h_\element\).
    Here, \(\regcoeff = \min\monosetc{s,k}\) with \(k\) denoting the polynomial degree of the approximation space \cref{eq:BrokenPolySpace}.
\end{lemma}

The following lemma shows an important relation between the Maxwell operators \cref{eq:ContinuousMaxwellOperators} and their discrete counterparts \cref{eq:DiscreteMaxwellOperators}.
\begin{lemma}\label{lem:Consistency}
    \begin{enumerate}
        \item\label{lem:Consistency:Consisten}
            The operators \(\DgMOpH, \DgMOpE\) are consistent, \ie{}, for \(\vvu = (\vvH,\vvE) \in \Vast\) it holds
            \begin{align*}
                \OProj\MOpH \vvH &= \DgMOpH\vvH,\\
                \OProj\MOpE \vvE &= \DgMOpE\vvE.
            \end{align*}
        \item\label{lem:Consistency:NonConsisten}
            The operator \(\pDgMOpH\) is non-consistent, \ie{}, for \(\vtH \in \pVastH\) it holds
            \begin{align*}
                \OProj\pMOpH\vtH &= \pDgMOpH\vtH - \LiftOp(\jmp{\vtH \times \nint}_\Fint).
            \end{align*}
    \end{enumerate}
\end{lemma}
The result \labelcref{lem:Consistency:Consisten} is stated in \cite[Sec.~2.3]{Hoc.S.2016}.
Thus, we only prove \labelcref{lem:Consistency:NonConsisten} involving the new domain special to the \inhomogeneous{} interface problem. 
\begin{proof}
    Let \(\TeFunH\in\Vapprox\).
    With integration by parts, we obtain
    \begin{equation*}
      \begin{aligned}
        \sprodeps{\pMOpH\vtH}{\TeFunH}
        &= 
        \sum_{\element\in\meshh} \sprod{\vtH}{\curl\TeFunH}_\element
        - \sum_{\face\in \Fboundary} \sprod{\vtH\times\nvF}{\TeFunH}_\face\\
        &\quad+ \sum_{\face\in\Finterior} 
              \sprod{\jmp{\vtH}_\face\times\nvF}{\avg{\TeFunH}^{\overline{\mu c}}}_\face\\
        &\quad- \sum_{\face\in\Finterior}
              \sprod{\avg{\vtH}^{\mu c}_\face}{\jmp{\TeFunH}_\face \times\nvF}_\face.
      \end{aligned}
    \end{equation*}
    Thus, with definitions \cref{eq:LiftOp,eq:DGpMaxwellOperatorH}, we see that
    \begin{equation*}
        \sprodeps{\pMOpH\vtH}{\TeFunH} 
        = \sprodeps{\pDgMOpH\vtH}{\TeFunH}
        - \sprodeps{\LiftOp\bigl(\jmp{\vtH}_\Fint \times \nint\bigr)}{\TeFunH}.
    \end{equation*}
    This proves the statement since \(\sprodeps{\pMOpH\vtH}{\TeFunH} = \sprodeps{\OProj\pMOpH\vtH}{\TeFunH}\) by definition of the projection \cref{eq:BrokenOrthProjection}.
\end{proof}

\begin{lemma}\label{lem:ProjectionUnderDGCurl}
    Let \(\vvu\in \pVasth\cap \sobfuns{1+s}{\meshh}^6\) for \(s\geq 0\).
    It holds
    \begin{equation*}
        \normmueps{\pDgMOp (\vvu-\OProj\vvu)} \leq C 
        h^{\regcoeff} \singlenorm{\vvu}_{\sobfuns{\regcoeff+1}{\meshh}^6}
    \end{equation*}
    with a constant \(C>0\) which is independent of \(h\) and \(\vvu\).
    Here, \(\regcoeff=\min\monosetc{s,k}\) and \(k\) denotes the polynomial degree of the approximation space \cref{eq:BrokenPolySpace}.
\end{lemma}
    A proof of this statement is included in \cite[eq.~(5.5)]{Hoc.S.2016}.
    We emphasize that all estimates there hold since they are local to every element \(\element\in\meshh\) and, thus, do not depend on the domain \(\DompMOp\).

The following Lemma is essential for the \wellposedness{} of the \semidiscrete{} problem.
A proof is provided in \cite[Lem.~2.2]{Hoc.S.2016}.
\begin{lemma}\label{lem:SkewAdjointMOp}
    The operator \(\DisMOp\) is skew-adjoint on \(\Vapprox^2\) with respect to the inner product \(\sprodmueps{\cdot}{\cdot}\), \ie{}, for \(\vvu_h,\vvv_h \in \Vapprox^2\) it holds
    \begin{equation*}
        \sprodmueps{\DisMOp\vvu_h}{\vvv_h} = - \sprodmueps{\vvu_h}{\DisMOp\vvv_h}.
    \end{equation*}
\end{lemma}
We infer from the skew-adjointness that \(\DisMOp\) is a generator of a unitary \scsemigroup{} on \(\Vapprox^2\).
Therefore, the \semidiscrete{} problem \cref{eq:SemiDiscreteMaxwell} has a unique solution \(\vvu_h(t) = \bigl(\vvH_h(t),\vvE_h(t)\bigr) \in \Vapprox^2\) given by the variation-of-constants formula
\begin{equation*}
    \vvu_h(t) = e^{t\DisMOp}\vvu_h^0 + \int_0^t e^{(t-s)\DisMOp}\bigl(\jvolh(s) + \jsurfh(s)\bigr)\mathrm{d}s,
\end{equation*}
with \(\vvu_h^0 = \bigl(\vvH_h^0,\vvE_h^0\bigr)\), \(\jvolh = (0,-\Jvolh)\) and \(\jsurfh = (0,-\Jsurfh)\).

The following stability bound holds true for the \semidiscrete{} problem and is a discrete analogue to \cref{eq:StabilityBound}.
\begin{theorem}\label{thm:DiscreteStability}
    Under \Cref{ass:interface_aligned} and the assumptions of \cref{thm:SpatialConvergence}, the numerical solution \(\vvu_h = \bigl(\vvH_h,\vvE_H\bigr)\) of \cref{eq:SemiDiscreteMaxwell} is stable, \ie{}, for \(t\in [0,T]\) it holds
    \begin{align*}
        \normmueps{\vvu_h(t)}
          \lesssim&\ \normmueps{\vvu^0} 
          + \norm{\Jsurf(0)}_{\lebfuns{2}{\Fint}^3}
          + \norm{\Jsurf(t)}_{\lebfuns{2}{\Fint}^3}\\
          &+ \int_0^t \norm{\Jvol(s)}_{\lebfuns{2}{\domain}^3}\ \mathrm{d}s
          + \int_0^t \norm{\dt\Jsurf(s)}_{\TraceSpaceParallel}\ \mathrm{d}s\\
          &+ \int_0^t \norm{\Jsurf(s)}_{\Hparallelspace}\ \mathrm{d}s,
      \end{align*}
      with a constant which is independent of \(h\) and \(\vvu\).
\end{theorem}
\begin{proof}

    We proceed similar to the proof of \cref{thm:UniqueExistence} and introduce a shifted \semidiscrete{} solution \(\vtu_h(t) = \vvu_h(t) - \OProj \jextH(t)\), where \(\jextH(t) = \bigl(\JextH(t),0\bigr)\) denotes the extension of \cref{cor:NewExtensionCoro}.
    Thus, the shifted solution solves
    \begin{align*}
        \dt \vtu_h(t) 
        &= \pDgMOp \vvu_h(t) + \jvolh(t) + \jsurfh(t) - \OProj \dt \jextH(t)\\
        &= \pDgMOp \vtu_h(t) + \pDgMOp \OProj \jextH(t) + \jvolh(t) + \jsurfh(t) - \OProj \dt \jextH(t)
    \end{align*}
    By \cref{lem:Consistency} it holds
    \begin{equation*}
        \OProj \pMOp \jextH(t) 
        = \pDgMOp \jextH(t) -\Bigl(0, \LiftOp\bigl(\jmp{\JextH(t)\times \nint}_{\Fint}\bigr)\Bigr)
        = \pDgMOp \jextH(t) + \jsurfh(t).
    \end{equation*}
    Therefore, we obtain \(\dt \vtu_h(t) = \pDgMOp \vtu_h(t) + \vtr_h(t)\), with
    \begin{equation*}
        \vtr_h(t) = 
        - \pDgMOp \bigl(I-\OProj\bigr)\jextH(t)
        + \OProj\bigl(\jvol(t) + \pMOp\jextH(t) - \dt \jextH(t)\bigr).
    \end{equation*}
    We emphasize that \(\vtr_h(t) \in \Vapprox^2\) and thus write the solution by means of the 
    variations-of-constants formula as
    \begin{equation*}
        \vtu_h(t) = e^{t\DisMOp}\bigl(\vvu_h^0 - \OProj \jextH(0)\bigr)
        + \int_0^t e^{(t-s)\DisMOp} \vtr_h(s)\ \mathrm{d}s.
    \end{equation*}
    Furthermore, since \(\DisMOp\) generates a unitary \scsemigroup{} on \(\Vapprox^2\), we obtain that
    \begin{equation*}
        \normmueps{\vtu_h(t)} \leq \normmueps{\vvu^0} + \normmueps{\jextH(0)} + \int_0^t \normmueps{\vtr_h(s)}\ \mathrm{d}s 
    \end{equation*}
    It remains to bound \(\normmueps{\vtr_h(s)}\).
    By \cref{cor:NewExtensionCoro}, it holds \(\JextH(s) \in \psobfuns{1}^3\) and thus, by \cref{lem:ProjectionUnderDGCurl} with \(s=0\), we conclude that
    \begin{equation*}
        \normmueps{\pDgMOp\bigl(I-\OProj\bigr)\jextH(s)} 
        \leq C \singlenorm{\JextH(s)}_{\sobfuns{1}{\meshh}^3} 
        = C \singlenorm{\JextH(s)}_{\psobfuns{1}^3}.
    \end{equation*} 
    The right-hand side can be further estimated with \cref{cor:NewExtensionCoro}, and we obtain
    \begin{equation*}
        \normmueps{\pDgMOp\bigl(I-\OProj\bigr)\jextH(s)}
        \leq C \norm{\Jsurf[](s)}_{\Hparallelspace}.
    \end{equation*}
    The remaining parts of \(\vtr_h(s)\) can be bounded analogously by \cref{cor:NewExtensionCoro}.
    This proves the claim similar to \cref{thm:UniqueExistence}.
\end{proof}

\subsection*{Error analysis}\label{subsec:ErrorAnalysis}

We proceed by proving the main error bounds of this section.

\begin{proof}[Proof of \cref{thm:SpatialConvergence}]\label{prf:SpatialConvergence}
    We define the error \(\vve(t) = \vvu(t) - \vvu_h(t)\), where \(\vvu(t)\) denotes the solution of \cref{eq:MaxwellFormulation} and \(\vvu_h(t)\) denotes the \semidiscrete{} solution of \cref{eq:SemiDiscreteMaxwell}.
    We split the error into \(\vve(t) = \vve_\Pi(t) - \vve_h(t)\) with
    \begin{subequations}\label{eq:SpatialErrors}
      \begin{align}
        \vve_\Pi(t) &= \vvu(t) - \OProj\vvu(t),\\
        \vve_h(t) &= \vvu_h(t) - \OProj\vvu(t).
      \end{align}
    \end{subequations}
    Thus, \(\vve_\Pi(t)\) denotes the best approximation error and \(\vve_h(t)\) the dG-error.
    By \cref{eq:MaxwellFormulation} and  \cref{lem:Consistency}, it holds
    \begin{equation}\label{eq:ProjectionEquation}
        \dt\OProj \vvu(t) = \OProj\bigl(\pMOp\vvu(t) +\jvol\bigr) 
        = \pDgMOp\vvu(t) + \jvolh(t) + \jsurfh(t).
    \end{equation}
    Since \(\ufh[](t)\) solves \cref{eq:SemiDiscreteMaxwell}, \ie{},
    \begin{equation*}
      \dt\ufh[](t) = \pDgMOp\ufh[](t) + \jvolh[](t) + \jsurfh[](t),\ t\in [0,T],
        \quad \ufh[](0) = \OProj \uf[0] ,
    \end{equation*}
    we see that the dG-error solves the initial value problem
    \begin{equation}\label{eq:dGErrorIVP}
        \dt \vve_h(t) = \DisMOp \vve_h(t) + \vvd_\pi(t),\ t\in [0,T],
        \quad \vve_h(0) = 0,
    \end{equation}
    with the defect \(\vvd_\pi(t) = -\pDgMOp \vve_\Pi(t)\).
    We can write the solution of \cref{eq:dGErrorIVP} with the variation-of-constants formula and obtain
    \begin{equation*}
        \vve_h(t) = \int_0^t e^{(t-s)\DisMOp} \vvd_\pi(s)\ \mathrm{d}s.
    \end{equation*}
    Since \(\DisMOp\) is the generator of a unitary \scsemigroup{} on \(\Vapprox^2\), we conclude with \cref{lem:ProjectionUnderDGCurl} that
    \begin{equation*}
        \normmueps{\vve_h(t)} 
        \leq \int_0^t \normmueps{\vvd_\pi(s)}\ \mathrm{d}s
        \leq C h^{\regcoeff} \int_0^t \singlenorm{\vvu(s)}_{\sobfuns{\regcoeff+1}{\meshh}^6}\ \mathrm{d}s
    \end{equation*}
    with a constant independent of \(h\) and \(\vvu\).
    Together with the approximation properties of \cref{lem:BrokenOrthApprox}, we obtain
    \begin{align*}
        \normmueps{\vve(t)} &\leq \normmueps{\vve_\Pi(t)} + \normmueps{\vve_h(t)}\\
        &\leq \widetilde{C} h^{\regcoeff + 1} \singlenorm{\vvu(t)}_{\sobfuns{\regcoeff + 1}{\meshh}^6} 
            + C h^\regcoeff \int_0^t \singlenorm{\vvu(s)}_{\sobfuns{\regcoeff+1}{\meshh}^6}\ \mathrm{d}s ,
    \end{align*}
    which proves the claim.
\end{proof}

\begin{remark}
    Note that the stability result of \cref{thm:DiscreteStability} is not used in the proof of \cref{thm:SpatialConvergence}.
    The reason for that is the fact that the lifted surface current appears in \cref{eq:ProjectionEquation} due to \cref{lem:Consistency}~\labelcref{lem:Consistency:NonConsisten}.
    Thus, there is no contribution of the surface current in the defect.
    This, on the other hand, assumes that the lifted surface current can be calculated exactly which is not feasible in practice, \cf{} \cref{subsec:NodalInterpolation}.

    The following section deals with errors introduced due to nodal interpolation on the interface.
\end{remark}

\subsection*{Interpolation error}

The following local estimates for the nodal interpolation hold,
compare for example \cite[Thm.~11.13]{Ern.G.2021.FEAI}.
\begin{lemma}\label{lem:NodalElementApprox}
    For all \(\element \in\meshh\) and all \(v\in\sobfuns{1+s}{\element}\) with \(s>1/2\) it holds
    \begin{equation}\label{eq:NodalElementApprox}
        \norm{v - \NInterp v}_{\lebfuns{2}{\element}} 
            \leq C h^{\regcoeff + 1}_{\element} |v|_{\sobfuns{\regcoeff+1}{\element}}
    \end{equation}
    with a constant \(C>0\) which is independent of \(h_\element\).
    Here, \(\regcoeff = \min\monosetc{s,k}\) and \(k\) denotes the polynomial degree of the approximation space \cref{eq:BrokenPolySpace}.
\end{lemma}
In order to obtain approximation properties for the local interpolation operator \(\FInterp\) on the \submesh{}, we need to ensure that \(\Finterface\) does not degenerate, \ie{} that the \submesh{} is again shape regular.
Recall the following notation.
For \(\face\in\Finterface\), we denote with \(h_\face\) the largest diameter of \(\face\) and with \(\rho_\face\) the diameter of the largest inscribing ball of \(\face\).
It is clear from the definition that \(h_\face \leq h_\element\).
Furthermore, \cite[Thm.~10, (10)]{Gan.L.1996} shows that \(\rho_\element\leq \rho_\face\), \ie{}, the diameter of the largest inscribing ball of \(\element\) is always less or equal to the diameter of the largest inscribing ball of \(\face\).
Therefore,
\begin{equation*}
  h_\element \leq \sigma \rho_\element
  \implies h_\face \leq \sigma \rho_\face,
\end{equation*}
\ie{}, the \submesh{} \(\Finterface\) inherits the shape regularity from \(\meshh\).
We infer again from \cite[Thm.~11.13]{Ern.G.2021.FEAI} the following approximation properties.
\begin{lemma}\label{lem:NodalFaceApprox}
    For all \(\face \in \Finterface\) and all \(w \in\sobfuns{1 + s}{\face}\) with \(s>0\) it holds
    \begin{equation}\label{eq:NodalFaceApprox}
        \norm{w-\FInterp w}_{\lebfuns{2}{\face}}
            \leq \CFInterp h^{\regcoeff + 1}_{\face} |w|_{\sobfuns{\regcoeff + 1}{\face}}
    \end{equation}
    with a constant \(\CFInterp>0\) which is independent of \(h_\face\).
    Here, \(\regcoeff = \min\monosetc{s,k}\) and \(k\) denotes the polynomial degree of the approximation space \cref{eq:BrokenPolySpace}.
\end{lemma}

Similar to \cref{lem:ProjectionUnderDGCurl}, we obtain an approximation result under the discrete lift operator.
\begin{lemma}\label{lem:LiftOpInterpolationEstimate}
    Let \(\vvV\in\sobfuns{1+s}{\Finterface}^3\) with \(s>0\).
    Under \Cref{ass:interface_aligned}, it holds
    \begin{equation*}
        \normeps{\LiftOp\bigl(\vvV - \FInterp \vvV\bigr)} \leq C h^{\regcoeff+1/2} \singlenorm{\vvV}_{\sobfuns{\regcoeff + 1}{\Finterface}^3}
    \end{equation*}
    with a constant \(C > 0\) which is independent of \(h\) and \(\vvV\).
    Here, \(\regcoeff = \min\monosetc{s,k}\) and the polynomial degree of approximation space \cref{eq:BrokenPolySpace} is denoted with \(k\).
\end{lemma}
\begin{proof}
    Let \(\TeFunH\in\Vapprox\).
    By the definition of the discrete lift operator \cref{eq:LiftOp} and the Cauchy-Schwarz inequality it holds
    \begin{equation}\label{eq:LiftEstimateCS}
        \singlenorm{\sprodeps{\LiftOp\bigl(\vvV-\FInterp\vvV\bigr)}{\TeFunH}}
        \leq 
        \begin{multlined}[t]
            \biggl(\sum_{\face\in\Finterface} \omega_\face^{-1} \norm{\vvV-\FInterp\vvV}_{\lebfuns{2}{\face}^3}^2\biggr)^{1/2}\\
            \cdot \biggl(\sum_{\face\in\Finterface} \omega_\face \norm{\avg{\TeFunH}^{\varepsilon\sol}}_{\lebfuns{2}{\face}^3}^2\biggr)^{1/2}
        \end{multlined}
    \end{equation}
    with the weight \(\omega_\face = \min\monosetc{h_{\elementKFl},h_{\elementKFr}}\).
    By \cref{lem:NodalFaceApprox}, we obtain the estimate
    \begin{equation}\label{eq:LiftEstimateV}
        \omega_\face^{-1} \norm{\vvV-\FInterp\vvV}_{\lebfuns{2}{\face}^3}^2 
        \leq C^2 \omega_\face^{-1} h_\face^{2\regcoeff + 2} \singlenorm{\vvV}_{\sobfuns{\regcoeff + 1}{\face}^3}^2
    \end{equation}
    Since by definition \(h_\face \leq \max\monosetc{h_{\elementKFl}, h_{\elementKFr}}\), we obtain with the shape regularity that
    \begin{equation}\label{eq:ShapeRegularityUsed}
      h_\face \leq \max\monosetc{h_{\elementKFl}, h_{\elementKFr}} 
      \leq \sigma \rho_F 
      \leq \sigma \min\monosetc{h_{\elementKFl}, h_{\elementKFr}}
      \leq \sigma \omega_\face.
    \end{equation}
    Therefore, we lose one \(h_F\) in \cref{eq:LiftEstimateV} due to the weight \(\omega_F\) and end up with the estimate
    \begin{equation*}
      \omega_\face^{-1} \norm{\vvV-\FInterp\vvV}_{\lebfuns{2}{\face}^3}^2
      \leq C^2 \sigma h_\face^{2\regcoeff + 1} \singlenorm{\vvV}_{\sobfuns{\regcoeff + 1}{\face}^3}^2
    \end{equation*} 
    With the discrete trace inequality \cite[Lem.~12.8]{Ern.G.2021.FEAI}, we further estimate
    \begin{equation}\label{eq:LiftEstimatePhi}
        \norm{\avg{\TeFunH}^{\varepsilon\sol}}_{\lebfuns{2}{\face}^3}^2
        \leq 2 C^2 \solmax^2 \Msusmax  \bigl(
            h_{\elementKFl}^{-1} \normeps{\restrict{\TeFunH}{\elementKFl}}[\elementKFl]^2
            + h_{\elementKFr}^{-1} \normeps{\restrict{\TeFunH}{\elementKFr}}[\elementKFr]^2
        \bigr)
    \end{equation}
    with a constant \(C>0\) which is independent of \(\face,\elementKFl\) and \(\elementKFr\),
    but depends on the polynomial degree \(k\).
    
    Multiplication of \cref{eq:LiftEstimatePhi} with \(\omega_\face\) proves the statement together with \cref{eq:LiftEstimateV} and \cref{eq:LiftEstimateCS}.
\end{proof}

With \cref{lem:LiftOpInterpolationEstimate}, we have all ingredients to prove the second main result of this section.
\begin{proof}[Proof of \cref{thm:InterpolationComparison}]\label{prf:InterpolationComparison}
    Since \(\vvH \in \confuns{0}{[0,T], \pVastH\cap \sobfuns{1+s}{\meshh}^3}\) with \(s > 1/2\), we conclude by \cite[Thm.~3.10]{Ern.G.2021.FEAI} that 
    \begin{equation}
        \Jsurf = \jmp{\vvH\times \nint}_\Fint \in  \confuns{0}{[0,T], \sobfuns{1 + \kappa}{\Finterface}^3}
    \end{equation}
    with \(\kappa = s - 1/2 > 0\).

    We write \cref{eq:InterpolateSemiDiscreteMaxwell} in vector form, \ie{}, \(\vcu_h(t) = \bigl(\vcH_h(t),\vcE_h(t)\bigr)\) such that
    \begin{equation}\label{eq:InterpolateSemiDiscreteMaxwellVector}
        \dt \vcu_h(t) = \pDgMOp \vcu_h(t) + \jvolh(t) + \jsurfhInterp(t),\ \text{for } t\in [0,T],\hspace*{1em} \vcu_h(0) = \vvu_h^0,
    \end{equation}
    with \(\jvolh = \bigl(0, -\Jvolh\bigr)\), \(\jsurfhInterp = \bigl(0,-\JsurfhInterp\bigr)\) and \(\vcu_h^0 = \bigl(\vcH_h^0,\vcE_h^0\bigr)\).

    Writing \(\vce_h(t) = \vvu_h(t) - \vcu_h(t)\) and subtracting \cref{eq:SemiDiscreteMaxwell,eq:InterpolateSemiDiscreteMaxwellVector}, we obtain
    \begin{equation}\label{eq:DefectEqComparison}
        \dt \vce_h(t) = \pDgMOp \vce_h(t) + \vcd_h(t),\ \text{for } t\in [0,T],\hspace*{1em} \vce_h(0) = 0,
    \end{equation}
    with a defect \(\vcd_h(t) = \jsurfh(t) - \jsurfhInterp(t)\).
    
    We can write the solution of \cref{eq:DefectEqComparison} with the variations-of-constants formula and obtain with \Cref{lem:LiftOpInterpolationEstimate} the estimate
    \begin{equation*}
        \normmueps{\vce_h(t)} 
        \leq C h^{\min\monosetc{\kappa,k} + 1/2} \int_0^t\singlenorm{\Jsurf(s)}_{\sobfuns{1+\kappa}{\Finterface}^3}\ \mathrm{d}s.
    \end{equation*}
    This proves the claim since \(\kappa+1/2 = s\).
\end{proof}

\section{Full discretization}\label{sec:FullDiscretization}

In time, we discretize \cref{eq:SemiDiscreteMaxwell}  with the explicit \leapfrog{} scheme with \stepsize{} \(\tau > 0\)  and set \(t_n = n\tau\) for \(n\in\bbN\).
The fully discrete scheme reads
\begin{subequations}\label{eq:LeapfrogScheme}
    \begin{align}
        \Hfh[n+1/2] - \Hfh[n] &= -\frac{\tau}{2}\DgMOpE \Efh[n],\\
        \Efh[n+1] - \Efh[n] &= \tau\pDgMOpH \Hfh[n+1/2]
        - \frac{\tau}{2}(\Jvolh[n] + \Jvolh[n+1])
        - \frac{\tau}{2}(\Jsurfh[n] + \Jsurfh[n+1]),\\
        \Hfh[n+1] - \Hfh[n+1/2] &= -\frac{\tau}{2} \DgMOpE \Efh[n+1],
    \end{align}
\end{subequations}
for \(n\geq 0\) and \(\Hfh[0] = \OProj \Hf[0]\), \(\Efh[0] = \OProj \Ef[0]\).
It is well-known that the leapfrog scheme is stable if for some \(\thetaCFL \in (0,1)\),
the CFL condition 
\begin{equation} \label{def:CFLCondition}
	\tau < \tauCFL = \frac{2\thetaCFL}{\normeps{\pDgMOpH\DgMOpE}} \lesssim \thetaCFL \min_{\element\in\meshh} h_\element 
\end{equation}
is satisfied.
Here, $\normeps{\cdot}$ denotes the induced operator norm, cf.\ \eqref{eq:weightedL2}.
For more details on the constant within the CFL condition, we refer to \cite[eq.~(2.35)]{Hoc.S.2016}.

The main result of this section is the following bound on the full discretization error.
\begin{theorem}\label{thm:FullDiscretization}
	Let \Cref{ass:interface_aligned} hold and further 
	let the solution \(\vvu=\bigl(\vvH,\vvE\bigr)\) of \cref{eq:MaxwellFormulation} satisfy
    \begin{equation}\label{eq:FullRegularity}
      \vvu \in \confuns{0}{[0,T], \pVast \cap \sobfuns{1+s}{\meshh}^6} \cap \confuns{3}{[0,T], \lebfuns{2}{\domain}^6},
    \end{equation}
    with \(s\geq 0\) and assume that the CFL condition \cref{def:CFLCondition} holds.
    Then,  the approximations $\ufh[n] = \bigl(\Hfh[n],\Efh[n]\bigr)$ defined in \eqref{eq:LeapfrogScheme}
    with approximation space \eqref{eq:BrokenandVapprox} satisfies 
    \begin{equation*}
      \normmueps{\uf(t_n) - \ufh[n]} \leq C(h^{\regcoeff} + \tau^2),
      \qquad 0\leq t_n \leq T.
    \end{equation*}
    Here, \(\regcoeff=\min\monosetc{s,k}\) and  \(C>0\) is a constant which is independent of \(h\) and \(\tau\).
\end{theorem}

\begin{remark}\label{rmk:OnTheCFLCondition}
   Note that the interface condition does not induce an additional \stepsize{} restriction, since the CFL condition \eqref{def:CFLCondition} coincides with that for problems on the full domain $\domain$.
\end{remark}

\begin{remark}\label{rmk:CompatibiltyCondition}
	We note that in order to prove the regularity assumptions on \(\uf\) in \cref{thm:FullDiscretization} certain compatibility conditions have to be satisfied at the initial time.
    Assuming that the solution is sufficiently smooth, we obtain from \cref{eq:InterfaceCond:Tangential}
    \begin{align*}
        \dt\Jsurf 
            &= \jmp{\dt\Hf\times\nint}_\Fint 
                = -\jmp{\Msus[{-1}]\curl\Ef\times\nint}_\Fint,\\
        \dt^2\Jsurf
            &= -\jmp{\Msus[{-1}]\curl\dt\Ef\times\nint}_\Fint
                = -\jmp{\Msus[{-1}]\curl\Eperm[{-1}]\curl\Hf\times\nint}_\Fint.
    \end{align*}
    The reader should refer to \cite[Thm.~2.4-2.6]{Dor.Z.2023} for a thorough treatment.
\end{remark}

Our analysis is inspired by \cite{Hoc.S.2016}, where the locally implicit method for linear Maxwell equations is considered.  
With the discrete Maxwell operators from \eqref{eq:DiscreteMaxwellOperators} and  $\ufh[n] = \bigl(\Hfh[n],\Efh[n]\bigr)$, we write \cref{eq:LeapfrogScheme} in the following one-step formulation
\begin{subequations}  \label{eq:LeapfrogAll}
\begin{equation}\label{eq:LeapfrogSchemeOneStep}
	\Rhl\ufh[n+1] = \Rhr\ufh[n] 
	+ \frac{\tau}{2}\bigl(\jvolh[n+1] + \jvolh[n]\bigr) 
	+ \frac{\tau}{2}\bigl(\jsurfh[n+1] + \jsurfh[n]\bigr)
\end{equation}
for \(n\geq 0\) with operators
\begin{equation}\label{def:PropagationOp}
	\Rhlr \defpnt \Vapprox^2 \to \Vapprox^2,
	\quad \Rhlr = I \pm \frac{\tau}{2}\pDgMOp - \frac{\tau^2}{4}\DgLFPert
\end{equation}
and perturbation operator
\begin{equation}\label{def:PerturbationOp}
	\DgLFPert \defpnt \Vapprox^2 \to \Vapprox^2,
	\quad \DgLFPert = \begin{pmatrix}
		0 & 0\\
		0 & \pDgMOpH\DgMOpE
	\end{pmatrix}.
\end{equation}
\end{subequations}
For \(\DgLFPert =0\), the scheme \eqref{eq:LeapfrogAll} is equivalent to the \CrankNicolson{} method and thus, one can interpret the leapfrog scheme as perturbation of it.
We further use the operator \(\Rh \defpnt \Vapprox^2\to\Vapprox^2\), defined as \(\Rh = \Rhl^{-1}\Rhr\).

Note that for the special choice of $\bm{\mathcal{C}_H}^i=0$ and $\bm{\mathcal{C}_E}^i=0$ in \cite[eq.~(2.34)]{Hoc.S.2016}, one obtains the leapfrog scheme on the whole spatial domain. Thus, we can use bounds on the operators in \eqref{eq:LeapfrogAll} from that work.

\subsection*{Stability}
The following theorem provides stability for the fully discrete scheme and is a discrete analogue of the bound provided in \cref{thm:UniqueExistence} for the exact solution.

\begin{theorem}\label{thm:FullyDiscreteStability}
    Assume that the CFL condition \cref{def:CFLCondition}, \Cref{ass:interface_aligned}, and the assumptions of \cref{thm:FullDiscretization} are satisfied.
    Then, the fully discrete scheme \eqref{eq:LeapfrogScheme} is stable, \ie{}, for all \(n\geq 0\) it holds
    \begin{align*}
        \normmueps{\ufh[n]}
          &\lesssim \normmueps{\uf[0]} 
          + \norm{\Jsurf(t_0)}_{\lebfuns{2}{\Fint}^3}
          + \norm{\Jsurf(t_n)}_{\lebfuns{2}{\Fint}^3}\\
          &\quad + \frac{\tau}{2}\sum_{\ell=0}^{n-1} \norm{\Jvol(t_{\ell+1}) + \Jvol(t_\ell)}_{\lebfuns{2}{\domain}^3}\\
          &\quad + \int_{t_0}^{t_n} \norm{\dt\Jsurf(s)}_{\TraceSpaceParallel}\ds\\
          &\quad + \frac{\tau}{2}\sum_{\ell=0}^{n-1}
          \begin{aligned}[t]
            &\Big\|\Jsurf(t_{{\ell+1}}) + \Jsurf(t_{\ell})\Big\|_{\Hparallelspace},
          \end{aligned}
      \end{align*}
      with a constant which is independent of \(h\), \(\tau\) and \(\uf[]\).
\end{theorem}
\begin{proof}
	The proof relies on the same arguments as in \cref{thm:UniqueExistence,thm:DiscreteStability}. Hence, we introduce the shifted field
    \begin{equation}\label{eq:FullDiscretization:FullyShifted}
        \vcu_h^n = \ufh[n] - \OProj \jextH[n],
    \end{equation}
    where \(\jextH[n] = \bigl(\JextH[n],0\bigr)\).
    The shifted variables satisfy the recursion
    \begin{equation}\label{eq:FullDiscretization:FirstDerivation}
        \begin{aligned}
            \Rhl\vcu_h^{n+1}
            &= \Rhr\vcu_h^n
                + \frac{\tau}{2}\bigl(\jvolh[n+1] + \jvolh[n]\bigr) 
                + \frac{\tau}{2}\bigl(\jsurfh[n+1] + \jsurfh[n]\bigr)\\
            &\qquad+ \Rhr \OProj\jextH[n] - \Rhl \OProj \jextH[n+1] .
        \end{aligned}
    \end{equation}

    Next, we study the action of \(\Rhlr\) on \(\OProj\jextH[\ell]\) for \(\ell \geq 0\).
    Since the second component of \(\jextH[\ell]\) is zero, it holds \(\DgLFPert \OProj\jextH[\ell] = 0\).
    Thus,  \cref{def:PropagationOp} yields
    \begin{align*}
        \Rhlr \OProj \jextH[\ell] 
       & = \OProj \jextH[\ell] \pm \frac{\tau}{2}\pDgMOp \OProj \jextH[\ell]\\
       & = \OProj \jextH[\ell] \pm \frac{\tau}{2}\pDgMOp \bigl(\OProj - I\bigr)\jextH[\ell] \pm \frac{\tau}{2}\OProj\pMOp \jextH[\ell] \mp \frac{\tau}{2}\jsurfh[\ell].
    \end{align*}
    Here, the second identity follows from \cref{cor:NewExtensionCoro,lem:Consistency}.
    Inserting this into \cref{eq:FullDiscretization:FirstDerivation} leads to
    \begin{equation}\label{eq:FullDiscretization:FinalDerivation}
        \begin{aligned}
            \Rhl\vcu_h^{n+1}
            &= \Rhr\vcu_h^n +\vcr_h^n
        \end{aligned}
    \end{equation}
    with the remaining terms
    \begin{align*}
        \vcr_h^n
        &= \frac{\tau}{2}\bigl(\jvolh[n+1] + \jvolh[n]\bigr)\\
        &\qquad-\OProj (\jextH[n+1] - \jextH[n])\\
        &\qquad+\frac{\tau}{2}\OProj\pMOp (\jextH[n+1]+\jextH[n])
        +\frac{\tau}{2}\pDgMOp\bigl(\OProj - I\bigr)(\jextH[n+1]+\jextH[n]).
    \end{align*}
    Solving the recursion \cref{eq:FullDiscretization:FinalDerivation}, we obtain
    \begin{equation*}
        \vcu_h^{n} = \Rh^{n}\vcu_h^0 + \sum_{\ell = 0}^{n-1} \Rh^{n-1-\ell}\Rhl^{-1}\vcr_h^\ell.
    \end{equation*}

    By \cite[Lem.~11.14]{Dor.H.K.Et.2023.WPMA} we have
    \begin{subequations} \label{eq:BoundsOnall}
   	 \begin{equation}\label{eq:BoundOnRMinusInverse}
		\normmueps{\Rhl^{-1}} \leq \sqrt{1+\theta^2+\theta^4}
	\end{equation}
	and by \cite[Lem.~4.2]{Hoc.S.2016} it holds
	\begin{equation}\label{eq:BoundOnPowersOfR}
		\normmueps{\Rh^{m}} \leq \Cstab =(1-\thetaCFL^2)^{-1/2}, \quad m=0,1,\ldots .
	\end{equation}
    \end{subequations}
    Together with \cref{eq:FullDiscretization:FullyShifted} we infer
    \begin{equation*}
        \normmueps{\ufh[n]} 
        \leq \Cstab \Bigl(\normmueps{\ufh[0]}
        +  \normmueps{\OProj \jextH[0]}
        + \sqrt{3} \sum_{\ell=0}^{n-1} \normmueps{\vcr_h^\ell}\Bigr)
        + \normmueps{\OProj \jextH[n]}
        .
    \end{equation*}
    The remainders $\normmueps{\vcr_h^\ell}$ are bounded in the same way as in the proof of \cref{thm:DiscreteStability}.
\end{proof}

\subsection*{Error analysis}

As usual, we split the full discretization error into
\begin{equation}\label{eq:FullDisErrorSplit}
    \uf[n] - \ufh[n] = \uf[n] - \OProj \uf[n] + \OProj \uf[n]- \ufh[n] = \vve_\Pi^n + \vve_h^n.
\end{equation}
Here, \(\vve_\Pi^n = \uf[n] - \OProj \uf[n]\) is the best approximation error  and \(\vve_h^n = \OProj \uf[n]- \ufh[n]\) is the dG-\leapfrog-error  at time \(t_n\).
Since the best approximation error is covered by projection results, \cf{} \cref{lem:BrokenOrthApprox}, we determine the defect \(\vvd^n\) by inserting the projected exact solution into the scheme \cref{eq:LeapfrogSchemeOneStep}, \ie, 
\begin{equation}\label{eq:ProjectExactSolutionInMethod}
    \Rhl\OProj\uf[n+1] = \Rhr\OProj\uf[n] 
        + \frac{\tau}{2}\bigl(\jvolh[n+1] + \jvolh[n]\bigr) 
        + \frac{\tau}{2}\bigl(\jsurfh[n+1] + \jsurfh[n]\bigr)
        - \vvd^n.
\end{equation}
This allows us to infer the error recursion for the dG-leapfrog scheme.

\begin{lemma}\label{lem:Defect}
	Let the assumptions of \cref{thm:FullDiscretization} be satisfied. 
    Then, the dG-leapfrog-error \(\vve_h^n\) defined in \cref{eq:FullDisErrorSplit} satisfies the error recursion
    \begin{subequations} \label{eq:dg-lf-error-all}
    \begin{equation}  \label{eq:dg-lf-error-a}
        \Rhl\vve_h^{n+1} = \Rhr\vve_h^{n} + \vvd^n, 
        \qquad 
        \vvd^n = \vvd_{\Pi}^n + \vvdelta^n +  \bigl(\Rhl - \Rhr\bigr)\vvd_h^n
    \end{equation}
    with
        \begin{align} \label{eq:dg-lf-error-b}
            \vvd_\Pi^n &= 
                -\frac{\tau}{2} \pDgMOp\bigl(I-\OProj\bigr)(\uf[n+1] + \uf[n])
                -\frac{\tau^2}{4} \DgLFPert\bigl(I-\OProj\bigr)(\uf[n+1] - \uf[n]),\\
                \label{eq:dg-lf-error-c}
            \trapedefect[n] &= 
                \tau^2\OProj\int_{t_n}^{t_{n+1}}\frac{(s-t_n)(t_{n+1}-s)}{2\tau^2}\dt^3 \uf(s)\ds,\\
                \label{eq:dg-lf-error-d}
            \vvd_h^n &= 
                -\frac{\tau}{4}\begin{pmatrix}\OProj \bigl(\dt\Hf[n+1] - \dt\Hf[n]\bigr)\\0\end{pmatrix}.
        \end{align}
    \end{subequations}
\end{lemma}
\begin{proof}
    With the fundamental theorem of calculus and the error estimate of the trapezoidal rule, we obtain
    \begin{align*}
        \uf[n+1] - \uf[n] 
        &= \frac{\tau}{2}\pMOp\bigl(\uf[n+1] + \uf[n]\bigr)
            +\frac{\tau}{2}\bigl(\jvol[n+1] + \jvol[n]\bigr) \\
            &\qquad -\tau^2\int_{t_n}^{t_{n+1}}\frac{(s-t_n)(t_{n+1}-s)}{2\tau^2}\dt^3 \uf(s)\ds.
    \end{align*}
    Projecting both sides onto $\Vapprox^2$ and using \cref{lem:Consistency}, we infer that
    \begin{equation*}
        \OProj\uf[n+1] - \frac{\tau}{2}\pDgMOp\uf[n+1] 
        = \OProj\uf[n] + \frac{\tau}{2}\pDgMOp\uf[n] 
            + \frac{\tau}{2}\bigl(\jvolh[n+1] + \jvolh[n]\bigr)
            + \frac{\tau}{2}\bigl(\jsurfh[n+1] + \jsurfh[n]\bigr)
            - \trapedefect[n].
    \end{equation*}
    Writing $\pDgMOp = \pDgMOp \OProj + \pDgMOp(I-\OProj)$ and comparing with \cref{eq:ProjectExactSolutionInMethod} gives
    \begin{equation*}
        \vvd^n 
        = \trapedefect[n]
        - \frac{\tau}{2}\pDgMOp\Bigl(I-\OProj\Bigr)\bigl(\uf[n+1] + \uf[n]\bigr)
        + \frac{\tau^2}{4}\DgLFPert\OProj\bigl(\uf[n+1] - \uf[n]\bigr).
    \end{equation*}
    Moreover, as in \cite[eq.~(5.21)]{Hoc.S.2016}, we can write
    \begin{align*}
        \frac{\tau^2}{4}\DgLFPert\OProj\bigl(\uf[n+1] - \uf[n]\bigr) 
        = &   -\frac{\tau^2}{4}\DgLFPert\Bigl(I-\OProj\Bigr)\bigl(\uf[n+1] - \uf[n]\bigr)\\
            &- \frac{\tau}{4}\Bigl(\Rhl - \Rhr\Bigr) 
                \begin{pmatrix}\OProj \bigl(\dt\Hf[n+1] - \dt\Hf[n]\bigr)\\0\end{pmatrix}.
    \end{align*}
    This proves the claim.
\end{proof}

With this representation of the defect, we are able to prove the main result of this section.
\begin{proof}[Proof of \cref{thm:FullDiscretization}.]
    The proof proceeds in three steps and makes use of a generic constant $C$ which is independent of 
 	\(h\) and \(\tau\).   
    First, we bound the projection error \(\vve_\Pi^n\).
    Second, we solve the error recursion for the dG-\leapfrog-error and estimate the defects separately.
    The claim then follows in a third step via the application of the triangle inequality.  
    
    \Cref{lem:BrokenOrthApprox} directly implies
    \begin{equation}\label{eq:FullProjectionError}
        \normmueps{\vve_\Pi^n} 
            = \normmueps{\uf[n] - \OProj \uf[n]}
            \leq C h^{\regcoeff+1} |\uf[n]|_{\sobfuns{\regcoeff + 1}{\meshh}^6}.
    \end{equation}

    With \cref{lem:Defect}, the error recursion \eqref{eq:dg-lf-error-a}, and the discrete variation-of-constants formula we conclude
    \begin{equation*}
        \vve_h^n
            = \sum_{\ell=0}^{n-1} \Rh^{n-1-\ell}\Rhl^{-1}\vvd_\Pi^\ell
                + \sum_{\ell=0}^{n-1} \Rh^{n-1-\ell}\Rhl^{-1}\vvdelta^\ell
                + \sum_{\ell=0}^{n-1} \Rh^{n-1-\ell}\bigl(I-\Rh\bigr)\vvd_h^\ell.
    \end{equation*}
    
    To bound the first term on the right-hand side we use \cref{lem:ProjectionUnderDGCurl} to see
	\begin{equation*}
		\normmueps{\vvd_\Pi^n} \leq 
		C h^{\regcoeff} \frac{\tau}{2}\Bigl(
		\singlenorm{\uf[n+1]+\uf[n]}_{\sobfuns{\regcoeff+1}{\meshh}^6}
		+ \singlenorm{\Ef[n+1]-\Ef[n]}_{\sobfuns[]{\regcoeff+1}{\meshh}^3}
		\Bigr).
	\end{equation*}
    Utilizing \cref{eq:BoundsOnall} shows
    \begin{align*}
        \sum_{\ell=0}^{n-1} \normmueps{\Rh^{n-1-\ell}\Rhl^{-1}\vvd_\Pi^\ell}
        &\leq \Cstab \sqrt{3}\sum_{\ell=0}^{n-1} \frac{\tau}{2}\normmueps{\vvd_\Pi^\ell}
        \leq C h^{\regcoeff}.
    \end{align*}
	For the second term, we obtain 
    \begin{equation*}
        \sum_{\ell=0}^{n-1} \normmueps{\Rh^{n-1-\ell}\Rhl^{-1}\vvdelta^\ell}
        \leq C \tau^2 \int_{t_0}^{t_n} \normmueps{\dt^3 \uf(s)} \ds.
    \end{equation*}
    Hence, it remains to bound the third term. Using summation-by-parts, we infer
    \begin{equation*}
        \sum_{\ell=0}^{n-1} \Rh^{n-1-\ell}\bigl(I-\Rh\bigr)\vvd_h^\ell
        = -\Rh^n\vvd_h^0 
            + \vvd_h^{n-1} 
            + \sum_{\ell=0}^{n-2}\Rh^{n-1-\ell}(\vvd_h^{\ell+1} - \vvd_h^{\ell}).
    \end{equation*}
    We estimate all terms separately.
    Again, using  \eqref{eq:dg-lf-error-d} and the fundamental theorem of calculus implies
    \begin{equation*}
        \normmueps{\Rh^n\vvd_h^0 } 
            \leq C\tau\int_{t_0}^{t_1} \normmu{\dt^2\Hf[](s)}\ds
            \leq C\tau^2 \max_{s\in [t_0,t_1]} \normmu{\dt^2\Hf(s)}
    \end{equation*}
    and 
    \begin{equation*}
        \normmueps{\vvd_h^n}\leq C\tau^2 \max_{s\in [t_{n-1},t_n]} \normmu{\dt^2\Hf(s)}.
    \end{equation*}
    For the third sum, the fundamental theorem is used twice in order to exploit the difference of the defects.
    We obtain
    \begin{equation*}
        \frac{\tau}{4} \OProj\bigl((\dt\Hf[\ell+2] - 2\dt\Hf[\ell-1] + \dt\Hf[\ell]\bigr)
        = \frac{\tau^2}{4} \int_{t_\ell}^{t_{\ell+2}} \Bigl(1-\frac{|t_{\ell+1} -s|}{\tau}\Bigr) \OProj \dt^3 \Hf[](s)\ds.
    \end{equation*}
    Thus, we end up with the bound
    \begin{equation*}
        \sum_{\ell=0}^{n-2}\normmueps{\Rh^{n-1-\ell}(\vvd_h^{\ell+1} - \vvd_h^{\ell})}
        \leq C\tau^2 \int_{t_1}^{t_{n-1}} \normmu{\dt^3\Hf(s)} \ds.
    \end{equation*}

    Combining all estimates, we have shown that
    \begin{equation*}
        \normmueps{\vve_h^n} \leq C (h^{\regcoeff} + \tau^2 ). 
    \end{equation*}
    Together with \cref{eq:FullProjectionError} this proves the claim.
\end{proof}

\section{Numerical experiments}\label{sec:NumericalExperiments}

In this section, we present several numerical experiments that underline our theoretical findings.
The software was build with the Maxwell toolbox 
\href{https://gitlab.kit.edu/kit/ianm/ag-numerik/projects/dg-maxwell/timaxdg}{\texttt{TiMaxdG}}\footnote{\url{https://gitlab.kit.edu/kit/ianm/ag-numerik/projects/dg-maxwell/timaxdg}} 
which is build upon the finite element library 
\href{https://www.dealii.org}{\texttt{deal.II}}\footnote{\url{https://www.dealii.org}} \cite{dealII95}.
The full software with executables for reproduction purposes can be found under
\begin{center}
  \url{https://doi.org/10.35097/nwa3t4pv0jxwpf5v}.
\end{center}

All experiments are conducted in transverse electric (TE) polarization, i.e., $H_{1} = H_2 = E_3 = 0$,
to reduce the computational effort,
 with material parameters \(\Msuspm[] = \Epermpm[] = 1\).
Thus, we solve the system
\begin{subequations}
  \begin{align}
    \partial_t H_{3,\pm} &= -\partial_1 E_{2,\pm}  + \partial_2 E_{1,\pm},
      &&\text{in}\ \domainlr,\\
    \partial_t E_{1,\pm} &= \partial_2 H_{3,\pm} - J_{1,\pm},
      &&\text{in}\ \domainlr,\\
    \partial_t E_{2,\pm} &= -\partial_1 H_{3,\pm} - J_{2,\pm},
      &&\text{in}\ \domainlr,\\
    H_{3}(0) &= H_{3}^0,\ E_{1}(0) = E_{1}^0,\  E_{2}(0) = E_{2}^0,
      &&\text{in}\ \domain,\\
    \jmp{H_3}_\Fint &= J_{1,\operatorname{surf}},
      &&\text{on}\ \Fint.
  \end{align}
\end{subequations}

\subsection*{Cavity solution}\label{subsec:CavitySolution}
In this experiment, we use the well-known cavity solution, \cf{} \cite[Sec.~6]{Hoc.K.2022}, to construct a regular reference solution of the interface problem \cref{eq:MaxwellFormulation}.
On each cuboid \(\domainleft\), \(\domainright\) we make the ansatz
\begin{subequations}\label{eq:CavityAnsatz}
    \begin{align}
      H_{3,\pm}(x_1, x_2, t) 
        &= \frac{1}{\omega^{\pm}} 
          (k^{\pm}_1  A_2 - k_2  A^{\pm}_1)\cos{(k_2 x_2)}\cos{(k^{\pm}_1 (x_1+1))}\sin{(\omega^{\pm} t)},\\
      E_{1,\pm}(x_1, x_2, t) &= -A^{\pm}_1 \sin{(k_2 x_2)} \sin{(k^{\pm}_1 (x_1+1))} \cos{(\omega^{\pm} t)},\\
      E_{2,\pm}(x_1, x_2, t) &= -A_2 \cos{(k_2 x_2)} \sin{(k^{\pm}_1 (x_1+1))} \cos{(\omega^{\pm} t)},
    \end{align}
    with spatial wave numbers
    \begin{equation}
      k^{\pm}_1 = \frac{\pi k^\pm}{2},
      \quad 
       k_2 = \pi m,
      \quad k^\pm, m\in\bbN,
    \end{equation}
    temporal wave numbers
    \begin{equation}
      \omega^{\pm} = \sqrt{\bigl(k^{\pm}_1\bigr)^2 + k_2^2},
    \end{equation}
    and amplitudes
    \begin{equation}
      A^{\pm}_1 = - A_2 \frac{  k_2}{k^{\pm}_1},\quad A_2 \in \bbR.
    \end{equation}
  \end{subequations}
  We choose the data such that 
  \begin{subequations} 
    \begin{align}\label{eq:CavityData1}
      J_{\operatorname{surf}}(x_2, t) 
      &= \lim_{x_1\to 0^+} H_{3,+}(x_1,x_2,t) - \lim_{x_1\to 0^-} H_{3,-}(x_1,x_2,t),\\\label{eq:CavityData2}
      J_{1,\pm} &= J_{2,\pm} = 0.
    \end{align}
  \end{subequations}
  The remaining constants are chosen such that \(J_{\operatorname{surf}} \neq 0\).
  In our simulation we used the specific values
  \begin{equation*}
    k^- = 2,
    \qquad k^+ = 4,
    \qquad m = 1,
    \qquad A_2 = 1.
  \end{equation*}

  We chose a mesh sequence of 20 meshes with \meshsize s in the range of \(1\cdot 10^{-1}\) and \(1\cdot 10^{-2}\),
  a fixed time-step width \(\tau = 1\cdot 10^{-4}\),
  and polynomial degrees between one and four.
  For each \meshsize, we calculated two different numerical solutions differing in the treatment of the surface current \cref{eq:CavityData1}.
  One series of simulations is done with the lifting defined in \cref{eq:LiftOp} and one series is done with the interpolation of the surface current described in \cref{eq:InterpolateSemiDiscreteMaxwell}.
  At several \timestep s, we calculated  the \(L^2\)-error against the reference solution \cref{eq:CavityAnsatz}.
  \Cref{fig:CavityErrorPlot} depicts the different \meshsize s on the \(x\)-axis and the maximal \(L^2\)-error obtained on the \(y\)-axis.
  We observe for \(k\)-th order ansatz polynomials \(k\)-th order spatial convergence until a plateau is reached where the error of the time discretization dominates.
  This agrees with both, \cref{thm:SpatialConvergence,cor:InterpolationSpatialConvergence}
  Additionally, \cref{fig:CavityErrorPlot} shows that the interpolation of the surface current leads to the same spatial error.
  This is expected since the surface current \cref{eq:CavityData1} is smooth.

\begin{figure}
  \centering\scalebox{0.85}{
    \begin{tikzpicture}

\definecolor{darkgray176}{RGB}{176,176,176}
\definecolor{gray}{RGB}{128,128,128}
\definecolor{lightgray204}{RGB}{204,204,204}

\definecolor{color01}{HTML}{1f77b4}
\definecolor{color02}{HTML}{ff7f0e}
\definecolor{color03}{HTML}{2ca02c}
\definecolor{color04}{HTML}{d62728}
\definecolor{color05}{HTML}{9467bd}
\definecolor{color06}{HTML}{8c564b}
\definecolor{color07}{HTML}{e377c2}
\definecolor{color08}{HTML}{7f7f7f}
\definecolor{color09}{HTML}{bcbd22}
\definecolor{color10}{HTML}{17becf}

\begin{axis}[
legend cell align={left},
legend style={
  fill opacity=0.8,
  draw opacity=1,
  text opacity=1,
  at={(1.09,0.5)},
  anchor=west,
  draw=lightgray204
},
log basis x={10},
log basis y={10},
tick align=outside,
tick pos=left,
x grid style={darkgray176},
xlabel={\(\displaystyle h\)},
xmajorgrids,
xmin=0.00893597900601822, xmax=0.113381012625214,
xmode=log,
xtick style={color=black},
xtick={0.0001,0.001,0.01,0.1,1,10},
minor x tick num=10,
xticklabels={
  \(\displaystyle {10^{-4}}\),
  \(\displaystyle {10^{-3}}\),
  \(\displaystyle {10^{-2}}\),
  \(\displaystyle {10^{-1}}\),
  \(\displaystyle {10^{0}}\),
  \(\displaystyle {10^{1}}\)
},
y grid style={darkgray176},
ylabel={\(\displaystyle \max_n \|\mathbf{u}^n_h - \mathbf{u}(t_n)\|_{L^2(Q)^6}\)},
ymajorgrids,
ymin=4e-9, ymax=0.309175469033012,
ymode=log,
ytick style={color=black},
ytick={0.000000001, 0.00000001, 0.0000001, 0.000001, 0.00001, 0.0001, 0.001, 0.01, 0.1},
yticklabels={
  \(\displaystyle {10^{-9}}\),
  \(\displaystyle {10^{-8}}\),
  \(\displaystyle {10^{-7}}\),
  \(\displaystyle {10^{-6}}\),
  \(\displaystyle {10^{-5}}\),
  \(\displaystyle {10^{-4}}\),
  \(\displaystyle {10^{-3}}\),
  \(\displaystyle {10^{-2}}\),
  \(\displaystyle {10^{-1}}\)
}
]

\addplot [semithick, color01, mark=10-pointed star, mark size=1.5, mark options={solid}]
table {%
0.101015 0.059981
0.0883883 0.0523818
0.0785674 0.0465003
0.0707107 0.0418112
0.0614875 0.0363225
0.0543928 0.0321108
0.048766 0.0287761
0.042855 0.0252775
0.038222 0.0225384
0.0336718 0.0198503
0.0294628 0.0173656
0.0261891 0.015434
0.0231838 0.0136613
0.0207973 0.0122541
0.0183664 0.010821
0.0162553 0.00957671
0.0144308 0.00850142
0.0127407 0.00750551
0.0113137 0.00666472
0.0100299 0.00590832
};
\addlegendentry{$k = 1$}

\addplot [semithick, color01, mark=o, mark size=3, mark options={solid, color=red}]
table {%
0.101015 0.061216
0.0883883 0.0532165
0.0785674 0.0470902
0.0707107 0.0422431
0.0614875 0.0366079
0.0543928 0.032309
0.048766 0.0289192
0.042855 0.0253749
0.038222 0.0226076
0.0336718 0.0198977
0.0294628 0.0173973
0.0261891 0.0154563
0.0231838 0.0136768
0.0207973 0.0122653
0.0183664 0.0108287
0.0162553 0.00958206
0.0144308 0.00850517
0.0127407 0.00750808
0.0113137 0.00666653
0.0100299 0.00590958
};
\addlegendentry{$k = 1$, interpolation}

\addplot [semithick, color02, mark=10-pointed star, mark size=1.5, mark options={solid}]
table {%
0.101015 0.000565927
0.0883883 0.000479417
0.0785674 0.000327014
0.0707107 0.000268407
0.0614875 0.000196602
0.0543928 0.000155265
0.048766 0.00012579
0.042855 9.36902e-05
0.038222 7.84926e-05
0.0336718 5.77729e-05
0.0294628 4.45609e-05
0.0261891 3.5334e-05
0.0231838 2.64406e-05
0.0207973 2.26963e-05
0.0183664 1.64756e-05
0.0162553 1.36151e-05
0.0144308 1.07957e-05
0.0127407 8.24299e-06
0.0113137 6.55059e-06
0.0100299 5.07358e-06
};
\addlegendentry{$k = 2$}

\addplot [semithick, color02, mark=o, mark size=3, mark options={solid, color=red}]
table {%
0.101015 0.000639237
0.0883883 0.000519692
0.0785674 0.000359711
0.0707107 0.000291036
0.0614875 0.000210951
0.0543928 0.000164704
0.048766 0.000132214
0.042855 9.80065e-05
0.038222 8.13219e-05
0.0336718 5.97557e-05
0.0294628 4.58375e-05
0.0261891 3.61985e-05
0.0231838 2.70572e-05
0.0207973 2.31022e-05
0.0183664 1.67691e-05
0.0162553 1.38046e-05
0.0144308 1.09251e-05
0.0127407 8.33256e-06
0.0113137 6.61188e-06
0.0100299 5.11645e-06
};
\addlegendentry{$k = 2$, interpolation}

\addplot [semithick, color03, mark=10-pointed star, mark size=1.5, mark options={solid}]
table {%
0.101015 6.16602e-05
0.0883883 4.22648e-05
0.0785674 2.9909e-05
0.0707107 2.18143e-05
0.0614875 1.44021e-05
0.0543928 9.98772e-06
0.048766 7.20377e-06
0.042855 4.8878e-06
0.038222 3.47092e-06
0.0336718 2.3699e-06
0.0294628 1.58776e-06
0.0261891 1.11537e-06
0.0231838 7.73774e-07
0.0207973 5.59044e-07
0.0183664 3.86084e-07
0.0162553 2.69139e-07
0.0144308 1.90444e-07
0.0127407 1.35279e-07
0.0113137 1.09972e-07
0.0100299 9.49847e-08
};
\addlegendentry{$k = 3$}

\addplot [semithick, color03, mark=o, mark size=3, mark options={solid, color=red}]
table {%
0.101015 6.22216e-05
0.0883883 4.24552e-05
0.0785674 3.0003e-05
0.0707107 2.19455e-05
0.0614875 1.44687e-05
0.0543928 1.00221e-05
0.048766 7.2217e-06
0.042855 4.89765e-06
0.038222 3.47293e-06
0.0336718 2.37282e-06
0.0294628 1.58879e-06
0.0261891 1.11561e-06
0.0231838 7.74087e-07
0.0207973 5.59272e-07
0.0183664 3.86152e-07
0.0162553 2.69187e-07
0.0144308 1.90449e-07
0.0127407 1.35275e-07
0.0113137 1.09971e-07
0.0100299 9.49854e-08
};
\addlegendentry{$k = 3$, interpolation}

\addplot [semithick, color04, mark=10-pointed star, mark size=1.5, mark options={solid}]
table {%
0.101015 6.78096e-07
0.0883883 4.54181e-07
0.0785674 2.54205e-07
0.0707107 1.75834e-07
0.0614875 1.05782e-07
0.0543928 7.89142e-08
0.048766 7.85214e-08
0.042855 7.83435e-08
0.038222 7.82965e-08
0.0336718 7.82743e-08
0.0294628 7.82668e-08
0.0261891 7.82646e-08
0.0231838 7.82637e-08
0.0207973 7.82634e-08
0.0183664 7.82632e-08
0.0162553 7.82632e-08
0.0144308 7.82632e-08
0.0127407 7.82632e-08
0.0113137 7.82632e-08
0.0100299 7.82632e-08
};
\addlegendentry{$k = 4$}

\addplot [semithick, color04, mark=o, mark size=3, mark options={solid, color=red}]
table {%
0.101015 6.88447e-07
0.0883883 4.57641e-07
0.0785674 2.56454e-07
0.0707107 1.76989e-07
0.0614875 1.06233e-07
0.0543928 7.91101e-08
0.048766 7.85865e-08
0.042855 7.83612e-08
0.038222 7.83022e-08
0.0336718 7.8276e-08
0.0294628 7.82672e-08
0.0261891 7.82647e-08
0.0231838 7.82637e-08
0.0207973 7.82634e-08
0.0183664 7.82633e-08
0.0162553 7.82632e-08
0.0144308 7.82632e-08
0.0127407 7.82632e-08
0.0113137 7.82632e-08
0.0100299 7.82632e-08
};
\addlegendentry{$k = 4$, interpolation}

\addplot [semithick, gray, dashed]
table {%
0.101015 0.1101888
0.0883883 0.0964153908928377
0.0785674 0.0857025939228827
0.0707107 0.0771323781632431
0.0614875 0.0670715620452408
0.0543928 0.0593325482417463
0.048766 0.0531947435608573
0.042855 0.0467469289115478
0.038222 0.0416931773855368
0.0336718 0.0367297454421621
0.0294628 0.0321384999914864
0.0261891 0.0285674949470871
0.0231838 0.0252892649749047
0.0207973 0.0226860320768203
0.0183664 0.0200343669387715
0.0162553 0.0177315448264119
0.0144308 0.0157413506413899
0.0127407 0.0138977621557194
0.0113137 0.0123411674163243
0.0100299 0.0109407775589764
};

\addplot [semithick, gray, dashed]
table {%
0.101015 0.0011506266
0.0883883 0.00088095197901529
0.0785674 0.000696061216677391
0.0707107 0.000563810223386761
0.0614875 0.000426320625607311
0.0543928 0.000333614950633009
0.048766 0.000268161818495487
0.042855 0.000207093120563769
0.038222 0.000164736369141493
0.0336718 0.000127848412504306
0.0294628 9.78837747093633e-05
0.0261891 7.73399383106091e-05
0.0231838 6.06082740363264e-05
0.0207973 4.8772676213075e-05
0.0183664 3.80373935955328e-05
0.0162553 2.97956354118356e-05
0.0144308 2.34824544657483e-05
0.0127407 1.83041367933624e-05
0.0113137 1.44335111005611e-05
0.0100299 1.1343729960236e-05
};

\addplot [semithick, gray, dashed]
table {%
0.101015 0.00011199888
0.0883883 7.5030945347936e-05
0.0785674 5.26966757408705e-05
0.0707107 3.84159418090527e-05
0.0614875 2.5259024949744e-05
0.0543928 1.74855912690684e-05
0.048766 1.26010715680192e-05
0.042855 8.55185982171109e-06
0.038222 6.06731159995304e-06
0.0336718 4.14815594397906e-06
0.0294628 2.77893333377165e-06
0.0261891 1.95172127871558e-06
0.0231838 1.35397280039635e-06
0.0207973 9.7741037531471e-07
0.0183664 6.73175328278889e-07
0.0162553 4.66703483280175e-07
0.0144308 3.26533165408464e-07
0.0127407 2.24717008210946e-07
0.0113137 1.5735119696053e-07
0.0100299 1.09634157692165e-07
};

\addplot [semithick, gray, dashed]
table {%
0.101015 1.2392046e-06
0.0883883 7.26404738316086e-07
0.0785674 4.53491415085276e-07
0.0707107 2.97536390684669e-07
0.0614875 1.70116675310377e-07
0.0543928 1.04175407308662e-07
0.048766 6.73082162670098e-08
0.042855 4.01426040682708e-08
0.038222 2.54011381677172e-08
0.0336718 1.5299064545253e-08
0.0294628 8.96800031674101e-09
0.0261891 5.59863094182393e-09
0.0231838 3.43825458353961e-09
0.0207973 2.22652394632174e-09
0.0183664 1.35424019020952e-09
0.0162553 8.30958817004488e-10
0.0144308 5.16132462447178e-10
0.0127407 3.13597580199635e-10
0.0113137 1.94992621713536e-10
0.0100299 1.20444218028669e-10
};
\end{axis}

\end{tikzpicture}
    }
    \caption{Error in the $L^2$-norm of the cavity solution \cref{eq:CavityAnsatz} plotted against the \meshsize{} for a fixed time-step width \(\tau = 1\cdot 10^{-4}\). The dotted lines correspond to the lift defined in \cref{eq:LiftOp}, and the circled lines to the interpolation in \cref{eq:InterpolateSemiDiscreteMaxwell}.}
    \label{fig:CavityErrorPlot}
\end{figure}

\subsection*{Low regularity surface current}\label{subsec:LowRegularity}

The aim of this experiment is to show the effect of spatial regularity of surfaces currents on the spatial convergence order.
We follow the ideas in \cite{Hoc.L.O.2020} and construct for \(\alpha \geq 0\) trigonometric polynomials
\begin{equation}\label{eq:TrigonometricPolynomial}
  f_\alpha(x) = \sum_{j=-M/2+1}^{M/2} \nu_{\alpha,j} e^{ijx},
  \qquad x\in [-\pi,\pi],
  \qquad M = 2^m, m\in\bbN,
\end{equation}
with coefficients
\begin{equation*}
  \begin{aligned}
    \nu_{\alpha,0} &= \nu_{\alpha,M/2} = 0,\\
    \nu_{\alpha,j} &= \frac{i \cdot r_j}{(1+j^2)^{\frac{1}{2}(\frac{1}{2}+\alpha)}} &&\text{for } j=1,\ldots,\frac{M}{2}-1,\\
    \nu_{\alpha,j} &= -\nu_{\alpha,j+M/2} &&\text{for } j=-\frac{M}{2}+1,\ldots,-1.
  \end{aligned}
\end{equation*}
The factors \((r_j)_{j=1}^{M/2-1}\) are uniformly sampled numbers from the interval \([-1,1]\).
In the limit \(M\to\infty\), the sequence of trigonometric polynomials converges to a function in the Sobolev space \(H^{\alpha}_{\operatorname{per}}(-\pi,\pi)\) with norm
\begin{equation*}
  \norm{g}_{\alpha}^2 = 2\pi \sum_{j\in\bbZ} (1+j^2)^\alpha |\widehat{g}_j|^2,
  \quad g(x) = \sum_{j\in\bbZ} \widehat{g}_j e^{ijx}.
\end{equation*}
The coefficients are chosen in such a way that the trace of \(f_\alpha\) vanishes on the boundary \(\{-\pi,\pi\}\) and, thus, satisfy homogenous Dirichlet boundary conditions.
By construction, the norm \(\norm{f_\alpha}_{\eta}\) is bounded uniformly in \(M\) for \(\eta\leq \alpha\) and grows otherwise.
\Cref{fig:DiscreteRegularity} demonstrates this behavior.
\begin{figure}
  \centering\scalebox{0.70}{
    \input{res/plots/surface_current/multiple_alphas.tex}
  }
  \caption{Illustration of discrete regularity: the norm $\|f_\alpha\|_{H^{\eta}}$, for $f_\alpha$ given in \cref{eq:TrigonometricPolynomial},
  	 for different values of $\alpha$ and $\eta$ plotted against the number $M$ of Fourier modes.}
  \label{fig:DiscreteRegularity}
\end{figure}
The surface current is defined as
\begin{equation*}
  J_{\operatorname{surf}}(t,x_2) 
    = \frac{1}{\norm{f_\alpha}_0} f_\alpha(2\pi x_2 - \pi) \sin^2(\pi t)
\end{equation*}
for \(x_2\in [0,1]\) and \(t\in[0,T]\).

The fully discretize scheme \cref{eq:LeapfrogScheme} is used without an additional interpolation on the interface.
In our experiment, we chose $M = 2^m$ for \(m = 22\) and a series of regularity coefficients \(\alpha \in [0,4]\).
For every \(\alpha\), we calculated a reference solution on fine mesh 
with  polynomial degree \(k_{\operatorname*{ref}} = 3\), \meshsize{} \(h_{\operatorname*{ref}} = 1\cdot 10^{-3}\), and \stepsize{} \(\tau_{\operatorname*{ref}} = 5\cdot 10^{-5}\).
We then compared the \(L^2\)-error of a sequence of solutions on 8 different meshes with \meshsize s in the range between \(1\cdot 10^{-1}\) and  \(5\cdot 10^{-3}\) at the end time \(T=1\) against the reference solution and estimated the order of convergence (EOC).
The experiment was performed for first and second order polynomials with the fixed \timestep{} size \(\tau = 2.5\cdot 10^{-4}\).

\cref{fig:RegPlots} shows the dependence of the convergence order on the spatial regularity of the surface current.
For \(\alpha < 1/2\) little to no convergence is observed.
The order then grows linearly for \(\alpha\in [0.5,1.5]\) until it stagnates.

This agrees with \cref{thm:SpatialConvergence} provided that one can improve on the results 
in \cite[Sec.~2]{Dor.Z.2023} to solutions with piecewise regularity $ \psobfuns{s}$ for \(s=\min\monosetc{\alpha+1/2,2}\), \(\alpha > 1/2\).
In particular, 
looking at the proofs one would expect that 
the regularity requirements on $\Jsurf$ can be reduced by one order of Sobolev regularity.

\begin{figure}
  \centering\scalebox{0.85}{
    \begin{tikzpicture}

    \definecolor{darkgray176}{RGB}{176,176,176}
    \definecolor{lightgray204}{RGB}{204,204,204}
    \definecolor{orange}{RGB}{255,165,0}
    \definecolor{darkorange25512714}{RGB}{255,127,14}
    \definecolor{steelblue31119180}{RGB}{31,119,180}

    \definecolor{color01}{HTML}{1f77b4}
    \definecolor{color02}{HTML}{ff7f0e}
    \definecolor{color03}{HTML}{2ca02c}
    \definecolor{color04}{HTML}{d62728}
    \definecolor{color05}{HTML}{9467bd}
    \definecolor{color06}{HTML}{8c564b}
    \definecolor{color07}{HTML}{e377c2}
    \definecolor{color08}{HTML}{7f7f7f}
    \definecolor{color09}{HTML}{bcbd22}
    \definecolor{color10}{HTML}{17becf}
    
    \begin{axis}[
    legend cell align={left},
    legend style={
      fill opacity=0.8,
      draw opacity=1,
      text opacity=1,
      at={(1.09,0.5)},
      anchor=west,
      draw=lightgray204
    },
    tick align=outside,
    tick pos=left,
    x grid style={darkgray176},
    xlabel={\(\displaystyle \alpha\)},
    xmin=-0.2, xmax=4.2,
    xtick style={color=black},
    y grid style={darkgray176},
    ylabel={Estimated order \(q\)},
    ymin=-0.200098889948952, ymax=2.20000470904519,
    ytick style={color=black}
    ]

    \addplot [semithick, gray, opacity=1, dashed, forget plot]
    table {%
    0.5 0
    0.5 4
    };

    \addplot [thick, gray, gray, dashed, opacity=1, forget plot]
    table{
      0.0 0
      0.5 0
      1.5 1
      4   1
    };

    \addplot [thick, gray, gray, dashed, opacity=1, forget plot]
    table{
      0.0 0
      0.5 0
      2.5 2
      4   2
    };

    \addplot [semithick, color01, mark=*, mark size=1.5, mark options={solid}]
    table {%
    0 -9.41753387451172e-05
    0.100000023841858 0.0128304958343506
    0.299999952316284 0.0495185852050781
    0.5 0.118548512458801
    0.700000047683716 0.237093329429626
    0.899999976158142 0.39873468875885
    1 0.487832188606262
    1.10000002384186 0.576982140541077
    1.20000004768372 0.661380529403687
    1.29999995231628 0.735798001289368
    1.39999997615814 0.795694708824158
    1.5 0.839012384414673
    1.60000002384186 0.867058515548706
    1.70000004768372 0.88339376449585
    1.79999995231628 0.89190661907196
    1.89999997615814 0.895661592483521
    2 0.896693229675293
    2.09999990463257 0.896235704421997
    2.20000004768372 0.895007491111755
    2.29999995231628 0.906915187835693
    2.40000009536743 0.891699314117432
    2.5 0.889979124069214
    2.75 0.886052966117859
    3 0.882893443107605
    3.25 0.880485773086548
    3.5 0.878710031509399
    3.75 0.877436995506287
    4 0.876549005508423
    };
    \addlegendentry{deg. \(p=1\)}

    \addplot [semithick, color02, mark=asterisk, mark size=2, mark options={solid}]
    table {%
    0 0.00367856025695801
    0.100000023841858 0.0234479904174805
    0.299999952316284 0.071052074432373
    0.5 0.144925832748413
    0.700000047683716 0.258962154388428
    0.899999976158142 0.41022515296936
    1 0.495564222335815
    1.10000002384186 0.584991812705994
    1.20000004768372 0.677031278610229
    1.29999995231628 0.770250201225281
    1.39999997615814 0.862627625465393
    1.5 0.950343608856201
    1.60000002384186 1.02625620365143
    1.70000004768372 1.08046507835388
    1.79999995231628 1.10666215419769
    1.89999997615814 1.10893678665161
    2 1.09796285629272
    2.09999990463257 1.08277106285095
    2.20000004768372 1.06814467906952
    2.29999995231628 1.01932394504547
    2.40000009536743 1.04612720012665
    2.5 1.03870093822479
    2.75 1.02721798419952
    3 1.02154850959778
    3.25 1.01859092712402
    3.5 1.01692819595337
    3.75 1.01591622829437
    4 1.01525640487671
    };
    \addlegendentry{deg. \(p=2\)}
    \end{axis}
  \end{tikzpicture}
  }
  \caption{Plot of the estimated order of convergence against the discrete regularity parameter $\alpha$. For each $\alpha$, we computed 	approximations using $8$ different \meshsize s between  	$1\cdot 10^{-1}$ and  $5\cdot 10^{-3}$ with a fixed \timestep{} size $\tau = 2.5\cdot 10^{-4}$. They are then compared to a reference solution computed with $k_{\operatorname*{ref}}=3$, $h_{\operatorname*{ref}} = 1\cdot 10^{-3}$ and $\tau_{\operatorname*{ref}} = 5\cdot 10^{-5}$.}
  \label{fig:RegPlots}
\end{figure}

\subsection*{Polynomial solution}\label{subsec:PolynomialSolution}
In this example, we investigate the temporal errors by constructing a polynomial solution which does not create spatial errors.
We construct a solution that is polynomial in space in order to isolate the error introduced by time discretization.
\begin{subequations} \label{eq:pol_sol}
  The ansatz polynomials in space are given by 
  \begin{equation}
    q_-(x_1) = 2 + x_1,
    \quad q_+(x_1) = -1 + x_1,
    \quad r(x_2) = x_2(1-x_2)
  \end{equation}
  and the ansatz function in time by
  \begin{equation}
    p(t) = \sin(2\pi t).
  \end{equation}
  We define on \(\domainlr\) the fields
  \begin{align}
    H_{3,\pm}(x_1, x_2, t) &= p(t)q_\pm(x_1)r'(x_2),\\
    E_{1,\pm}(x_1, x_2, t) &= p'(t)q_\pm(x_1)r(x_2),\\
    E_{2,\pm}(x_1, x_2, t) &= 0
  \end{align}
  and define the surface current again as
  \begin{equation}
    J_{\operatorname{surf}}(x_2, t)
    = \lim_{x_1\to 0^+} H_{3,+}(x_1,x_2,t) - \lim_{x_1\to 0^-} H_{3,-}(x_1,x_2,t).
  \end{equation}
  Additionally, we define the volume current on \(\domainlr\) as
  \begin{align} 
    J_{1,\pm}(x_1,x_2,t) &= p(t)q_\pm(x_1)r''(x_2),\\
    J_{2,\pm}(x_1,x_2,t) &= -p(t)q_\pm'(x_1)r'(x_2).
  \end{align}
\end{subequations}

The fully discretized scheme \cref{eq:LeapfrogScheme} is used without an additional interpolation on the interface.
We chose a mesh sequence with 5 different \meshsize s between \(5\cdot 10^{-1}\) and \(5\cdot 10^{-2}\).
For every mesh in the sequence, we compared the \(L^2\)-error between the reference solution and the numerical scheme at several \timestep s for a total of 40 different \timestep{} sizes in the range between \(1\cdot 10^{-1}\) and \(1\cdot 10^{-4}\).
Throughout all calculations, we used a polynomial degree of 3 in order to discretize the reference solution exactly in space.
\Cref{fig:PolynomialErrorPlot} shows on the x-axis the \timestep size \(\tau\) and the maximal \(L^2\)-error on y-axis.
The method converges with second order in time if the CFL condition \cref{def:CFLCondition} is satisfied.

\begin{figure}
  \centering\scalebox{0.85}{
      \begin{tikzpicture}

\definecolor{color01}{HTML}{1f77b4}
\definecolor{color02}{HTML}{ff7f0e}
\definecolor{color03}{HTML}{2ca02c}
\definecolor{color04}{HTML}{d62728}
\definecolor{color05}{HTML}{9467bd}
\definecolor{color06}{HTML}{8c564b}
\definecolor{color07}{HTML}{e377c2}
\definecolor{color08}{HTML}{7f7f7f}
\definecolor{color09}{HTML}{bcbd22}
\definecolor{color10}{HTML}{17becf}

\definecolor{darkgray176}{RGB}{176,176,176}
\definecolor{gray}{RGB}{128,128,128}
\definecolor{lightgray204}{RGB}{204,204,204}

\begin{axis}[
legend cell align={left},
legend style={
  fill opacity=0.8,
  draw opacity=1,
  text opacity=1,
  at={(1.09,0.5)},
  anchor=west,
  draw=lightgray204
},
log basis x={10},
log basis y={10},
tick align=outside,
tick pos=left,
x grid style={darkgray176},
xlabel={\(\displaystyle \tau\)},
xmajorgrids,
xmin=7.07945784384137e-05, xmax=0.141253754462275,
xmode=log,
xtick style={color=black},
xtick={1e-00, 1e-01, 1e-02, 1e-03, 1e-04, 1e-05},
xticklabels={
  \(\displaystyle {10^{0}}\),
  \(\displaystyle {10^{-1}}\),
  \(\displaystyle {10^{-2}}\),
  \(\displaystyle {10^{-3}}\),
  \(\displaystyle {10^{-4}}\),
  \(\displaystyle {10^{-5}}\),
},
y grid style={darkgray176},
ylabel={\(\displaystyle \max_n \|\mathbf{u}^n_h - \mathbf{u}(t_n)\|_{L^2(Q)^6}\)},
ymajorgrids,
ymin=1e-08, ymax=0.1,
ymode=log,
ytick style={color=black},
ytick={1e-00, 1e-01, 1e-02, 1e-03, 1e-04, 1e-05, 1e-06, 1e-07, 1e-08},
yticklabels={
  \(\displaystyle {10^{0}}\),
  \(\displaystyle {10^{-1}}\),
  \(\displaystyle {10^{-2}}\),
  \(\displaystyle {10^{-3}}\),
  \(\displaystyle {10^{-4}}\),
  \(\displaystyle {10^{-5}}\),
  \(\displaystyle {10^{-6}}\),
  \(\displaystyle {10^{-7}}\),
  \(\displaystyle {10^{-8}}\)
}
]
\addplot [semithick, color01, mark=asterisk, mark size=3, mark options={solid}]
table {%
0.1 1000
0.0833333 1000
0.0714286 1000
0.0588235 1000
0.05 1000
0.0416667 1000
0.0344828 0.011364
0.0285714 0.00779554
0.0243902 0.0056782
0.0204082 0.00397395
0.0169492 0.00274026
0.0142857 0.00194636
0.0119048 0.00135146
0.01 0.000953504
0.00840336 0.000673289
0.00699301 0.000466233
0.00588235 0.000329886
0.00492611 0.00023302
0.00413223 0.000162785
0.00346021 0.000114142
0.00289017 7.96316e-05
0.00242718 5.67631e-05
0.00203252 3.93825e-05
0.00170068 2.75726e-05
0.0014245 1.93769e-05
0.00119332 1.37139e-05
0.001 9.53303e-06
0.000837521 6.68686e-06
0.000701754 4.73305e-06
0.000587889 3.31436e-06
0.000492368 2.32819e-06
0.000412541 1.63685e-06
0.000345543 1.15075e-06
0.000289436 8.06839e-07
0.000242424 5.66317e-07
0.000203087 3.97191e-07
0.000170126 2.789e-07
0.000142511 1.95718e-07
0.000119374 1.37345e-07
0.0001 9.63704e-08
};
\addlegendentry{$h=0.47141$}

\addplot [semithick, color02, mark=asterisk, mark size=3, mark options={solid}]
table {%
0.1 1000
0.0833333 1000
0.0714286 1000
0.0588235 1000
0.05 1000
0.0416667 1000
0.0344828 1000
0.0285714 1000
0.0243902 1000
0.0204082 0.00397396
0.0169492 0.00274027
0.0142857 0.00194637
0.0119048 0.00135146
0.01 0.000953507
0.00840336 0.00067329
0.00699301 0.000466234
0.00588235 0.000329887
0.00492611 0.00023302
0.00413223 0.000162786
0.00346021 0.000114142
0.00289017 7.96317e-05
0.00242718 5.6763e-05
0.00203252 3.93826e-05
0.00170068 2.75727e-05
0.0014245 1.93769e-05
0.00119332 1.37139e-05
0.001 9.53304e-06
0.000837521 6.68687e-06
0.000701754 4.73305e-06
0.000587889 3.31444e-06
0.000492368 2.32819e-06
0.000412541 1.63686e-06
0.000345543 1.15075e-06
0.000289436 8.06839e-07
0.000242424 5.66316e-07
0.000203087 3.97191e-07
0.000170126 2.789e-07
0.000142511 1.95717e-07
0.000119374 1.37345e-07
0.0001 9.63702e-08
};
\addlegendentry{$h=0.28284$}

\addplot [semithick, color03, mark=asterisk, mark size=3, mark options={solid}]
table {%
0.1 1000
0.0833333 1000
0.0714286 1000
0.0588235 1000
0.05 1000
0.0416667 1000
0.0344828 1000
0.0285714 1000
0.0243902 1000
0.0204082 1000
0.0169492 1000
0.0142857 1000
0.0119048 0.00135147
0.01 0.000953505
0.00840336 0.000673289
0.00699301 0.000466233
0.00588235 0.000329886
0.00492611 0.000233019
0.00413223 0.000162785
0.00346021 0.000114142
0.00289017 7.96317e-05
0.00242718 5.67643e-05
0.00203252 3.93826e-05
0.00170068 2.75727e-05
0.0014245 1.93769e-05
0.00119332 1.37147e-05
0.001 9.53304e-06
0.000837521 6.68687e-06
0.000701754 4.73306e-06
0.000587889 3.3146e-06
0.000492368 2.32819e-06
0.000412541 1.63696e-06
0.000345543 1.15079e-06
0.000289436 8.06883e-07
0.000242424 5.66332e-07
0.000203087 3.97213e-07
0.000170126 2.78907e-07
0.000142511 1.95723e-07
0.000119374 1.37349e-07
0.0001 9.63748e-08
};
\addlegendentry{$h=0.15713$}
\addplot [semithick, gray, dashed, forget plot]
table {%
0.1 0.0481874
0.0833333 0.0334634454514498
0.0714286 0.0245854278315958
0.0588235 0.0166738241566132
0.05 0.01204685
0.0416667 0.0083658814409498
0.0344828 0.00572978782994404
0.0285714 0.00393365743879577
0.0243902 0.00286658099497419
0.0204082 0.00200697932006648
0.0169492 0.00138430536770519
0.0142857 0.000983414359698943
0.0119048 0.00068293237528137
0.01 0.000481874
0.00840336 0.000340282357037167
0.00699301 0.000235646933547718
0.00588235 0.000166738241566132
0.00492611 0.00011693424204346
0.00413223 8.22815504961641e-05
0.00346021 5.76950305894744e-05
0.00289017 4.02513333871856e-05
0.00242718 2.83881743511e-05
0.00203252 1.99068777596145e-05
0.00170068 1.39373017550654e-05
0.0014245 9.778188412685e-06
0.00119332 6.86194658406378e-06
0.001 4.81874e-06
0.000837521 3.38006385442956e-06
0.000701754 2.37303032287471e-06
0.000587889 1.66542148288706e-06
0.000492368 1.16818905551193e-06
0.000412541 8.20101730105802e-07
0.000345543 5.7535738661647e-07
0.000289436 4.03681260593119e-07
0.000242424 2.83194438201642e-07
0.000203087 1.98745700667323e-07
0.000170126 1.39468097323916e-07
0.000142511 9.78656464579676e-08
0.000119374 6.86677768509562e-08
0.0001 4.81874e-08
};
\addplot [semithick, color04, mark=asterisk, mark size=3, mark options={solid}]
table {%
0.1 1000
0.0833333 1000
0.0714286 1000
0.0588235 1000
0.05 1000
0.0416667 1000
0.0344828 1000
0.0285714 1000
0.0243902 1000
0.0204082 1000
0.0169492 1000
0.0142857 1000
0.0119048 1000
0.01 1000
0.00840336 1000
0.00699301 1000
0.00588235 0.000329887
0.00492611 0.000233021
0.00413223 0.000162786
0.00346021 0.000114142
0.00289017 7.96318e-05
0.00242718 5.67643e-05
0.00203252 3.93826e-05
0.00170068 2.75727e-05
0.0014245 1.93769e-05
0.00119332 1.37145e-05
0.001 9.53305e-06
0.000837521 6.68687e-06
0.000701754 4.73309e-06
0.000587889 3.31456e-06
0.000492368 2.3282e-06
0.000412541 1.63694e-06
0.000345543 1.15078e-06
0.000289436 8.06875e-07
0.000242424 5.66329e-07
0.000203087 3.97209e-07
0.000170126 2.78906e-07
0.000142511 1.95722e-07
0.000119374 1.37348e-07
0.0001 9.63739e-08
};
\addlegendentry{$h=0.08839$}

\addplot [semithick, color05, mark=asterisk, mark size=3, mark options={solid}]
table {%
0.1 1000
0.0833333 1000
0.0714286 1000
0.0588235 1000
0.05 1000
0.0416667 1000
0.0344828 1000
0.0285714 1000
0.0243902 1000
0.0204082 1000
0.0169492 1000
0.0142857 1000
0.0119048 1000
0.01 1000
0.00840336 1000
0.00699301 1000
0.00588235 1000
0.00492611 1000
0.00413223 1000
0.00346021 0.000114142
0.00289017 7.96318e-05
0.00242718 5.67643e-05
0.00203252 3.93826e-05
0.00170068 2.75727e-05
0.0014245 1.93769e-05
0.00119332 1.37145e-05
0.001 9.53305e-06
0.000837521 6.68687e-06
0.000701754 4.73308e-06
0.000587889 3.31455e-06
0.000492368 2.3282e-06
0.000412541 1.63693e-06
0.000345543 1.15078e-06
0.000289436 8.06874e-07
0.000242424 5.6633e-07
0.000203087 3.97209e-07
0.000170126 2.78907e-07
0.000142511 1.95722e-07
0.000119374 1.37349e-07
0.0001 9.6374e-08
};
\addlegendentry{$h=0.05051$}
\addplot [semithick, gray, dashed, forget plot]
table {%
0.1 0.048187
0.0833333 0.0334631676738942
0.0714286 0.0245852237497998
0.0588235 0.0166736857484471
0.05 0.01204675
0.0416667 0.00836581199639424
0.0344828 0.00572974026740421
0.0285714 0.00393362478579985
0.0243902 0.00286655719969995
0.0204082 0.00200696266028139
0.0169492 0.00138429387668997
0.0142857 0.000983406196449963
0.0119048 0.000682926706310848
0.01 0.00048187
0.00840336 0.000340279532378796
0.00699301 0.000235644977460164
0.00588235 0.000166736857484471
0.00492611 0.00011693327138107
0.00413223 8.22808674831732e-05
0.00346021 5.76945516673447e-05
0.00289017 4.02509992638804e-05
0.00242718 2.83879387029899e-05
0.00203252 1.99067125141125e-05
0.00170068 1.39371860625669e-05
0.0014245 9.778107244675e-06
0.00119332 6.86188962355888e-06
0.001 4.8187e-06
0.000837521 3.38003579677255e-06
0.000701754 2.37301062452765e-06
0.000587889 1.665407658348e-06
0.000492368 1.16817935846203e-06
0.000412541 8.20094922502735e-07
0.000345543 5.75352610617876e-07
0.000289436 4.03677909665195e-07
0.000242424 2.83192087425811e-07
0.000203087 1.9874405089414e-07
0.000170126 1.39466939609681e-07
0.000142511 9.78648340825627e-08
0.000119374 6.86672068448812e-08
0.0001 4.8187e-08
};
\end{axis}

\end{tikzpicture}
  }
  \caption{Error in the $L^2$-norm of the polynomial solution \cref{eq:pol_sol} plotted against the time step size $\tau$ for 5 different \meshsize s between $5\cdot 10^{-1}$ and $5\cdot 10^{-2}$ using elements of order $k=3$.}
  \label{fig:PolynomialErrorPlot}
\end{figure}

\section*{Acknowledgement}

We thank Kurt Busch for the fruitful discussions on plasmonic nanogaps and the insights into the exciting physical phenomena, and Constantin Carle for many helpful discussions on the implementation.

\bibliographystyle{amsalpha}
\bibliography{literature}

\begin{bibdiv}
\begin{biblist}

\bib{dealII95}{article}{
      author={Arndt, D.},
      author={Bangerth, W.},
      author={Bergbauer, M.},
      author={Feder, M.},
      author={Fehling, M.},
      author={Heinz, J.},
      author={Heister, T.},
      author={Heltai, L.},
      author={Kronbichler, M.},
      author={Maier, M.},
      author={Munch, P.},
      author={Pelteret, J.-P.},
      author={Turcksin, B.},
      author={d.~Wells},
      author={Zampini, S.},
       title={The \texttt{deal.II} library, version 9.5},
        date={2023},
     journal={Journal of Numerical Mathematics},
      volume={31},
      number={3},
       pages={231\ndash 246},
         url={https://dealii.org/deal95-preprint.pdf},
      review={\MR{4637502}},
}

\bib{Buf.C.2001}{article}{
      author={Buffa, A.},
      author={Ciarlet, P., Jr.},
       title={On traces for functional spaces related to {M}axwell's equations.
  {I}. {A}n integration by parts formula in {L}ipschitz polyhedra},
        date={2001},
        ISSN={0170-4214,1099-1476},
     journal={Math. Methods Appl. Sci.},
      volume={24},
      number={1},
       pages={9\ndash 30},
  url={https://doi.org/10.1002/1099-1476(20010110)24:1<9::AID-MMA191>3.0.CO;2-2},
      review={\MR{1809491}},
}

\bib{Buf.C.2001b}{article}{
      author={Buffa, A.},
      author={Ciarlet, P., Jr.},
       title={On traces for functional spaces related to {M}axwell's equations.
  {II}. {H}odge decompositions on the boundary of {L}ipschitz polyhedra and
  applications},
        date={2001},
        ISSN={0170-4214,1099-1476},
     journal={Math. Methods Appl. Sci.},
      volume={24},
      number={1},
       pages={31\ndash 48},
  url={https://doi.org/10.1002/1099-1476(20010110)24:1<9::AID-MMA191>3.0.CO;2-2},
      review={\MR{1809492}},
}

\bib{Cha.E.Z.Et.2008}{incollection}{
      author={Charlier, J.-C.},
      author={Eklund, P.~C.},
      author={Zhu, J.},
      author={Ferrari, A.~C.},
       title={Electron and {{Phonon Properties}} of {{Graphene}}: {{Their
  Relationship}} with {{Carbon Nanotubes}}},
        date={2008},
   booktitle={Carbon {{Nanotubes}}: {{Advanced Topics}} in the {{Synthesis}},
  {{Structure}}, {{Properties}} and {{Applications}}},
      editor={Jorio, Ado},
      editor={Dresselhaus, Gene},
      editor={Dresselhaus, Mildred~S.},
      series={Topics in {{Applied Physics}}},
   publisher={{Springer}},
       pages={673\ndash 709},
         url={https://doi.org/10.1007/978-3-540-72865-8_21},
        note={\url{https://doi.org/10.1007/978-3-540-72865-8_21}},
}

\bib{Che.Z.1998}{article}{
      author={Chen, Z.},
      author={Zou, J.},
       title={Finite element methods and their convergence for elliptic and
  parabolic interface problems},
        date={1998},
        ISSN={0029-599X,0945-3245},
     journal={Numer. Math.},
      volume={79},
      number={2},
       pages={175\ndash 202},
         url={https://doi.org/10.1007/s002110050336},
      review={\MR{1622502}},
}

\bib{Dau.L.1990.MANM}{book}{
      author={Dautray, R.},
      author={Lions, J.-L.},
       title={Mathematical analysis and numerical methods for science and
  technology. {V}ol. 1},
   publisher={Springer-Verlag, Berlin},
        date={1990},
        ISBN={3-540-50207-6; 3-540-66097-6},
        note={Physical origins and classical methods, With the collaboration of
  Philippe B\'enilan, Michel Cessenat, Andr\'e{} Gervat, Alain Kavenoky and
  H\'el\`ene Lanchon, Translated from the French by Ian N. Sneddon, With a
  preface by Jean Teillac},
      review={\MR{1036731}},
}

\bib{Dek.2017}{article}{
      author={Deka, B.},
       title={{\it {A} priori} {$L^\infty(L^2)$} error estimates for finite
  element approximations to the wave equation with interface},
        date={2017},
        ISSN={0168-9274,1873-5460},
     journal={Appl. Numer. Math.},
      volume={115},
       pages={142\ndash 159},
         url={https://doi.org/10.1016/j.apnum.2017.01.004},
      review={\MR{3611360}},
}

\bib{Dek.S.2012}{article}{
      author={Deka, B.},
      author={Sinha, R.~K.},
       title={Finite element methods for second order linear hyperbolic
  interface problems},
        date={2012},
        ISSN={0096-3003,1873-5649},
     journal={Appl. Math. Comput.},
      volume={218},
      number={22},
       pages={10922\ndash 10933},
         url={https://doi.org/10.1016/j.amc.2012.04.055},
      review={\MR{2942377}},
}

\bib{Dep.2016.GOES}{book}{
      author={Depine, R.~A.},
       title={Graphene {{Optics}}: {{Electromagnetic Solution}} of {{Canonical
  Problems}}},
      series={2053-2571},
   publisher={{Morgan \& Claypool Publishers}},
        date={2016},
        ISBN={978-1-68174-309-7},
         url={https://dx.doi.org/10.1088/978-1-6817-4309-7},
        note={\url{https://dx.doi.org/10.1088/978-1-6817-4309-7}},
}

\bib{Di.E.2012.MADG}{book}{
      author={Di~Pietro, D.~A.},
      author={Ern, A.},
       title={Mathematical aspects of discontinuous {G}alerkin methods},
      series={Math\'ematiques \& Applications (Berlin) [Mathematics \&
  Applications]},
   publisher={Springer, Heidelberg},
        date={2012},
      volume={69},
        ISBN={978-3-642-22979-4},
         url={https://doi.org/10.1007/978-3-642-22980-0},
      review={\MR{2882148}},
}

\bib{Dor.H.K.Et.2023.WPMA}{book}{
      author={D\"orfler, W.},
      author={Hochbruck, M.},
      author={K\"ohler, J.},
      author={Rieder, A.},
      author={Schnaubelt, R.},
      author={Wieners, C.},
       title={Wave phenomena---mathematical analysis and numerical
  approximation},
      series={Oberwolfach Seminars},
   publisher={Birkh\"auser/Springer, Cham},
        date={2023},
      volume={49},
        ISBN={978-3-031-05792-2; 978-3-031-05793-9},
         url={https://doi.org/10.1007/978-3-031-05793-9},
      review={\MR{4592531}},
}

\bib{Dor.Z.2023}{article}{
      author={D\"orich, B.},
      author={Zerulla, K.},
       title={Wellposedness and regularity for linear {M}axwell equations with
  surface current},
        date={2023},
        ISSN={0044-2275,1420-9039},
     journal={Z. Angew. Math. Phys.},
      volume={74},
      number={4},
       pages={Paper No. 131, 24},
         url={https://doi.org/10.1007/s00033-023-02021-w},
      review={\MR{4599675}},
}

\bib{Edm.E.2023.FSSI}{book}{
      author={Edmunds, D.~E.},
      author={Evans, W.~D.},
       title={Fractional {S}obolev spaces and inequalities},
      series={Cambridge Tracts in Mathematics},
   publisher={Cambridge University Press, Cambridge},
        date={2023},
      volume={230},
        ISBN={978-1-009-25463-2},
      review={\MR{4480576}},
}

\bib{Ern.G.2021.FEAI}{book}{
      author={Ern, A.},
      author={Guermond, J.-L},
       title={Finite elements {I}---{A}pproximation and interpolation},
      series={Texts in Applied Mathematics},
   publisher={Springer, Cham},
        date={[2021] \copyright2021},
      volume={72},
        ISBN={978-3-030-56340-0; 978-3-030-56341-7},
         url={https://doi.org/10.1007/978-3-030-56341-7},
      review={\MR{4242224}},
}

\bib{Gir.R.1986.FEMNa}{book}{
      author={Girault, V.},
      author={Raviart, P.-A.},
       title={Finite element methods for {N}avier-{S}tokes equations},
      series={Springer Series in Computational Mathematics},
   publisher={Springer-Verlag, Berlin},
        date={1986},
      volume={5},
        ISBN={3-540-15796-4},
         url={https://doi.org/10.1007/978-3-642-61623-5},
        note={Theory and algorithms},
      review={\MR{851383}},
}

\bib{Hes.W.2002}{article}{
      author={Hesthaven, J.~S.},
      author={Warburton, T.},
       title={Nodal high-order methods on unstructured grids. {I}.
  {T}ime-domain solution of {M}axwell's equations},
        date={2002},
        ISSN={0021-9991,1090-2716},
     journal={J. Comput. Phys.},
      volume={181},
      number={1},
       pages={186\ndash 221},
         url={https://doi.org/10.1006/jcph.2002.7118},
      review={\MR{1925981}},
}

\bib{Hoc.K.2022}{article}{
      author={Hochbruck, M.},
      author={K\"ohler, J.},
       title={Error analysis of a fully discrete discontinuous {G}alerkin
  alternating direction implicit discretization of a class of linear wave-type
  problems},
        date={2022},
        ISSN={0029-599X,0945-3245},
     journal={Numer. Math.},
      volume={150},
      number={3},
       pages={893\ndash 927},
         url={https://doi.org/10.1007/s00211-021-01262-z},
      review={\MR{4394004}},
}

\bib{Hoc.L.O.2020}{article}{
      author={Hochbruck, M.},
      author={Leibold, J.},
      author={Ostermann, A.},
       title={On the convergence of {L}awson methods for semilinear stiff
  problems},
        date={2020},
        ISSN={0029-599X,0945-3245},
     journal={Numer. Math.},
      volume={145},
      number={3},
       pages={553\ndash 580},
         url={https://doi.org/10.1007/s00211-020-01120-4},
      review={\MR{4119515}},
}

\bib{Hoc.S.2016}{article}{
      author={Hochbruck, M.},
      author={Sturm, A.},
       title={Error analysis of a second-order locally implicit method for
  linear {M}axwell's equations},
        date={2016},
        ISSN={0036-1429,1095-7170},
     journal={SIAM J. Numer. Anal.},
      volume={54},
      number={5},
       pages={3167\ndash 3191},
         url={https://doi.org/10.1137/15M1038037},
      review={\MR{3565552}},
}

\bib{Kan.M.2005}{article}{
      author={Kane, C.~L.},
      author={Mele, E.~J.},
       title={Quantum spin hall effect in graphene},
        date={2005Nov},
     journal={Phys. Rev. Lett.},
      volume={95},
       pages={226801},
         url={https://link.aps.org/doi/10.1103/PhysRevLett.95.226801},
        note={\url{https://doi.org/10.1103/PhysRevLett.95.226801}},
}

\bib{Gan.L.1996}{article}{
      author={Leng, G.},
      author={Tang, L.},
       title={Some inequalities on the inradii of a simplex and of its faces},
        date={1996},
        ISSN={0046-5755,1572-9168},
     journal={Geom. Dedicata},
      volume={61},
      number={1},
       pages={43\ndash 49},
         url={https://doi.org/10.1007/BF00149418},
      review={\MR{1389636}},
}

\bib{Paz.1983.SLOA}{book}{
      author={Pazy, A.},
       title={Semigroups of linear operators and applications to partial
  differential equations},
      series={Applied Mathematical Sciences},
   publisher={Springer-Verlag, New York},
        date={1983},
      volume={44},
        ISBN={0-387-90845-5},
         url={https://doi.org/10.1007/978-1-4612-5561-1},
      review={\MR{710486}},
}

\bib{Sch.S.2022}{article}{
      author={Schnaubelt, R.},
      author={Spitz, M.},
       title={Local wellposedness of quasilinear {M}axwell equations with
  conservative interface conditions},
        date={2022},
        ISSN={1539-6746,1945-0796},
     journal={Commun. Math. Sci.},
      volume={20},
      number={8},
       pages={2265\ndash 2313},
      review={\MR{4521041}},
}

\bib{Sta.2014}{article}{
      author={Stauber, T.},
       title={Plasmonics in {{Dirac}} systems: From graphene to topological
  insulators},
        date={2014},
        ISSN={0953-8984},
     journal={Journal of Physics: Condensed Matter},
      volume={26},
      number={12},
       pages={123201},
         url={https://dx.doi.org/10.1088/0953-8984/26/12/123201},
        note={\url{https://dx.doi.org/10.1088/0953-8984/26/12/123201}},
}

\bib{Ted.L.O.2016}{article}{
      author={Tedstone, A.~A.},
      author={Lewis, D.~J.},
      author={O’Brien, P.},
       title={Synthesis, properties, and applications of transition metal-doped
  layered transition metal dichalcogenides},
        date={2016},
     journal={Chemistry of Materials},
      volume={28},
      number={7},
       pages={1965\ndash 1974},
         url={https://doi.org/10.1021/acs.chemmater.6b00430},
        note={\url{https://doi.org/10.1021/acs.chemmater.6b00430}},
}

\bib{Cho.L.S.Et.2010}{article}{
      author={W.~Choi, R.~Seelaboyina, I.~Lahiri},
      author={Kang, Y.~S.},
       title={Synthesis of graphene and its applications: A review},
        date={2010},
     journal={Critical Reviews in Solid State and Materials Sciences},
      volume={35},
      number={1},
       pages={52\ndash 71},
         url={https://doi.org/10.1080/10408430903505036},
        note={\url{https://doi.org/10.1080/10408430903505036}},
}

\bib{Wer.K.B.Et.2016}{article}{
      author={Werra, J. F.~M.},
      author={Kr{\"u}ger, Peter},
      author={Busch, K.},
      author={Intravaia, Francesco},
       title={Determining graphene's induced band gap with magnetic and
  electric emitters},
        date={2016},
     journal={Phys. Rev. B},
      volume={93},
      number={8},
       pages={081404},
         url={https://link.aps.org/doi/10.1103/PhysRevB.93.081404},
        note={\url{https://doi.org/10.1103/PhysRevB.93.081404}},
}

\bib{Wer.W.M.Et.2015}{inproceedings}{
      author={Werra, J. F.~M.},
      author={Wolff, C.},
      author={Matyssek, C.},
      author={Busch, K.},
       title={Current sheets in the {{Discontinuous Galerkin Time-Domain}}
  method: An application to graphene},
        date={2015},
   booktitle={Metamaterials {{X}}},
      editor={Kuzmiak, Vladim{\'i}r},
      editor={Markos, Peter},
      editor={Szoplik, Tomasz},
      volume={9502},
   publisher={{SPIE}},
       pages={95020E},
         url={https://doi.org/10.1117/12.2182571},
        note={\url{https://doi.org/10.1117/12.2182571}},
}

\end{biblist}
\end{bibdiv}

\end{document}